\documentclass{amsart}

\usepackage{amsmath,amsfonts,amssymb}
\usepackage{color}
\usepackage{tikz}
\usepackage{ytableau}
\usepackage{float}
\usepackage{fullpage}
\usetikzlibrary{arrows.meta,chains,matrix,decorations.pathreplacing}
\usetikzlibrary{patterns}
\usepackage{mathabx}

\newtheorem*{theorem*}{Theorem}

\newtheorem{theorem}{Theorem}[section]
\newtheorem{lemma}[theorem]{Lemma}
\newtheorem{proposition}[theorem]{Proposition}
\newtheorem{corollary}[theorem]{Corollary}

\theoremstyle{definition}
\newtheorem{definition}[theorem]{Definition}
\newtheorem{example}[theorem]{Example}

\theoremstyle{remark}
\newtheorem{remark}[theorem]{Remark}

\theoremstyle{observation}

\numberwithin{equation}{section}

\newcommand{\gl}{\ensuremath\mathfrak{gl}}

\newcommand{\coinv}{\ensuremath\mathrm{coinv}}
\newcommand{\maj}{\ensuremath\mathrm{maj}}

\newcommand{\wt}{\ensuremath\mathrm{wt}}

\newcommand{\SSYT}{\ensuremath\mathrm{SSYT}}
\newcommand{\SSKD}{\ensuremath\mathrm{SSKD}}
\newcommand{\SSKT}{\ensuremath\mathrm{SSKT}}

\newcommand{\KD}{\ensuremath\mathrm{KD}}
\newcommand{\D}{\ensuremath\mathbb{D}}
\newcommand{\T}{\ensuremath\mathbb{T}}

\newcommand{\key}{\ensuremath\kappa}
\newcommand{\rectify}{\ensuremath\mathrm{rect}}
\newcommand{\embed}{\Aries}

\newcommand{\De}{\tilde{e}} 
\newcommand{\Df}{\tilde{f}}
\newcommand{\Ye}{\hat{e}} 
\newcommand{\Yf}{\hat{f}}
\newcommand{\E}{\reflectbox{e}} 

\newcommand{\B}{\ensuremath\mathcal{B}}

\newcommand{\newword}[1]{\emph{\textbf{#1}}}

\newlength\cellsize 

\savebox2{%
\begin{picture}(10,10)
\put(0,0){\line(1,0){10}}
\put(0,0){\line(0,1){10}}
\put(10,0){\line(0,1){10}}
\put(0,10){\line(1,0){10}}
\end{picture}}

\newcommand\cellify[1]{\def\thearg{#1}\def\nothing{}%
\ifx\thearg\nothing\vrule width0pt height\cellsize depth0pt%
  \else\hbox to 0pt{\usebox2\hss}\fi%
  \vbox to 10\unitlength{\vss\hbox to 10\unitlength{\hss$#1$\hss}\vss}}

\newcommand\tableau[1]{\vtop{\let\\=\cr
\setlength\baselineskip{-10000pt}
\setlength\lineskiplimit{10000pt}
\setlength\lineskip{0pt}
\halign{&\cellify{##}\cr#1\crcr}}}

\savebox7{
\begin{picture}(10,10)
  \put(5,5){\circle{9.4}}
\end{picture}}

\newcommand{\cirfy}[1]{\def\thearg{#1}\def\nothing{}%
\ifx\thearg\nothing\vrule width0pt height\cellsize depth0pt%
  \else\hbox to 0pt{\usebox7\hss}\fi%
  \vbox to 10\unitlength{\vss\hbox to 10\unitlength{\hss$#1$\hss}\vss}}

\newcommand\cirtab[1]{\vtop{\let\\=\cr
\setlength\baselineskip{-10000pt}
\setlength\lineskiplimit{10000pt}
\setlength\lineskip{0pt}
\halign{&\cirfy{##}\cr#1\crcr}}}

\newcommand{\bball}{%
  \begin{tikzpicture}
    \filldraw[fill=blue!30,draw=black] circle (4.7pt);
  \end{tikzpicture}
}
\newcommand{\rball}{%
  \begin{tikzpicture}
    \filldraw[fill=red!30,draw=black] circle (4.7pt);
  \end{tikzpicture}
}
\newcommand{\vball}{%
  \begin{tikzpicture}
    \filldraw[fill=violet!30,draw=black] circle (4.7pt);
  \end{tikzpicture}
}
\newcommand{\gball}{%
  \begin{tikzpicture}
    \filldraw[fill=green!30,draw=black] circle (4.7pt);
  \end{tikzpicture}
}

\newcommand{\yball}{%
  \begin{tikzpicture}
    \filldraw[fill=orange!30,draw=black] circle (4.7pt);
  \end{tikzpicture}
}

\newcommand{\lball}{%
  \begin{tikzpicture}
   \filldraw[fill=gray!20, draw=gray!10] circle (4.9pt);
  \end{tikzpicture}
}

\newcommand{\csquare}[1]{ 
\begin{scope}[shift={#1}]
\draw (-.5,-.5) rectangle (.5,.5);
\end{scope}
 }

\makeatletter
\tikzset{
        hatch distance/.store in=\hatchdistance,
        hatch distance=5pt,
        hatch thickness/.store in=\hatchthickness,
        hatch thickness=5pt
        }
\pgfdeclarepatternformonly[\hatchdistance,\hatchthickness]{north east hatch}
    {\pgfqpoint{-1pt}{-1pt}}
    {\pgfqpoint{\hatchdistance}{\hatchdistance}}
    {\pgfpoint{\hatchdistance-1pt}{\hatchdistance-1pt}}%
    {
        \pgfsetcolor{\tikz@pattern@color}
        \pgfsetlinewidth{\hatchthickness}
        \pgfpathmoveto{\pgfqpoint{0pt}{0pt}}
        \pgfpathlineto{\pgfqpoint{\hatchdistance}{\hatchdistance}}
        \pgfusepath{stroke}
    }
\makeatother

\newcommand{\sball}{%
  \begin{tikzpicture}
  [Pattern/.style={pattern=north east hatch, pattern color=red!50, hatch distance=7pt, hatch thickness=2pt}]

    \filldraw[preaction={fill=blue!40}, Pattern,draw=black] circle (4.7pt);
  \end{tikzpicture}
}

\usepackage{xcolor}
\usepackage[colorinlistoftodos]{todonotes}

\input xy
\usepackage[all]{xy}
\xyoption{line}
\xyoption{arrow}
\xyoption{color}
\SelectTips{cm}{}

\usepackage{tikz}
\usetikzlibrary{decorations.markings}
\usetikzlibrary{decorations.pathreplacing}
\newcommand{\hackcenter}[1]{
 \xy (0,0)*{#1}; \endxy}

\tikzstyle directed=[postaction={decorate,decoration={markings,
    mark=at position #1 with {\arrow{>}}}}]
\tikzstyle rdirected=[postaction={decorate,decoration={markings,
    mark=at position #1 with {\arrow{<}}}}]

\usetikzlibrary{calc}
\usepackage{relsize}

\tikzset{fontscale/.style = {font=\relsize{#1}}
    }

\begin{document}


\title[Demazure crystals for Macdonald polynomials]{Demazure crystals for specialized nonsymmetric Macdonald polynomials}

\author[Assaf]{Sami Assaf}
\address{Department of Mathematics, University of Southern California, 3620 S. Vermont Ave., Los Angeles, CA 90089-2532, U.S.A.}
\email{shassaf@usc.edu}
\thanks{S.A. supported in part by NSF DMS-1763336. N.G. was supported in part by NSF grants DMS-1255334 and DMS-1664240.}

\author[Gonz\'{a}lez]{Nicolle E. S. Gonz\'{a}lez}
\address{Department of Mathematics, University of Southern California, 3620 S. Vermont Ave., Los Angeles, CA 90089-2532, U.S.A.}
\email{nesandov@usc.edu}

\subjclass[2010]{Primary ; Secondary }




\begin{abstract}
  We give an explicit, nonnegative formula for the expansion of nonsymmetric Macdonald polynomials specialized at $t=0$ in terms of Demazure characters. Our formula results from constructing Demazure crystals whose characters are the nonsymmetric Macdonald polynomials, which also gives a new proof that these specialized nonsymmetric Macdonald polynomials are positive graded sums of Demazure characters. Demazure crystals are certain truncations of classical crystals that give a combinatorial skeleton for Demazure modules. To prove our construction, we develop further properties of Demazure crystals, including an efficient algorithm for computing their characters from highest weight elements. As a corollary, we obtain a new formula for the Schur expansion of Hall--Littlewood polynomials in terms of a simple statistic on highest weight elements of our crystals.
\end{abstract}

\maketitle
\tableofcontents

%
\section{Introduction}
%
\label{sec:introduction}

Macdonald \cite{Mac88} defined symmetric functions $P_{\mu}(X;q,t)$ with two parameters $q,t$ indexed by partitions as the unique symmetric function basis satisfying certain triangularity (with respect to monomials in infinitely many variables $X=x_1,x_2,\ldots$) and orthogonality (with respect to a generalized Hall inner product) conditions. The \newword{Macdonald symmetric functions} give a simultaneous generalization of Hall--Littlewood symmetric functions $P_{\lambda}(X;0,t)$ and Jack symmetric functions $\lim_{t\rightarrow 1}P_{\lambda}(X;t^{\alpha},t)$. 

The coefficients of $P_{\mu}(X;q,t)$ when written as a sum of monomials are rational functions in the parameters $q$ and $t$. Macdonald conjectured that the \newword{Kostka--Macdonald coefficients} $K_{\lambda,\mu}(q,t)$ defined by expanding the \emph{integral form} $J_{\mu}(X;q,t)$, a scalar multiple of the original $P_{\mu}(X;q,t)$, into the \newword{plethystic Schur basis}, 
\[ J_{\mu}(X;q,t) = \sum_{\lambda} K_{\lambda,\mu}(q,t) s_{\lambda}[X(1-t)], \]
are polynomials in $q$ and $t$ with nonnegative integer coefficients. Here the square brackets denote \newword{plethystic substitution}. In short, $s_{\lambda}[A]$ means $s_{\lambda}$ applied as a $\Lambda$-ring operator to the expression $A$, where $\Lambda$ is the ring of symmetric functions. For details, see \cite{Mac95}(I.8).

Inspired by work of Garsia and Procesi \cite{GP92} on Hall--Littlewood symmetric functions, Garsia and Haiman~\cite{GH93} constructed a bi-graded module for the symmetric group and conjectured that the Frobenius character is
\[ H_{\mu}(X;q,t) = J_{\mu}[X/(1-t); q, t]. \]
Thus, the Kostka--Macdonald coefficients give the Schur function expansion of $H_{\mu}(X;q,t)$. This conjecture gives a representation theoretic interpretation for the Kostka-Macdonald polynomials as the graded coefficients of the irreducible decomposition of these modules. Haiman \cite{Hai01} resolved both conjectures by analyzing the isospectral Hilbert scheme of points in a plane, ultimately showing that it is Cohen-Macaulay.
\smallskip

The \newword{nonsymmetric Macdonald polynomials} $E_{a}(X;q,t)$ are indexed by weak compositions and form a basis for the full polynomial ring. They generalize Macdonald polynomials in the sense that
\begin{eqnarray*}
  E_{0^m \times a}(x_1,\ldots,x_m,0,\ldots,0;q,t) & = & P_{\mathrm{sort}(a)}(x_1,\ldots,x_m;q,t),
\end{eqnarray*}
where $0^m \times a$ denotes the composition obtained by prepending $m$ $0$'s to $a$. The shift to the full polynomial ring begun by Opdam \cite{Opd95}, continued by Macdonald \cite{Mac96}, and generalized by Cherednik \cite{Che95} broadened the existing theory in the hopes that the additional structure of the polynomial ring would shed more light on these important functions. Work of Knop and Sahi \cite{KS97} on Jack polynomials helped to validate this approach, and their independently derived recurrences \cite{Kno97,Sah96} ultimately inspired the combinatorial formula for nonsymmetric Macdonald polynomials of Haglund, Haiman, and Loehr \cite{HHL08}.

Generalizing Haglund's elegant combinatorial formula for $H_{\mu}(X;q,t)$ \cite{Hag04,HHL05}, Haglund, Haiman and Loehr~\cite{HHL08} gave a combinatorial formula for $E_{a}(X;q,t)$ as
\[  E_{a}(X;q,t) = \sum_{\substack{T:a\rightarrow [n] \\ \mathrm{non-attacking}}} q^{\maj(T)} t^{\coinv(T)} X^{\wt(T)}
  \prod_{c \neq \mathrm{left}(c)} \frac{1-t}{1 - q^{\mathrm{leg}(c)+1} t^{\mathrm{arm}(c)+1}}, \]
where the sum is over certain positive integer fillings $T$ of the diagram of the composition $a$ and $\coinv$ and $\maj$ are nonnegative integer statistics. In stark contrast with the symmetric case, there are no known (nor even conjectured) positivity results for the nonsymmetric Macdonald polynomials.
\smallskip

Demazure \cite{Dem74} generalized the Weyl character formula to certain submodules, eponymously named \newword{Demazure modules}, which are generated by extremal weight spaces under the action of a Borel subalgebra of a Lie algebra. The resulting \newword{Demazure characters} $\key_a$, where $a = w \cdot \lambda$, for $w$ a Weyl group element acting on a highest weight $\lambda$, arose in connection with Schubert calculus \cite{Dem74a}, and, in type A, also form a basis of the polynomial ring. Recent work of Assaf and Searles \cite{AS18} indicates that the type A Demazure characters are the most natural pull backs of Schur functions to the polynomial ring. That is to say, the combinatorics of the former stabilizes to that of the latter,
\begin{eqnarray*}
  \key_{0^m \times a}(x_1,\ldots,x_m,0,\ldots,0) & = & s_{\mathrm{sort}(a)}(x_1,\ldots,x_m).
\end{eqnarray*}
Therefore, in the search for polynomial analogs of Schur positivity statements for nonsymmetric Macdonald polynomials, the natural basis for comparison is the basis of Demazure characters.

Sanderson \cite{San00} first made the connection between specializations of Macdonald polynomials and Demazure characters by using the theory of nonsymmetric Macdonald polynomials in type A to construct an \emph{affine} Demazure module with graded character $P_{\mu}(X;q,0)$, parallel to the construction of Garsia and Procesi \cite{GP92} for Hall-Littlewood symmetric functions $H_{\mu}(X;0,t)$. Ion \cite{Ion03} generalized this result to nonsymmetric Macdonald polynomials in general type using the method of intertwiners in double affine Hecke algebras to realize $E_{a}(X;q,0)$ as a single \emph{affine} Demazure character. Inspired by this, Lenart, Naito, Sagaki, Schilling and Shimozono \cite{LNSS17} constructed a connected Kirillov--Reshetikhin crystal to give a combinatorial proof of the coincidence with affine Demazure characters using similar methods. 

Recently, Assaf \cite{Ass18} proved the specialization $E_{a}(X;q,0)$ is a nonnegative, graded sum of \emph{finite} Demazure characters. The proof utilizes the machinery of weak dual equivalence \cite{Ass-W}. Hence, the resulting formula is difficult to work with and, in practice, requires computing the full fundamental slide polynomial \cite{AS17} expansion of $E_{a}(X;q,0)$. In order to have a better understanding of this nonnegativity and to have a deeper connection with the underlying representation theory of Demazure modules, we use crystal theory to give a new proof of this graded nonnegativity for \emph{finite} Demazure characters from which we extract an explicit formula for the Demazure expansion. The immediate benefit of our new approach is two-fold. On the one hand, our method yields a formula which is very easily computable. On the other,  weak dual equivalence exists only for the general linear group whereas the crystal theory used in our new approach extends to all types, thus our results give hope that these new methods might be a key to a result in general type. 

Kashiwara \cite{Kas91} introduced the notion of \newword{crystal bases} in his study of the representation theory of quantized universal enveloping algebras $U_q(\mathfrak{g})$ for complex, semi-simple Lie algebra $\mathfrak{g}$ at $q=0$. The theory of \newword{canonical bases}, developed earlier by Lusztig \cite{Lus90}, studies the same problem from a more geometric viewpoint, though many of the main ideas from \cite{Lus90} carry over to \cite{Kas91}. A \newword{crystal base} is a basis of a representation for $U_q(\mathfrak{g})$ on which the Chevalley generators have a relatively simple action. Combinatorially, a \newword{crystal} is a directed, colored graph with vertex set given by the crystal base and directed edges given by deformations of the Chevalley generators. By constructing a $\mathfrak{gl}_n$ crystal for a set of combinatorial objects, we create a combinatorial skeleton of the corresponding $\mathfrak{gl}_n$ modules whose character is the generating function of those objects. In particular, the generating function is Schur positive. Moreover, crystal theory provides unique highest weight elements, from which tractable formulas can be derived. Stembridge \cite{Ste03} gave a local characterization of simply-laced crystals that allows one to prove that a given construction is indeed a crystal by analyzing local properties of the raising and lowering operators which give rise to the edges of the graph. 

\newword{Demazure crystals}, whose structure was conjectured by Littlemann \cite{Lit95} and proved by Kashiwara \cite{Kas93}, are certain truncations of classical crystals that give a combinatorial skeleton for Demazure modules. Unlike full crystals, Demazure crystals are not uniquely characterized by their highest weight elements. Further complicating matters, in the Demazure case the crystals are truncated so Stembridge's methods are not immediately applicable.

In this paper, we remedy this impediment and develop a new local characterization of Demazure crystals. These tools allow us to overcome the apparent limitations of Stembridge's axioms and readily surpass the difficulties mentioned above. In particular, in \S\ref{sec:demazure} we consider different families of subsets of crystals with certain nice properties. This leads us to Definition~\ref{def:demazure}, where we present six local axioms for a subset of a normal $\mathfrak{gl}_n$ crystal to be considered a \newword{Demazure subset}. Our first principal result, stated precisely in Theorems~\ref{thm:dem-well-def} and \ref{thm:main-dem}, is the following:
\begin{theorem*}
  Every Demazure $\mathfrak{gl}_n$ crystal is a Demazure subset of a normal $\mathfrak{gl}_n$ crystal, and every Demazure subset of a normal $\mathfrak{gl}_n$ crystal is a Demazure $\mathfrak{gl}_n$ crystal.
\end{theorem*}
This provides a universal method for proving that a given subset of a crystal is a Demazure crystal. 
\smallskip

Furthermore, since the characters for Demazure crystals depend on the highest weight \emph{and} an element of the Weyl group, the existence of an explicit Demazure crystal does not immediately yield a formula for the character. Instead, the Demazure character is determined by a specific \emph{lowest} weight, but since lowest weights are not unique, this requires inspecting the entire crystal to determine the \emph{global} lowest weight. To overcome this obstacle, we present an algorithm in Definition~\ref{def:dem-low} that deterministically computes the global lowest weight beginning with the unique highest weight. That is, from Theorem~\ref{thm:dem-low}, we obtain a formula:
\begin{theorem*}
  If $D$ is a Demazure $\mathfrak{gl}_n$ crystal, then its character is
  \[ \mathrm{ch}(D) = \sum_{b\in D} x_1^{\wt(b)_1} x_2^{\wt(b)_2} \cdots x_{n}^{\wt(b)_{n}} = \sum_{\substack{b\in D \\ b \ \text{highest weight}}} \key_{\wt(Z(b))}, \]
  where the latter sum is over highest weight elements, $Z(b)$ is the result of applying Definition~\ref{def:dem-low} to $b$, and $\key_a$ denotes the Demazure character.
\end{theorem*}
Thus, we have an efficient formula for characters of Demazure crystals.
\smallskip

Our motivation for deriving the results in \S\ref{sec:demazure} provides our immediate application, which is to construct Demazure crystals whose characters are the nonsymmetric Macdonald polynomials specialized at $t=0$. This we do in \S\ref{sec:tabloid} Definition~\ref{def:raise-tabloid}, in which we define explicit raising and lowering operators on \newword{semistandard key tabloids}, the combinatorial objects for which the specialized nonsymmetric Macdonald polynomials are the generating functions. We use Kohnert's paradigm for Demazure characters to define an explicit map $\embed$ that embeds our structure into the normal $\mathfrak{gl}_n$ crystals on semistandard Young tableaux, giving Theorem~\ref{thm:commute}:
\begin{theorem*}
  The map $\embed$ from semistandard key tabloids to semistandard Young tableaux is a weight-preserving injective map that intertwines the crystal operators. In particular, the image of $\embed$ is a subset of a normal $\mathfrak{gl}_n$ crystal.
\end{theorem*}
Hence, we are now in the situation to apply our characterization of Demazure crystals, which culminates in Theorem~\ref{thm:demazure}, and states:
\begin{theorem*}
  The graph on semistandard key tabloids defined by the raising operators is a Demazure subcrystal of a normal crystal. Therefore, writing $E_b(X_n;q,0) = \sum_{a} K_{a,b}(q) \key_a(X_n)$, we have
  \[ K_{a,b}(q) = \sum_{\substack{T \in \SSKD(b) \\ T \ \text{highest weight} \\ \wt(Z(T)) = a }} q^{\maj(T)} , \]
  where $\SSKD(b)$ denotes the set of semistandard key tabloids of shape $b$, and $\maj$ is the Haglund--Haiman--Loehr statistic. In particular, nonsymmetric Macdonald polynomials specialized at $t=0$ are a nonnegative $q$-graded sum of Demazure characters.
\end{theorem*}

Our results give an explicit formula for this expansion, however, in the symmetric case we can say more. The \newword{Hall--Littlewood symmetric functions} may be regarded as the $q=0$ specialization of Macdonald symmetric functions. They are long known to be Schur positive and their Schur coefficients, the \newword{Kostka--Foulkes polynomials} $K_{\lambda,\mu}(t)$, have rich interpretations in geometry and representation theory. Lascoux and Sch\"{u}tzenberger \cite{LS78} recursively defined a statistic called \newword{charge} on these objects that precisely gives $K_{\lambda,\mu}(t)$. Using our formula for nonsymmetric Macdonald polynomials, we arrive at a new expression for $K_{\lambda,\mu}(t)$ using the much simpler $\maj$ statistic. In Theorem~\ref{thm:charge-crystal}, we prove the following:
\begin{theorem*}
  The Kostka--Foulkes polynomials $K_{\lambda,\mu}(t)$ are given by
  \[ K_{\lambda,\mu}(t) = \sum_{\substack{T\in\SSKD(0^{|\mu|-\mu_1}\times\mathrm{rev}(\mu^{\prime})) \\ T \ \text{highest weight} \\ \wt(T) = \lambda^{\prime}}} t^{\maj(T)} , \]
  where $\lambda^{\prime}$ denotes the conjugate of $\lambda$.
\end{theorem*}
We conclude by noting that the Demazure coefficients of specialized nonsymmetric Macdonald polynomials give a refinement of the Kostka--Foulkes polynomials that removes certain multiplicities. Moreover, as nonnegative expansions into Demazure characters are becoming more ubiquitous among geometrically significant bases for the polynomial ring, we expect our methods to have wider applications to come.

%
\section{Macdonald polynomials}
%
\label{sec:macdonald}

Symmetric functions arise in many areas of mathematics, appearing as characters of polynomial representations of the general linear group, Frobenius characters of representations of the symmetric groups, and as natural representatives of Schubert classes for Grassmannians. In these contexts, the \emph{Schur functions} and their generalizations play the pivotal role of irreducible objects, and the problem of determining the coefficients of a given symmetric function in the Schur basis combinatorializes problems of finding irreducible decompositions, branching rules, and computing intersection numbers.

In \S\ref{sec:mac-sym}, we review the rich contexts in which we find Schur functions, Hall--Littlewood symmetric functions, and Macdonald symmetric functions along with their associated combinatorics on Young tableaux. In \S\ref{sec:mac-nonsym}, we generalize these symmetric functions to the nonsymmetric setting of the full polynomial ring, where generalizations to other root systems become more accessible. Finally, in \S\ref{sec:mac-SSKD}, we motivate the specialization considered in this paper from the points of view of simplified combinatorial structures where positivity manifests in meaningful ways.

\subsection{Symmetric polynomials}
\label{sec:mac-sym}

We begin by reviewing several of the classical bases for the ring $\Lambda_{\mathbb{Q}}$ of symmetric functions in variables $X = x_1,x_2,x_3,\ldots$ over the rational numbers; for more details, see \cite{Mac95}. Bases for $\Lambda_{\mathbb{Q}}$ are naturally indexed by \newword{partitions}, which are weakly decreasing sequences of nonnegative integers. Perhaps the simplest basis for $\Lambda_{\mathbb{Q}}$ is the basis of \newword{monomial symmetric functions}, denoted by $m_{\lambda}(X)$, and defined by
\begin{equation}
  m_{\lambda}(X) = \sum_{\mathrm{sort}(a)=\lambda} x_1^{a_1} x_2^{a_2} x_2^{a_3}\cdots ,
\end{equation}
where the sum is over all weak compositions $a = (a_1,a_2,a_3,\ldots)$ whose nonzero parts rearrange the partition $\lambda$. As we shall see in the examples to come, monomial positivity is often the necessary precursor to deeper positivity results. Most of the bases we consider will also exhibit triangularity with respect to the monomial basis under the \newword{dominance partial order} on partition defined by
\begin{eqnarray}
  \lambda \leq \mu & \Leftrightarrow & \lambda_1 + \cdots + \lambda_k \leq \mu_1 + \cdots + \mu_k \qquad \forall k .
\end{eqnarray}
Dominance order refines lexicographic order, the latter of which is a total order.

Another important basis for symmetric functions with deep connects to the representation theory of the symmetric group is the basis of \newword{power sum symmetric functions}, denoted by $p_{\lambda}(X)$, and defined multiplicatively by the rules
\begin{eqnarray}
  p_{k}(X) & = & x_{1}^k + x_{2}^k + x_{3}^k + \cdots , \\
  p_{\lambda}(X) & = & p_{\lambda_1}(X) p_{\lambda_2}(X)  \cdots p_{\lambda_{\ell}}(X),
\end{eqnarray}
when $\lambda$ is a partition of length $\ell$. We can use the power sum basis to define the \newword{Hall inner product} on symmetric functions by setting
\begin{equation}
  \langle p_{\lambda}(X) , p_{\mu}(X) \rangle = z_{\lambda} \delta_{\lambda,\mu},
  \label{e:Hall}
\end{equation}
where $z_{\lambda} = \prod_{i \geq 1} i^{m_i} m_i!$ for $m_i$ the multiplicity of $i$ in $\lambda$. From the formula above, the power sum basis is orthogonal with respect to this inner product.

The basis of \newword{Schur functions}, denoted by $s_{\lambda}(X)$, is the unique symmetric function basis that is upper uni-triangular with respect to the monomial basis and orthogonal with respect to the Hall inner product. Schur polynomials may be defined combinatorially as the generating polynomial for \newword{semi-standard Young tableaux}.

The \newword{diagram} of a partition $\lambda$ has $\lambda_i$ left justified unit cells in row $i$.

\begin{definition}
  Given a partition $\lambda$, a \newword{semistandard Young tableau} of shape $\lambda$ is a filling of the Young diagram of $\lambda$ with positive integers such that entries weakly increase left to right along rows and strictly increase bottom to top along columns. We denote the set of semistandard Young tableaux of shape $\lambda$ with entries in $\{1,2,\ldots,n\}$ by $\SSYT_n(\lambda)$.
  \label{def:SSYT}
\end{definition}

For example, the semistandard Young tableaux of shape $(2,2,1)$ with entries in $\{1,2,3,4\}$ are shown in Fig.~\ref{fig:SSYT}.

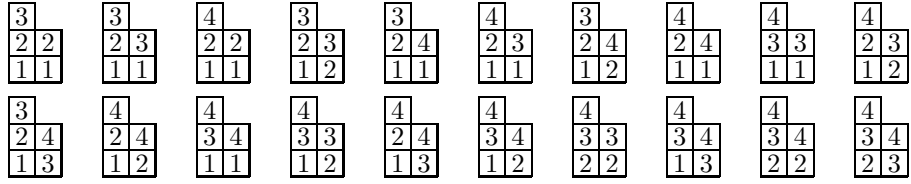
\begin{figure}[ht]
  \begin{center}
    \begin{tikzpicture}[xscale=1.25,yscale=1.25]
      \node at (0,1) (T00) {$\tableau{3 \\ 2 & 2 \\ 1 & 1}$};
      \node at (1,1) (T01) {$\tableau{3 \\ 2 & 3 \\ 1 & 1}$};
      \node at (2,1) (T11) {$\tableau{4 \\ 2 & 2 \\ 1 & 1}$};
      \node at (3,1) (Tb2) {$\tableau{3 \\ 2 & 3 \\ 1 & 2}$};
      \node at (4,1) (T22) {$\tableau{3 \\ 2 & 4 \\ 1 & 1}$};
      \node at (5,1) (T12) {$\tableau{4 \\ 2 & 3 \\ 1 & 1}$};
      \node at (6,1) (T03) {$\tableau{3 \\ 2 & 4 \\ 1 & 2}$};
      \node at (7,1) (T33) {$\tableau{4 \\ 2 & 4 \\ 1 & 1}$};
      \node at (8,1) (T13) {$\tableau{4 \\ 3 & 3 \\ 1 & 1}$};
      \node at (9,1) (Ta3) {$\tableau{4 \\ 2 & 3 \\ 1 & 2}$};
      \node at (0,0) (T04) {$\tableau{3 \\ 2 & 4 \\ 1 & 3}$};
      \node at (1,0) (T24) {$\tableau{4 \\ 2 & 4 \\ 1 & 2}$};
      \node at (2,0) (T34) {$\tableau{4 \\ 3 & 4 \\ 1 & 1}$};
      \node at (3,0) (Ta4) {$\tableau{4 \\ 3 & 3 \\ 1 & 2}$};
      \node at (4,0) (T25) {$\tableau{4 \\ 2 & 4 \\ 1 & 3}$};
      \node at (5,0) (T15) {$\tableau{4 \\ 3 & 4 \\ 1 & 2}$};
      \node at (6,0) (Tb5) {$\tableau{4 \\ 3 & 3 \\ 2 & 2}$};
      \node at (7,0) (T26) {$\tableau{4 \\ 3 & 4 \\ 1 & 3}$};
      \node at (8,0) (T06) {$\tableau{4 \\ 3 & 4 \\ 2 & 2}$};
      \node at (9,0) (T07) {$\tableau{4 \\ 3 & 4 \\ 2 & 3}$};
    \end{tikzpicture}
    \caption{\label{fig:SSYT}The twenty elements of $\SSYT_4(2,2,1)$.}
  \end{center}
\end{figure}

The \emph{weight} of a semistandard Young tableau $T$ is the weak composition $\wt(T)$ whose $i^{th}$ part is the number of entries equal to $i$.

\begin{definition}
  The \newword{Schur polynomial} $s_{\lambda}(x_1,\ldots,x_n)$ is given by
  \begin{equation}
    s_\lambda(x_1,\ldots,x_n) = \sum_{T \in \SSYT_n(\lambda)} x_1^{\wt(T)_1} \cdots x_n^{\wt(T)_n}.
  \end{equation}
  The \newword{Schur function} $s_{\lambda}(X)$ is the stable limit of the Schur polynomial as $n$ grows.
\end{definition}

We may define the \newword{Kostka numbers}, denoted by $K_{\lambda,\mu}$ as the transition coefficients between the Schur basis and the monomial basis, i.e.
\begin{equation}
  s_{\lambda}(X) = \sum_{\lambda} K_{\lambda,\mu} m_{\mu}(X) ,
  \label{e:kostka}
\end{equation}
where $K_{\lambda,\mu}$ is the number of semistandard Young tableaux of shape $\lambda$ and weight $\mu$. In particular, we have $K_{\lambda,\mu} \in \mathbb{N}$.

Schur polynomials arise in many important contexts wherein expansions of symmetric functions into the Schur basis becomes of fundamental importance. For $V^{\lambda}$ the irreducible polynomial representation of $\mathrm{GL}_n(\mathbb{C})$, its character is given by $\mathrm{char}(V^{\lambda}) = s_{\lambda}(x_1,\ldots,x_n)$. Given any polynomial representation $V$, its character $\mathrm{char}(V)$ is a symmetric polynomial, and so the expansion of $\mathrm{char}(V)$ into the Schur basis corresponds precisely to the irreducible decomposition of $V$, i.e.
\begin{eqnarray*}
  V \cong \bigoplus_{\lambda} \left(V^{\lambda}\right)^{\oplus c_{\lambda}}
  & \Leftrightarrow &
  \mathrm{char}(V) = \sum_{\lambda} c_{\lambda} s_{\lambda}(x_1,\ldots,x_n) .
\end{eqnarray*}
Under this paradigm, the Hall inner product corresponds precisely to the inner product on characters, and the coefficients of the Schur polynomials expanded into monomial basis give highest weight multiplicities. Dually, for $S^{\lambda}$ the irreducible representation of $\mathcal{S}_n$ over $\mathbb{C}$, its Frobenius character is given by $\mathrm{ch}(S^{\lambda}) = s_{\lambda}(X)$.

Geometrically, the Chern class of the Schubert variety $X_{\lambda}$ for the Grassmannian $\mathrm{Gr}(n,k)$ is naturally represented by the Schubert polynomial $\mathfrak{S}_{w(\lambda,k)} = s_{\lambda}(x_1,\ldots,x_k)$. Therefore intersection numbers for Grassmannian Schubert varieties can be computed by taking the Schur coefficients of the product of Schur polynomials.

Consider now symmetric functions over $\mathbb{Q}(q,t)$ for two independent indeterminants $q, t$. Here we may define a generalization of the Hall inner product by
\begin{equation}
  \langle p_{\lambda}(X) , p_{\mu}(X) \rangle_{q,t} = z_{\lambda} \delta_{\lambda,\mu} \prod_{i=1}^{\ell(\lambda)} \frac{1-q^{\lambda_i}}{1-t^{\lambda_i}} .
  \label{e:Hall-qt}
\end{equation}
Taking $q=t$ in Eq.~\eqref{e:Hall-qt} results in the classic Hall inner product in Eq.~\eqref{e:Hall}.

Macdonald \cite{Mac88} defined a new basis of symmetric functions over this larger ground field using this generalized inner product.

\begin{definition}[\cite{Mac88}]
  The \newword{Macdonald symmetric functions} $P_{\lambda}(X;q,t)$ are the unique basis for $\Lambda_{\mathbb{Q}(q,t)}$ that are upper uni-triangular with respect to monomial symmetric functions and are orthogonal with respect to the generalized Hall inner product in Eq.~\eqref{e:Hall-qt}.
\end{definition}

Given that dominance order is a partial order, this definition requires a theorem to be well-defined. However, the uniqueness is obvious, as is the specialization
\begin{equation}
  P_{\lambda}(X;q,q) = s_{\lambda}(X) .
\end{equation}
In fact, Macdonald defined this basis to be a simultaneous generalization of the \newword{Hall--Littlewood symmetric functions} $P_{\lambda}(X;0,t)$ and the \newword{Jack symmetric functions} $\lim_{t\rightarrow 1}P_{\lambda}(X;t^{\alpha},t)$, both of which have deep connections to representation theory and geometry.

Macdonald also considered a slight modification of the $P_{\mu}(X;q,t)$ basis called the \newword{integral form}, denoted by $J_{\mu}(X;q,t)$, and related to the $P_{\mu}(X;q,t)$ basis by
\begin{equation}
  J_{\mu}(X;q,t) = \prod_{c\in\lambda} \left(1 - q^{\mathrm{arm}(c)}t^{\mathrm{leg}(c)+1} \right) P_{\mu}(X;q,t) ,
\end{equation}
where for $c$ a cell of the diagram of $\lambda$, we set $\mathrm{arm}(c)$ to be the number of cells strictly right of $c$ and $\mathrm{leg}(c)$ the number of cells strictly above $c$. With this basis, we may define the \newword{Kostka--Macdonald polynomials} denoted by $K_{\lambda,\mu}(q,t)$ by
\begin{equation}
  J_{\mu}(X;q,t) = \sum_{\lambda} K_{\lambda,\mu}(q,t) s_{\lambda}[X(1-t)] ,
\end{equation}
where $s_{\lambda}[X(1-t)]$ denotes the \newword{plethystic Schur basis}, which may be defined as the dual basis to the Schur functions under the generalized Hall inner product Eq.~\eqref{e:Hall-qt} at $q=0$, i.e.
\[ \langle \ s_{\lambda}[X(1-t)] , s_{\mu}(X) \rangle_{0,t} = \delta_{\lambda,\mu} . \]
A priori, the coefficients $K_{\lambda,\mu}(q,t)$ are rational functions in the parameters $q,t$ with rational coefficients. Based on hand computations, Macdonald conjectured that, in fact, $K_{\lambda,\mu}(q,t)$ are \emph{polynomials} in $q,t$ with \emph{nonnegative integer} coefficients.

Garsia and Haiman \cite{GH93} considered the \newword{modified Macdonald polynomial} $H_{\mu}(X;q,t)$ that relates to the integral form via plethysm as
\begin{equation}
  H_{\mu}(X;q,t) = J_{\mu}[X\left(\textstyle{\frac{1}{1-t}}\right);q,t]  = \sum_{\lambda} K_{\lambda,\mu}(q,t) s_{\lambda}(X) ,
  \label{e:modified}
\end{equation}
where now the Kostka--Macdonald coefficients precisely give the Schur expansion of the modified Macdonald polynomial. Thus we have fallen into the fundamental problem of giving a combinatorial interpretation for the Schur coefficients of a given symmetric function.

Garsia and Haiman \cite{GH96}, building on earlier work of Garsia and Procesi \cite{GP92} on Hall--Littlewood polynomials, constructed a bi-graded $\mathcal{S}_n$ module and proved that if the dimension of the module is $n!$, then its bi-graded Frobenius character must be $H_{\mu}(X;q,t)$. As Schur functions are the Frobenius characters of the irreducible representations of $\mathcal{S}_n$, this would prove Macdonald's conjecture. Haiman \cite{Hai01} analyzed the isospectral Hilbert scheme of points in the plane, ultimately showing that it is Cohen--Macaulay (and Gorenstein), and from this established the $n!$ Conjecture of Garsia and Haiman as well as Macdonald positivity.

\begin{theorem}[\cite{Hai01}]
  The Kostka--Macdonald polynomials are polynomials in $q,t$ with nonnegative integers coefficients, i.e. $K_{\lambda,\mu}(q,t) \in \mathbb{N}[q,t]$.
\end{theorem}

Nevertheless, it remains an important open problem in algebraic combinatorics to give a manifestly positive formula for $K_{\lambda,\mu}(q,t)$.

\subsection{Nonsymmetric polynomials}
\label{sec:mac-nonsym}

We turn our focus now to the full polynomial ring $\mathbb{Q}[x_1,\ldots,x_n]$ in variables $X_n = x_1,\ldots,x_n$, which as the obvious basis of monomials indexed by weak compositions $a=(a_1,\ldots,a_n)\in\mathbb{N}^{n}$ given by $X_n^a = x_1^{a_1} \cdots x_n^{a_n}$. The \newword{Bruhat order} on weak compositions given by the transitive closure of the cover relations
\begin{displaymath}
  \begin{array}{cccl}
    a & > & t_{i,j} \cdot a & \text{ if } i<j \text{ and } a_i < a_j \\
    t_{i,j} \cdot a & > & a + \mathbf{e}_i - \mathbf{e}_j & \text{ if }i < j \text{ and } a_j - a_i > 1,
  \end{array}
\end{displaymath}
where $\mathbf{e}_i$ is the $i$th standard basis vector. Triangularity in the polynomial setting will be with respect to this partial order which refines lexicographic order.

Opdam \cite{Opd95} and Macdonald \cite{Mac96} introduced a polynomial generalization of Macdonald symmetric functions that were generalized to any root system by Cherednik \cite{Che95}. Expanding the ground field to include the two parameters $q,t$, the \newword{Cherednik inner product} on $\mathbb{Q}(q,t)[x_1,\ldots,x_n]$ is given by
\begin{equation}
  \langle f,g \rangle = [X_n^0]\left(f \ \overline{g} \ \frac{\Delta}{[X_n^0]\Delta} \right) ,
  \label{e:cherednik}
\end{equation}
where $\overline{\cdot}$ is defined linearly by $\overline{q}=q^{-1}$, $\overline{t}=t^{-1}$, $\overline{x_i}=x_i^{-1}$ and
\[ \Delta = \prod_{i<j} \prod_{k\geq 0} \frac{(1-q^kx_i/x_j)(1-q^{k+1}x_j/x_i)}{(1-tq^kx_i/x_j)(1-tq^{k+1}x_j/x_i)} . \]

Parallel to the characterization of $P_{\mu}(X;q,t)$, we have the following.

\begin{definition}[\cite{Che95}]
  The \newword{nonsymmetric Macdonald polynomials} $E_{a}(X_n;q,t)$ are the unique basis for $\mathbb{Q}(q,t)[x_1,\ldots,x_n]$ that are upper uni-triangular with respect to monomials and are orthogonal with respect to the Cherednik's inner product.
\end{definition}

The nonsymmetric Macdonald polynomials can be realized as a truncation of the nonsymmetric Macdonald polynomials which, in addition, shows that the symmetric \emph{functions} are the \newword{stable limit},
\begin{eqnarray}
  E_{0^m \times a}(x_1,\ldots,x_m,0,\ldots,0;q,t) & = & P_{\mathrm{sort}(a)}(x_1,\ldots,x_m;q,t)  \label{e:trunc-E}
 \\
 \lim_{m\rightarrow\infty} E_{0^m \times a}(x_1,\ldots,x_m,0,\ldots,0;q,t) & = & P_{\mathrm{sort}(a)}(x_1,x_2,\ldots;q,t),
  \label{e:stable-E}
\end{eqnarray}
where $\mathrm{sort}(a)$ is the partition rearrangement of the weak composition $a$.

Recall that Schur functions appear as a specialization of Macdonald symmetric functions, namely $P_{\lambda}(X;0,0) = s_{\lambda}(X)$. Ion \cite{Ion05} showed that the analogous specialization of nonsymmetric Macdonald polynomials is a \newword{Demazure character}, namely $E_{a}(X_n;0,0) = \key_{a}(X_n)$. Demazure characters, which form a geometrically significant basis for the full polynomial ring, are presented in depth in section~\ref{sec:crystal-demazure}, but for now we note that they are generalizations of Schur polynomials in the same senses as Eqs.~\eqref{e:trunc-E} and \eqref{e:stable-E}. This provides a natural place to begin searching for meaningful positivity results in the nonsymmetric setting.

Since Demazure characters are monomial positive, any positivity results for Demazure characters must include monomial positivity as well. The coefficient of $X_n^a$ in $E_{b}(X_n;q,t)$ is nonzero if and only if $a \leq b$ in Bruhat order, but these coefficients lie in $\mathbb{Q}(q,t)$, and so monomial positivity must lie elsewhere.

The \newword{nonsymmetric integral form}, denoted by $\mathcal{E}_{b}(X_n;q,t)$, is given by
\begin{equation}
  \mathcal{E}_{b}(X_n;q,t) = \prod_{c\in a} \left(1 - q^{\mathrm{leg}(c)+1}t^{\mathrm{arm}(c)+1} \right) E_{b}(X_n;q,t) ,
  \label{e:nsintegral}
\end{equation}
where the \newword{leg} of a cell $c$ in a composition diagram is the number of cells strictly right of $c$ in the same row, and the \newword{arm} of $c$ is the number of cells strictly below $c$ in the same column whose row is weakly shorter than that of $c$ plus the number of cells strictly above and one column left of $c$ whose row is strictly shorter.

Knop \cite{Kno97} showed that $\mathcal{E}_{b}(X_n;q,t)$ has its monomial coefficients in $\mathbb{Z}[q,t]$, paving the way for further positivity. However, recall that Macdonald positivity arose only when considering \emph{plethystic substitutions}. At present, there is no well-defined notion of plethysm for the full polynomial ring.

To circumvent this difficulty, from the combinatorial formula for nonsymmetric Macdonald polynomials due to Haglund, Haiman, and Loehr \cite{HHL08}, one sees that when specializing the single parameter $t=0$, the nonsymmetric Macdonald polynomial and its integral form coincide and, moreover, become \emph{monomial positive}. Assaf \cite{Ass18} proved this specialization $E_b(X_n;q,0)$ is, in fact, Demazure positive.

\begin{theorem}[\cite{Ass18}]
  For weak compositions $a,b$, define coefficients $K_{a,b}(q)$ by
  \begin{equation}
    E_b(X_n;q,0) = \sum_{a} K_{a,b}(q) \key_a(X_n) .
    \label{e:nskostka}
  \end{equation}
  Then we have $K_{a,b}(q) \in \mathbb{N}[q]$. In particular, nonsymmetric Macdonald polynomials specialized at $t=0$ are a nonnegative $q$-graded sum of Demazure characters.
  \label{thm:nskostka}
\end{theorem}

While Assaf's proof is combinatorial, it does not give a direct formula for the Demazure expansion. Assaf proves that $E_b(X_n;q,0)$ is nonnegative on the \newword{fundamental slide polynomials}, a basis for $\mathbb{Z}[x_1,x_2,\ldots]$ developed by Assaf and Searles \cite{AS17} arising from their study of Schubert polynomials. From there, she uses the machinery of \newword{weak dual equivalence} \cite{Ass-W} to group terms in the fundamental slide expansion into Demazure characters, which Assaf and Searles \cite{AS18} showed are fundamental slide positive. However, extracting a formula requires one to write fundamental slide polynomials in terms of Demazure characters, which is inefficient and introduces negative signs, albeit ones that ultimately cancel.

In the present paper, we use the theory of crystal bases to give a new combinatorial proof of Theorem~\ref{thm:nskostka} that yields a manifestly positive formula for $K_{a,b}(q)$. Moreover, as crystals themselves are combinatorial skeletons of representations, this also gives a representation theoretic model for these specialized nonsymmetric Macdonald polynomials.

\subsection{Semistandard key tabloids}
\label{sec:mac-SSKD}

Haglund, Haiman and Loehr \cite{HHL08} gave a combinatorial formula for the monomial expansion of nonsymmetric Macdonald polynomials. Integrality for the nonsymmetric integral form is immediate from their formula, as is monomial positivity for the specialization we consider.

The \newword{diagram} of a weak composition has $a_i$ cells left-justified in row $i$, indexed in coordinate notation with row $1$ at the bottom.

Two cells of a diagram are \newword{attacking} if they lie in the same column or if they lie in adjacent columns with the cell on the left strictly higher than the cell on the right. A filling is \newword{non-attacking} if no two attacking cells have the same value.

\begin{figure}[ht]
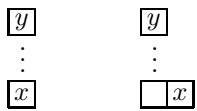

  \begin{displaymath}
    \begin{array}{l}
      \tableau{ y } \\[-0.5\cellsize] \hspace{0.4\cellsize} \vdots \\ \tableau{ x }
    \end{array}\hspace{3\cellsize}
    \begin{array}{l}
      \tableau{ y } \\[-0.5\cellsize] \hspace{0.4\cellsize} \vdots \\ \tableau{ \ & x }
    \end{array}
  \end{displaymath}
  \caption{\label{fig:attack}The two relative positions for attacking cells $x,y$.}
\end{figure}

For a non-attacking filling $T$, the \newword{major index} of $T$, denoted by $\maj(T)$, is the sum of the legs of all cells $c$ such that the entry in $c$ is strictly less than the entry immediately to its right, as illustrated in Fig.~\ref{fig:maj}.

\begin{figure}[ht]
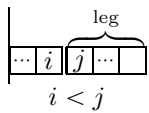

  \begin{displaymath}
    \begin{array}{c}
      \vline\tableau{_{\cdots}}\hss\tableau{i}\hss\overbrace{\tableau{j}\hss\tableau{_{\cdots}}\hss\tableau{ \ }}^{\mathrm{leg}} \\
      i < j
    \end{array}
  \end{displaymath}
  \caption{\label{fig:maj}The leg of a cell contributing to the major index.}
\end{figure}

A \newword{triple} is a collection of three cells with two row adjacent and either (Type I) the third cell is above the left and the lower row is strictly longer, or (Type II) the third cell is below the right and the higher row is weakly longer. The \newword{orientation} of a triple is determined by reading the entries of the cells from smallest to largest. A \newword{co-inversion triple} is a Type I triple oriented counterclockwise or a Type II triple oriented clockwise, as illustrated in Fig.~\ref{fig:inv}. Note that, as proved in \cite{HHL08}(Lemma~3.6.3), the entries of the cells that form a co-inversion triple are necessarily distinct.

\begin{figure}[ht]
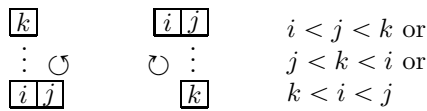

  \begin{displaymath}
      \begin{array}{l}
        \tableau{ k } \\[-0.5\cellsize] \hspace{0.4\cellsize} \vdots \hspace{0.5\cellsize} \circlearrowleft \\ \tableau{ i & j }
      \end{array}\hspace{2\cellsize}
      \begin{array}{r}
        \tableau{ i & j } \\[-0.5\cellsize] \circlearrowright \hspace{0.5\cellsize} \vdots \hspace{0.4\cellsize} \\ \tableau{ k }
      \end{array}\hspace{2\cellsize}
      \begin{array}{l}
        i < j < k \text{ or} \\
        j < k < i \text{ or} \\
        k < i < j
      \end{array}
  \end{displaymath}
  \caption{\label{fig:inv}The positions and orientation for co-inversion triples.}
\end{figure}


Generalizing their earlier formula for Macdonald symmetric functions \cite{HHL05}, Haglund, Haiman and Loehr gave the following explicit combinatorial formula for the monomial expansion of the nonsymmetric Macdonald polynomials \cite{HHL08}, which also yields a formula for the nonsymmetric integral form.

\begin{theorem}[\cite{HHL08}]
  The nonsymmetric Macdonald polynomial $E_{a}(X_n;q,t)$ is given by
  \begin{equation}
    E_{a}(X_n;q,t) = \sum_{\substack{T:a\rightarrow [n] \\ \mathrm{non-attacking}}} q^{\maj(T)} t^{\coinv(T)} X_n^{\wt(T)}
    \prod_{c \neq \mathrm{left}(c)} \frac{1-t}{1 - q^{\mathrm{leg}(c)+1} t^{\mathrm{arm}(c)+1}} .
  \end{equation}
  \label{thm:mac}
\end{theorem}

Comparing with Eq.~\eqref{e:nsintegral}, we see that while the denominator can be cleared, the appearance of negative signs is inevitable even for the integral form. However, when specializing to $t=0$, the product on the right collapses to $1$ giving a manifestly positive monomial expansion. We review notation from \cite{Ass18}.

\begin{definition}
  Given a weak composition $a$, a \newword{semistandard key tabloid} of shape $a$ is a non-attacking filling of the composition diagram of $a$ with positive integers such that there are no co-inversion triples. We denote the set of semistandard Young tableaux of shape $a$ by $\SSKD(a)$.
  \label{def:SSKD}
\end{definition}

For example, Fig.~\ref{fig:SSKD} shows the semistandard key tabloids of shape $(0,2,1,2)$.

\begin{figure}[ht]
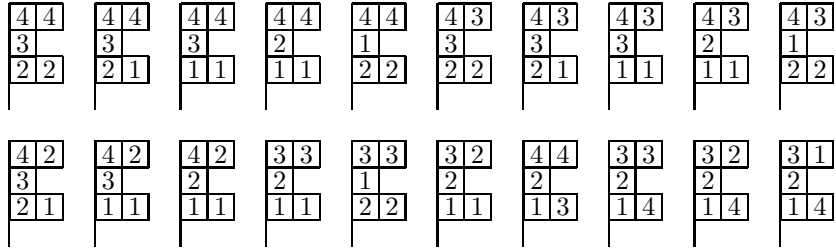

  \begin{displaymath}
    \arraycolsep=6pt
    \begin{array}{cccccccccc}
      \vline\tableau{4 & 4 \\ 3 \\ 2 & 2 \\ & } &
      \vline\tableau{4 & 4 \\ 3 \\ 2 & 1 \\ & } &
      \vline\tableau{4 & 4 \\ 3 \\ 1 & 1 \\ & } &
      \vline\tableau{4 & 4 \\ 2 \\ 1 & 1 \\ & } &
      \vline\tableau{4 & 4 \\ 1 \\ 2 & 2 \\ & } &
      \vline\tableau{4 & 3 \\ 3 \\ 2 & 2 \\ & } &
      \vline\tableau{4 & 3 \\ 3 \\ 2 & 1 \\ & } &
      \vline\tableau{4 & 3 \\ 3 \\ 1 & 1 \\ & } &
      \vline\tableau{4 & 3 \\ 2 \\ 1 & 1 \\ & } &
      \vline\tableau{4 & 3 \\ 1 \\ 2 & 2 \\ & } \\ \\
      \vline\tableau{4 & 2 \\ 3 \\ 2 & 1 \\ & } &
      \vline\tableau{4 & 2 \\ 3 \\ 1 & 1 \\ & } &
      \vline\tableau{4 & 2 \\ 2 \\ 1 & 1 \\ & } &
      \vline\tableau{3 & 3 \\ 2 \\ 1 & 1 \\ & } &
      \vline\tableau{3 & 3 \\ 1 \\ 2 & 2 \\ & } &
      \vline\tableau{3 & 2 \\ 2 \\ 1 & 1 \\ & } &
      \vline\tableau{4 & 4 \\ 2 \\ 1 & 3 \\ & } &
      \vline\tableau{3 & 3 \\ 2 \\ 1 & 4 \\ & } &
      \vline\tableau{3 & 2 \\ 2 \\ 1 & 4 \\ & } &
      \vline\tableau{3 & 1 \\ 2 \\ 1 & 4 \\ & }
    \end{array}
  \end{displaymath}
  \caption{\label{fig:SSKD}The semistandard key tabloids of shape $(0,2,1,2)$.}
\end{figure}

Classically, a \newword{semistandard Young tabloid} is a filling of a Young diagram with weakly increasing rows and no column condition. Thus a tabloid is determined by its row sets, since there is a unique ordering for each row that results in a valid filling. Our nomenclature for \emph{semistandard key tabloids} arises from the same paradigm, though now emphasis is placed on columns rather than on rows.

\begin{proposition}
  Given two semistandard key tabloids $S,T \in \SSKT(a)$, if $S$ and $T$ have the same set of entries within each column, then $S=T$.
\end{proposition}

\begin{proof}
  This follows for \emph{standard} key tabloids from \cite{Ass18}(Theorem~5.6), and extends to semistandard key tabloids by the usual destandardization used in the proof of \cite{Ass18}(Proposition~2.6).
\end{proof}

%
\section{Crystals for the general linear group}
%
\label{sec:crystal}

Kashiwara's theory of crystal bases \cite{Kas91} provides a powerful tool for studying representations as well as for categorifying Schur positive symmetric functions by providing the combinatorial skeleton of a representation whose character is the given symmetric function. 

In \S\ref{sec:crystal-normal}, we recall the basic definitions for abstract and normal crystals in the case when $\mathfrak{g}$ is the general linear group $\gl_n$. In \S\ref{sec:crystal-demazure}, we consider the action of the Borel subalgebra $\mathfrak{b}$ consisting of upper-triangular matrices on extremal weight spaces indexed by the Weyl group $\mathcal{S}_n$ and we review the corresponding crystal theory associated with these \newword{Demazure modules}. In \S\ref{sec:crystal-tableaux}, give an explicit realization of normal and Demazure crystals with base indexed by semistandard Young tableaux and semistandard key tableaux, respectively.

\subsection{Normal crystals}
\label{sec:crystal-normal}

Let $\mathbf{e}_1, \mathbf{e}_2, \ldots, \mathbf{e}_{n}$ denote the standard basis for $V = \mathbb{R}^{n}$ with the usual inner product. The \newword{root system} $\Phi = \{\mathbf{e}_i - \mathbf{e}_j \mid i \neq j\}$ contains a subset of \newword{positive roots} $\Phi^{+} = \{\mathbf{e}_i - \mathbf{e}_j \mid i<j\}$ which in turn contains the \newword{simple roots} $\alpha_i = \mathbf{e}_i - \mathbf{e}_{i+1}$ for $i=1,\ldots,n-1$. The \newword{weight lattice} $\Lambda = \mathbb{Z}^{n}$ contains a subset of \newword{dominant weights} $\Lambda^{+} \subset \Lambda$ defined as those $\lambda \in \Lambda$ such that $\lambda_1 \geq \lambda_2 \geq \cdots \geq \lambda_{n} \geq 0$.

\begin{definition}
  A finite \newword{$\gl_n$-crystal} of dimension $n$ is a nonempty, finite set $\B$ not containing $0$ together with \newword{crystal operators} $e_i, f_i  :  \B \rightarrow \B \cup \{0\}$ for $i=1,2,\ldots,n-1$ and a \newword{weight map} $\wt : \B \rightarrow \Lambda$ satisfying the conditions
  \begin{enumerate}
  \item for $b,b^{\prime}\in\B$, $e_i(b)=b^{\prime}$ if and only if $f_i(b^{\prime}) = b$, and in this case we have $\wt(b^{\prime}) = \wt(b) + \alpha_i$;
  \item for $b\in\B$ and $i=1,\ldots,n-1$, we have $\varphi_i(b) - \varepsilon_i(b) = \wt(b)_i - \wt(b)_{i+1}$, where $\varepsilon_i, \varphi_i : \B \rightarrow \mathbb{Z}$ are the \newword{string lengths} given by
    \begin{eqnarray*}
      \varepsilon_i(b) & = & \max\{k \in \mathbb{Z}_{\geq 0} \mid e_i^k(b) \neq 0 \} \\
      \varphi_i(b) & = & \max\{k \in \mathbb{Z}_{\geq 0} \mid f_i^k(b) \neq 0 \}.
    \end{eqnarray*}
  \end{enumerate}
  \label{def:base}
\end{definition}

Note that $\gl_n$-crystals may be defined more generally, though those under consideration in this paper will always be both finite and \newword{semi-normal}.

Abusing notation, we often refer to a crystal by its underlying set $\B$ when the weight map and crystal operators are understood from context.

A \newword{crystal graph} is a directed, colored graph with vertex set given by the crystal basis $\B$ and directed edges given by the crystal lowering operators $f_i$, where we draw an $i$-edge from $b$ to $f_i(b)$ if $f_i(b)\neq 0$ and all edges to $0$ are omitted.

\begin{example}
  The \newword{standard crystal} $\B(n)$ has basis $\left\{\raisebox{-0.3\cellsize}{$\tableau{i}$} \mid i=1,\ldots,n\right\}$, weight map $\wt\left(\,\raisebox{-0.3\cellsize}{$\tableau{i}$}\,\right) = \mathbf{e}_i$, crystal raising (resp. lowering) operators $e_j$ (resp. $f_j$) that act by decrementing (resp. incrementing) the entry if $j=i+1$ (resp. $j=i$), and taking it to $0$ otherwise. We draw the crystal graph for $\B(n)$ as shown in Fig.~\ref{fig:standard}.
\end{example}

\begin{figure}[ht]
  \begin{displaymath}
    \begin{tikzpicture}[xscale=1.5,yscale=1]
      \node at (0,0)   (a) {$\tableau{1}$};
      \node at (1,0)   (b) {$\tableau{2}$};
      \node at (2,0)   (c) {$\tableau{3}$};
      \node at (3,0)   (d) {$\cdots$};
      \node at (4,0)   (e) {$\tableau{n}$};
      \draw[thick,color=blue  ,->] (a) -- (b) node[midway,above] {$1$} ;
      \draw[thick,color=purple,->] (b) -- (c) node[midway,above] {$2$} ;
      \draw[thick,color=violet,->] (c) -- (d) node[midway,above] {$3$} ;
      \draw[thick,color=orange,->] (d) -- (e) node[midway,above] {$n-1$} ;
    \end{tikzpicture}
  \end{displaymath}
  \caption{\label{fig:standard}The standard crystal $\B(n)$ for $\gl_n$.}
\end{figure}
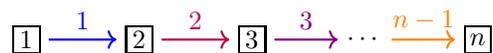

We say a crystal $\B$ is \newword{connected} if its underlying crystal graph is connected as a(n undirected) graph. A subset $X\subseteq \B$ has an induced structure coming from the crystal structure on $\B$, and whenever $X$ is a connected component of $\B$ this structure will be a crystal. In this case we call $X$ a \newword{full subcrystal} of $\B$.

\begin{definition}
  The \newword{character} of a crystal $\B$ is the polynomial
  \begin{equation}
    \mathrm{ch}(\B) = \sum_{b \in \B} x_1^{\wt(b)_1} x_2^{\wt(b)_2} \cdots x_{n}^{\wt(b)_{n}}.
    \label{e:char}
  \end{equation}
  \label{def:char}
\end{definition}

From Definition~\ref{def:base}(1), if $b,b^{\prime} \in \B$ are elements of the same full subcrystal of $\B$, then we have $\sum_i \wt(b)_i = \sum_i \wt(b^{\prime})_i$. In particular, the character of a full subcrystal is a homogeneous polynomial of fixed degree.

For example, the standard crystal $\B(n)$ degree $1$, and its character is $\mathrm{ch}(\B(n)) = x_1 + x_2 + \cdots + x_n$, which is both homogeneous and \newword{symmetric}.

The \newword{Weyl group} for $\gl_n$ is the symmetric group $\mathcal{S}_n$, which has a natural action on $\gl_n$-crystals described as follows. For $1 \leq i<n$, let $S_i$ act on $\B$ by reflecting $b\in\B$ across its $i$-string, written formally as
\begin{equation}\label{def:Stringflip-operator}
  S_i(b) = \left\{ \begin{array}{rl}
    f_i^{\wt(b)_i - \wt(b)_{i+1}} (b) & \text{if } \wt(b)_{i} \geq \wt(b)_{i+1}, \\
    e_i^{\wt(b)_{i+1} - \wt(b)_i} (b) & \text{if } \wt(b)_{i+1} \geq \wt(b)_{i} .
\end{array} \right.
\end{equation}
Kashiwara \cite{Kas91} showed these operators satisfy the braid relations for the symmetric group, thus to any permutation $w \in \mathcal{S}_n$ we may define $S_w$ by $S_{i_1} S_{i_2} \cdots S_{i_k}$ whenever $w = s_{i_1} s_{i_2} \cdots s_{i_k}$ is a reduced expression for $w$.

\begin{proposition}\label{prop:symchar}
  The character of a (finite, semi-normal) $\gl_n$-crystal is a symmetric polynomial in the variables $x_1, x_2,\ldots, x_n$.
\end{proposition}

\begin{definition}
  An element $u \in \B$ of a $\gl_n$-crystal is a \newword{highest weight element} if $e_i(u) = 0$ for all $i=1,2,\ldots,n-1$.
  \label{def:hw}
\end{definition}

Given Proposition~\ref{prop:symchar}, one might ask what symmetries the underlying $\gl_n$-crystal possesses. Indeed, analogously to Definition~\ref{def:hw}, any $\gl_n$-crystal contains a unique \newword{lowest weight element} $z$ characterized by the property $f_i(z)=0$ for all $1\leq i \leq n-1$. 

Just like every $x \in \B$ is connected to the highest weight element $u \in \B$ by a sequence of lowering operators $f_{i_m}^{r_m}\dots f_{i_1}^{r_1}(u)=x$, any $x \in \B$ is also connected to the lowest weight element $z \in \B$ via a sequence of raising operators $e_{j_p}^{r'_p}\dots e_{j_1}^{r'_1}(z)=x$. In particular, for every $x \in \B$ with $f_{i_m}^{r_m}\dots f_{i_1}^{r_1}(u)=x$, there exists $y \in \B$ such that $e_{n-i_m}^{r_m}\dots e_{n-i_1}^{r_1}(z)=y$ and vice versa. The symmetry within the crystal that swaps the highest and lowest elements and flips the remaining vertices accordingly is concretely stated as follows. 

\begin{definition}\label{def:crystalflip}
  Let $\B$ be a finite, connected semi-normal $\gl_n$-crystal, $b=u \in \B$ its highest weight element, and $z \in \B$ its lowest weight element. The \newword{crystal flip map} $\mathcal{F}: \B \to \B$ is the involution that sends each element
\[ f_{i_m}^{r_m}\dots f_{i_1}^{r_1}(u) \mapsto e_{n-i_m}^{r_m}\dots e_{n-i_1}^{r_1}(z)
\]
and any edge $x \xrightarrow{f_i} y$ to the corresponding edge $\mathcal{F}(x) \xrightarrow{e_{n-i}} \mathcal{F}(y)$ between the images of the vertices. 
\end{definition}

\begin{remark}
We note that at the level of characters, the map $\mathcal{F}$ acts by conjugating each summand by the element of the Weyl group corresponding to the half twist, that is, with the permutation $n \; n-1 \cdots 3 \; 2 \; 1$. Since by Proposition~\ref{prop:symchar} the character of a $\gl_n$-crystal is symmetric, then the action of $S_n$ is trivial, and thus the character the crystal is unchanged as expected. 
\end{remark}

From Definition~\ref{def:base}(2), since $\varepsilon_i(b)=0$ for any highest weight element $b$, we necessarily have $\wt(b)$ is dominant. For example, the highest weight element of $\B(n)$ is $\raisebox{-0.3\cellsize}{$\tableau{1}$}$, which has weight $\mathbf{e}_1 \in \Lambda^{+}$.

Conversely, dominant weights also index irreducible representations of $\gl_n$. In order to strengthen the connection between crystals and the representation theory of $\gl_n$, we must restrict our attention to \newword{normal crystals}, those arising as full subcrystals of tensor products of the standard crystal.

\begin{definition}
  Given two crystals $\B_1$ and $\B_2$, the \newword{tensor product} $\B_1 \otimes \B_2$ is the set $\B_1 \otimes \B_2$ together with crystal operators $e_i, f_i$ defined on $\B_1 \otimes \B_2$ by
  \begin{equation}
    f_i(b_1 \otimes b_2) = \left\{ \begin{array}{rl}
      f_i(b_1) \otimes b_2 & \mbox{if } \varepsilon_i(b_2) < \varphi_i(b_1), \\
      b_1 \otimes f_i(b_2) & \mbox{if } \varepsilon_i(b_2) \geq \varphi_i(b_1),
    \end{array} \right.
  \end{equation}
  and weight function $\wt(b_1 \otimes b_2) = \wt(b_1) + \wt(b_2)$ computed coordinate-wise.
\label{def:tensor-A}
\end{definition}

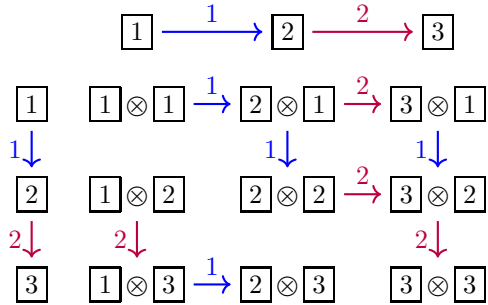
\begin{figure}[ht]
  \begin{center}
    \begin{tikzpicture}[xscale=2,yscale=1.2]
      \node at (1,2.8)   (T1)  {$\boxed{1}$};
      \node at (2,2.8)   (T2)  {$\boxed{2}$};
      \node at (3,2.8)   (T3)  {$\boxed{3}$};
      \node at (0.3,2)   (S1)  {$\boxed{1}$};
      \node at (0.3,1)   (S2)  {$\boxed{2}$};
      \node at (0.3,0)   (S3)  {$\boxed{3}$};
      \node at (1,2)   (U11)  {$\boxed{1}\otimes\boxed{1}$};
      \node at (2,2)   (U21)  {$\boxed{2}\otimes\boxed{1}$};
      \node at (3,2)   (U31)  {$\boxed{3}\otimes\boxed{1}$};
      \node at (1,1)   (U12)  {$\boxed{1}\otimes\boxed{2}$};
      \node at (2,1)   (U22)  {$\boxed{2}\otimes\boxed{2}$};
      \node at (3,1)   (U32)  {$\boxed{3}\otimes\boxed{2}$};
      \node at (1,0)   (U13)  {$\boxed{1}\otimes\boxed{3}$};
      \node at (2,0)   (U23)  {$\boxed{2}\otimes\boxed{3}$};
      \node at (3,0)   (U33)  {$\boxed{3}\otimes\boxed{3}$};
      \draw[thick,blue  ,->](T1) -- (T2)   node[midway,above]{$1$};
      \draw[thick,purple,->](T2) -- (T3)   node[midway,above]{$2$};
      \draw[thick,blue  ,->](S1) -- (S2)   node[midway,left] {$1$};
      \draw[thick,purple,->](S2) -- (S3)   node[midway,left] {$2$};
      \draw[thick,blue  ,->](U11) -- (U21) node[midway,above]{$1$};
      \draw[thick,blue  ,->](U21) -- (U22) node[midway,left] {$1$};
      \draw[thick,blue  ,->](U31) -- (U32) node[midway,left] {$1$};
      \draw[thick,blue  ,->](U13) -- (U23) node[midway,above]{$1$};
      \draw[thick,purple,->](U21) -- (U31) node[midway,above]{$2$};
      \draw[thick,purple,->](U22) -- (U32) node[midway,above]{$2$};
      \draw[thick,purple,->](U32) -- (U33) node[midway,left] {$2$};
      \draw[thick,purple,->](U12) -- (U13) node[midway,left] {$2$};
    \end{tikzpicture}
    \caption{\label{fig:tensor}Tensor product of two standard $\gl_3$ crystals.}
  \end{center}
\end{figure}

\begin{example}
  The tensor product of two copies of the standard $\gl_3$ crystal $\B(3)$ is shown in Fig.~\ref{fig:tensor}. Notice it has two connected components, one of dimension $6$ with highest weight $(2,0,0)$ and the other of dimension $3$ with highest weight $(1,1,0)$.
\end{example}

\begin{definition}
  An abstract $\gl_n$-crystal is \newword{normal} if every full subcrystal is isomorphic to a full subcrystal of $\B(n)^{\otimes k}$ for some positive integer $k$.
\end{definition}

A connected, normal crystal $\B$ has a unique highest weight $b$, and we call $\wt(b)\in\Lambda^{+}$ the \newword{highest weight of $\B$}. Moreover, two connected normal crystals are isomorphic as colored directed graphs if and only if they have the same highest weight. In other words, connected normal crystals are index by dominant weights, which in turn index irreducible representations. Given a dominant weight $\lambda \in \Lambda^{+}$, let $\B(\lambda)$ denote the connected normal crystal whose unique highest weight element has weight $\lambda$.

Even more compelling is the remarkable fact that the following combinatorial procedure on crystals corresponds to the tensor product of the corresponding representations. For example, Fig.~\ref{fig:tensor} computes that the tensor product of two copies of the standard crystal $\B(3)=\B((1,0,0))$ is given by $\B((2,0,0))$ and $\B((1,1,0))$.

The character of a connected normal crystal $\mathrm{ch}(\B(\lambda))$ is the character of the irreducible representation $V^{\lambda}$ which is the \emph{Schur polynomial} $s_{\lambda}(x_1,\ldots,x_{n})$.

Since the character of a crystal is symmetric and the character of a connected, normal crystal is a Schur polynomial, crystals provide a combinatorial method for proving symmetry and Schur positivity of a given polynomial. Moreover, the highest weights provide an efficient formula for the Schur expansion of the character of a normal crystal $\B$ by
\begin{equation}
  \mathrm{ch}(\B) = \sum_{\substack{b \in \B \\ b \ \text{highest weight}}} s_{\wt(b)}(X).
  \label{e:char-hw}
\end{equation}
Moreover, the existence of an explicit crystal structure gives a representation-theoretic interpretation for the corresponding polynomial by providing a natural action on a crystal base whose character is the given polynomial.

\subsection{Demazure crystals}
\label{sec:crystal-demazure}

Given a complex, semi-simple Lie algebra $\mathfrak{g}$ with a Cartan subalgebra $\mathfrak{h}$, we can decompose a representation $V$ of $\mathfrak{g}$ into weight spaces $V = \bigoplus V_{a}$. The \newword{extremal weights} are indexed by the Weyl group $W$, and the corresponding \newword{extremal weight spaces} are all of dimension $1$ with a natural action of $W$ permuting them. Demazure \cite{Dem74a} considered the action of a Borel subalgebra $\mathfrak{b} \supset \mathfrak{h}$ on an extremal weight space, and we call the resulting modules \newword{Demazure modules}. While the irreducible representations $V^{\lambda}$ of $\mathfrak{g}$ are indexed by dominant weights $\lambda$, the corresponding Demazure modules $V^{\lambda}_w$ are index by a pair $(\lambda,w)$ where $\lambda$ is a dominant weight and $w$ is an element of the Weyl group.

\begin{example}
  For $w = \mathrm{id}$ the identity, the Demazure module $V^{\lambda}_{\mathrm{id}}$ is the one-dimensional highest weight space of $V^{\lambda}$. For $w=w_0$ the long element of $W$, the Demazure $V^{\lambda}_{w_0}$ is the full $\mathfrak{g}$ representation $V^{\lambda}$. Thus Demazure modules can be regarded as an interpolation between the highest weight space and the full module.
\end{example}

In the case of $\gl_n$, the Borel is the subalgebra of upper triangular matrices, and the Demazure modules are indexed by pairs $(\lambda,w)$ where $\lambda$ is a partition of length $n$ and $w$ is a permutation. The data $(\lambda,w)$ is equivalent to the weak composition $a = w \cdot \lambda$, since we may recover $\lambda$ as the weakly decreasing rearrangement of $a$ and $w$ as the shortest (in Coxeter length) permutation taking $a$ to $\lambda$. To keep this correspondence bijective, for a given dominant weight $\lambda$, we consider only permutations $w$ for which $w$ acts \newword{faithfully} on $\lambda$, meaning $w$ is the shortest permutation $u$ for which $u \cdot \lambda = w \cdot \lambda$.

Demazure \cite{Dem74} gave a formula for the character of the Demazure module $V^{\lambda}_w$ which, in the case of the general linear group, can be described as follows. For $1 \leq i < n$, let $s_i$ denote the simple transposition that acts on polynomials in $n$ variables by interchanging $x_i$ and $x_{i+1}$. The \newword{divided difference operators} $\partial_i$ and $\pi_i$ for $1\leq  i <n$ act on polynomials in $n$ variables by
\begin{eqnarray}
  \partial_i f(x_1,\ldots,x_n) & = & \frac{f(x_1,\ldots,x_n) - s_i \cdot f(x_1,\ldots,x_n)}{x_i-x_{i+1}}, \\
  \pi_i f(x_1,\ldots,x_n) & = & \partial_i \left( x_i f(x_1,\ldots,x_n) \right).
\end{eqnarray}
It can be shown that the $\partial_i$ and $\pi_i$ satisfy the braid relations for the symmetric group. Therefore For $w\in S_n$, we may define $\partial_w = \partial_{i_1} \partial_{i_2} \cdots \partial_{i_k}$  and $\pi_w = \pi_{i_1} \pi_{i_2} \cdots \pi_{i_k}$ whenever $s_{i_1}s_{i_2} \cdots s_{i_k}$ is a \newword{reduced expression} for $w$, that is, whenever $w = s_{i_1}s_{i_2} \cdots s_{i_k}$ with $k$ minimal. Here $k$ is the \newword{length} of $w$.

\begin{theorem}
  For $\lambda$ a partition of length $n$ and $w$ a permutation of $\mathcal{S}_n$, the character of the Demazure module $V^{\lambda}_w$ is given by
  \begin{equation}
    \mathrm{ch}(V^{\lambda}_w) = \pi_{w} \left( x_1^{\lambda_1} x_2^{\lambda_2} \cdots x_n^{\lambda_n} \right) .
  \end{equation}
  \label{thm:key}
\end{theorem}

For $\gl_n$, the Demazure characters form a basis for the polynomial ring in $n$ variables. Thus it is natural to index them by weak compositions, and we define the \newword{Demazure character} $\key_a(x_1,\ldots,x_n)$ to be
\begin{equation}
  \key_a(x_1,\ldots,x_n) = \mathrm{ch}(V^{\mathrm{sort}(a)}_{w(a)}) ,
\end{equation}
where $\mathrm{sort}(a)$ is the weakly decreasing rearrangement of $a$ and $w(a)$ is the shortest permutation taking $a$ to $\mathrm{sort}(a)$.

\begin{example}
  We may compute the Demazure character $\key_{(1,2,0,2)}$ by taking $\lambda = (2,2,1,0)$ and $w = 2413$, and then computing the character of $V^{(2,2,1,0)}_{2413}$. Taking the reduced expression $s_3 s_1 s_2$ for the permutation $w = 2413$, we have
  \begin{eqnarray*}
    \key_{(1,2,0,2)} = \mathrm{ch}(V^{(2,2,1,0)}_{2413}) & = & \pi_3 \pi_1 \pi_2 \left(x_1^2 x_2^2 x_3\right) \\
    & = & \pi_3 \pi_1 \left(x_1^2 x_2^2 x_3 + x_1^2 x_2 x_3^2 \right) \\
    & = & \pi_3 \left(x_1^2 x_2^2 x_3 + x_1^2 x_2 x_3^2 + x_1 x_2^2 x_3^2 \right) \\
    & = & x_1^2 x_2^2 x_3 + x_1^2 x_2^2 x_4 + x_1^2 x_2 x_3^2 + x_1^2 x_2 x_3 x_4 \\
    & & + x_1^2 x_2 x_4^2 + x_1 x_2^2 x_3^2 + x_1 x_2^2 x_3 x_4 + x_1 x_2^2 x_4^2 .
  \end{eqnarray*}
\end{example}

Taking $w=w_0$, the long element of $\mathcal{S}_n$, the Demazure characters include the Schur polynomials. That is, when the weak composition $a$ is weakly increasing, we have
\begin{eqnarray}
  \key_{(\lambda_n,\lambda_{n-1},\ldots,\lambda_1)}(x_1,\ldots,x_n) & = & s_{\lambda}(x_1,\ldots,x_n) .
  \label{e:sym-key}
\end{eqnarray}
Furthermore, the Schur \emph{functions} can be realized as the \newword{stable limit},
\begin{eqnarray}
  \key_{0^m \times a}(x_1,\ldots,x_m,0,\ldots,0) & = & s_{\mathrm{sort}(a)}(x_1,\ldots,x_m)  \label{e:trunc-key}
 \\
 \lim_{m\rightarrow\infty} \key_{0^m \times a}(x_1,\ldots,x_m,0,\ldots,0) & = & s_{\mathrm{sort}(a)}(x_1,x_2,\ldots),
  \label{e:stable-key}
\end{eqnarray}
where $\mathrm{sort}(a)$ is the partition rearrangement of the weak composition $a$.

\newword{Demazure crystals} are certain truncations of highest weight crystals conjectured by Littelmann \cite{Lit95} and proved by Kashiwara \cite{Kas93} to generalize Demazure characters. Given a subset $X \subseteq \B(\lambda)$, we define operators $\mathfrak{D}_i$ by
\begin{equation}
  \mathfrak{D}_i X = \{ b \in \B(\lambda) \mid e_i^k(b) \in X \mbox{ for some } k \geq 0 \},
  \label{e:D}
\end{equation}
where $e_i$ denotes the raising operator. It can be shown that these operators satisfy the braid relations for the symmetric group, and so we may define
\begin{equation}
  \mathfrak{D}_w = \mathfrak{D}_{i_1} \cdots \mathfrak{D}_{i_k}
  \label{e:Dw}
\end{equation}
for any reduced expression $s_{i_1} \cdots s_{i_k}$ for the permutation $w$.

\begin{definition}
  For $\lambda$ a partition of length $n$ and $w$ a permutation of $\mathcal{S}_n$, the \newword{Demazure crystal} $\B_w(\lambda)$ is given by
  \begin{equation}
    \B_w(\lambda) = \mathfrak{D}_{w} \{ u_\lambda \},
    \label{e:BwL}
  \end{equation}
  where $u_\lambda$ is the highest weight element in $\B(\lambda)$.
  \label{def:BwL}
\end{definition}

\begin{theorem}[\cite{Kas93}]
  The character of the Demazure crystal $\B_w(\lambda)$ is the Demazure character $\key_{w \cdot \lambda}$.
\end{theorem}

Analogous to the case with normal crystals, Demazure crystals provide a combinatorial method for proving that a given polynomial expands nonnegatively into the Demazure character basis. Moreover, the existence of an explicit Demazure crystal structure gives a representation-theoretic interpretation for the corresponding polynomial by providing a natural action on a Demazure crystal base whose character is the given polynomial.

\subsection{Crystals on tableaux}
\label{sec:crystal-tableaux}

There is an explicit combinatorial construction of the crystal graph on tableaux defined by Kashiwara and Nakashima \cite{KN94} and Littlemann \cite{Lit95}.

\begin{definition}
  For $T\in\SSYT_n(\lambda)$ and $1 \leq i < n$, define the \newword{$i$-pairing} of cells of $T$ containing entries $i$ or $i+1$ as follows:
  \begin{itemize}
  \item $i$-pair cells containing $i$ and $i+1$ whenever they appear in the same column,
  \item iteratively $i$-pair an unpaired $i+1$ with an unpaired $i$ to its right whenever all entries $i$ and $i+1$ that lie between are already $i$-paired.
  \end{itemize}
  \label{def:SSYT-pair}
\end{definition}

It is important to note that this pairing rule determines the lengths of the $i$-strings through a given vertex. That is,
\begin{eqnarray}
&&\varepsilon_i(T) =\text{ number of unpaired $i+1$'s in $T$,} \label{eq:Estring}\\
&&\varphi_i(T) = \text{ number of unpaired $i$'s in $T$.}\label{eq:Fstring}
\end{eqnarray}
We define the action of raising (and, equivalently, lowering) as follows.

\begin{definition}
  For $T\in\SSYT_n(\lambda)$ and $1 \leq i < n$, define the \newword{raising operator} $\Ye_i$ on $T$ as follows: if $T$ has no unpaired entries $i+1$, then $\Ye_i(T)=0$; otherwise, change the leftmost unpaired $i+1$ to $i$ leaving all other entries unchanged.
  \label{def:SSYT-raise}
\end{definition}

For example, the full crystal structure for $\B(2,2,1,0)$ on $\SSYT_4(2,2,1)$ is shown in Fig.~\ref{fig:crystal-2210}. Note that the unique highest weight element has weight $(2,2,1,0)$, and the character is the Schur polynomial $s_{(2,2,1)}(x_1,\ldots,x_4)$.

\begin{figure}[ht]
  \begin{center}
    \begin{tikzpicture}[xscale=1.2,yscale=1.4]
      \node at (3,7) (T00) {$\tableau{3 \\ 2 & 2 \\ 1 & 1}$};
      \node at (1,6) (T01) {$\tableau{3 \\ 2 & 3 \\ 1 & 1}$};
      \node at (5,6) (T11) {$\tableau{4 \\ 2 & 2 \\ 1 & 1}$};
      \node at (0,5) (Tb2) {$\tableau{3 \\ 2 & 3 \\ 1 & 2}$};
      \node at (2,5) (T22) {$\tableau{3 \\ 2 & 4 \\ 1 & 1}$};
      \node at (6,5) (T12) {$\tableau{4 \\ 2 & 3 \\ 1 & 1}$};
      \node at (1,4) (T03) {$\tableau{3 \\ 2 & 4 \\ 1 & 2}$};
      \node at (3,4) (T33) {$\tableau{4 \\ 2 & 4 \\ 1 & 1}$};
      \node at (5,4) (T13) {$\tableau{4 \\ 3 & 3 \\ 1 & 1}$};
      \node at (7,4) (Ta3) {$\tableau{4 \\ 2 & 3 \\ 1 & 2}$};
      \node at (0,3) (T04) {$\tableau{3 \\ 2 & 4 \\ 1 & 3}$};
      \node at (2,3) (T24) {$\tableau{4 \\ 2 & 4 \\ 1 & 2}$};
      \node at (4,3) (T34) {$\tableau{4 \\ 3 & 4 \\ 1 & 1}$};
      \node at (6,3) (Ta4) {$\tableau{4 \\ 3 & 3 \\ 1 & 2}$};
      \node at (1,2) (T25) {$\tableau{4 \\ 2 & 4 \\ 1 & 3}$};
      \node at (5,2) (T15) {$\tableau{4 \\ 3 & 4 \\ 1 & 2}$};
      \node at (7,2) (Tb5) {$\tableau{4 \\ 3 & 3 \\ 2 & 2}$};
      \node at (2,1) (T26) {$\tableau{4 \\ 3 & 4 \\ 1 & 3}$};
      \node at (6,1) (T06) {$\tableau{4 \\ 3 & 4 \\ 2 & 2}$};
      \node at (4,0) (T07) {$\tableau{4 \\ 3 & 4 \\ 2 & 3}$};
      \draw[thick,->,blue  ] (T01) -- (Tb2) node[midway,above] {$1$};
      \draw[thick,->,blue  ] (T12) -- (Ta3) node[midway,above] {$1$};
      \draw[thick,->,blue  ] (T22) -- (T03) node[midway,above] {$1$};
      \draw[thick,->,blue  ] (T13) -- (Ta4) node[midway,above] {$1$};
      \draw[thick,->,blue  ] (T33) -- (T24) node[midway,above] {$1$};
      \draw[thick,->,blue  ] (Ta4) -- (Tb5) node[midway,above] {$1$};
      \draw[thick,->,blue  ] (T34) -- (T15) node[midway,above] {$1$};
      \draw[thick,->,blue  ] (T15) -- (T06) node[midway,above] {$1$};
      \draw[thick,->,blue  ] (T26) -- (T07) node[midway,above] {$1$};
      \draw[thick,->,purple] (T00) -- (T01) node[midway,above] {$2$};
      \draw[thick,->,purple] (T11) -- (T12) node[midway,above] {$2$};
      \draw[thick,->,purple] (T12) -- (T13) node[midway,above] {$2$};
      \draw[thick,->,purple] (Ta3) -- (Ta4) node[midway,above] {$2$};
      \draw[thick,->,purple] (T03) -- (T04) node[midway,above] {$2$};
      \draw[thick,->,purple] (T33) -- (T34) node[midway,above] {$2$};
      \draw[thick,->,purple] (T24) -- (T25) node[midway,above] {$2$};
      \draw[thick,->,purple] (T25) -- (T26) node[midway,above] {$2$};
      \draw[thick,->,purple] (T06) -- (T07) node[midway,above] {$2$};
      \draw[thick,->,violet] (T00) -- (T11) node[midway,above] {$3$};
      \draw[thick,->,violet] (T01) -- (T22) node[midway,above] {$3$};
      \draw[thick,->,violet] (Tb2) -- (T03) node[midway,above] {$3$};
      \draw[thick,->,violet] (T22) -- (T33) node[midway,above] {$3$};
      \draw[thick,->,violet] (T03) -- (T24) node[midway,above] {$3$};
      \draw[thick,->,violet] (T13) -- (T34) node[midway,above] {$3$};
      \draw[thick,->,violet] (Ta4) -- (T15) node[midway,above] {$3$};
      \draw[thick,->,violet] (T04) -- (T25) node[midway,above] {$3$};
      \draw[thick,->,violet] (Tb5) -- (T06) node[midway,above] {$3$};
    \end{tikzpicture}
    \caption{\label{fig:crystal-2210}The normal $\gl_4$ crystal with highest weight $(2,2,1,0)$.}
  \end{center}
\end{figure}
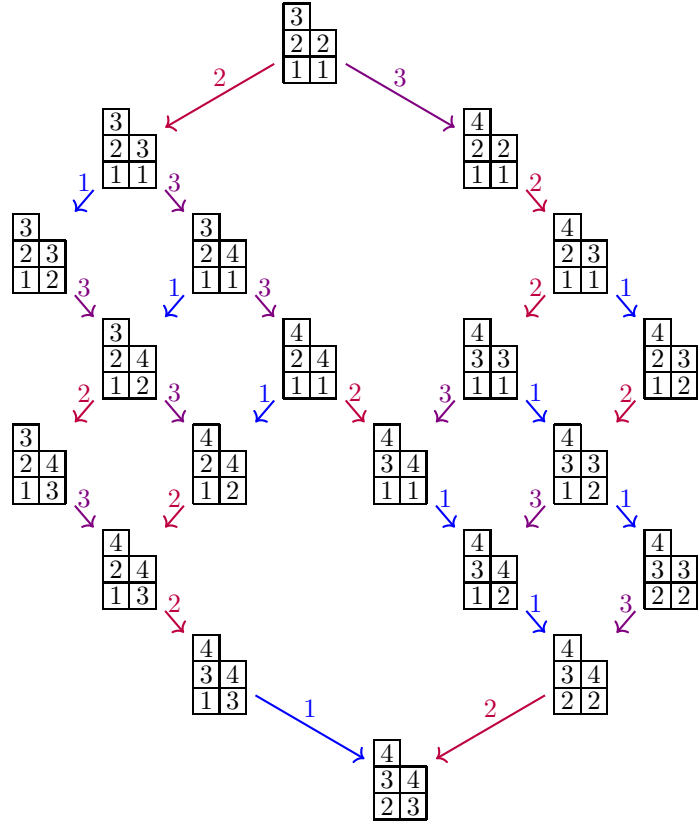

The lowering operators also have an explicit description, inverse to the raising operators.

\begin{definition}
  For $T\in\SSYT_n(\lambda)$ and $1 \leq i < n$, define the \newword{lowering operator} $\Yf_i$ on $T$ as follows: if $T$ has no unpaired entries $i$, then $\Yf_i(T)=0$; otherwise, change the rightmost unpaired $i$ to $i+1$ leaving all other entries unchanged.
  \label{def:SSYT-lower}
\end{definition}

With this explicit description, we can also describe the crystal flip map explicitly.

\begin{example}
  Consider the $\gl_4$-crystal $\B(2,2,1,0)$ shown in Fig.~\ref{fig:crystal-2210}. Under the action of $\mathcal{F}$ we have:
  \[ f_1 \mapsto e_3 \qquad f_2 \mapsto e_2 \qquad f_3 \mapsto e_1\]
  The highest weight element is mapped to the lowest weight element and vice versa as expected, 
  \[\hackcenter{\tableau{3 \\ 2 & 2 \\ 1 & 1}} \mapsto \hackcenter{\tableau{4 \\ 3 & 4 \\ 2 & 3}} \qquad \qquad \hackcenter{\tableau{4 \\ 3 & 4 \\ 2 & 3}} \mapsto \hackcenter{\tableau{3 \\ 2 & 2 \\ 1 & 1}} .\] 
  As further examples, we have
  \begin{eqnarray*}
    \hackcenter{\tableau{3 \\ 2 & 3 \\ 1 & 2}}= f_1f_2 \left( \; \hackcenter{\tableau{3 \\ 2 & 2 \\ 1 & 1}} \; \right) & \mapsto & e_3e_2 \left( \; \hackcenter{\tableau{4 \\ 3 & 4 \\ 2 & 3}} \; \right) = \hackcenter{\tableau{4 \\ 3 & 3 \\ 2 & 2}} \\
    \hackcenter{\tableau{4 \\ 3 & 3 \\ 2 & 2}} = f_1^2f_2^2f_3 \left( \; \hackcenter{\tableau{3 \\ 2 & 2 \\ 1 & 1}} \; \right) & \mapsto & e_3^2e_2^2e_1 \left( \; \hackcenter{\tableau{4 \\ 3 & 4 \\ 2 & 3}} \; \right) = \hackcenter{\tableau{3 \\ 2 & 3 \\ 1 & 2}} 
  \end{eqnarray*}
  At the level of characters, we see that this map sends the monomials $x_1^2x_2^2x_3 \mapsto x_2x_3^2x_4^2$ and $x_1x_2^2x_3^2 \mapsto x_2^2x_3^2x_4$, that is, it conjugates the monomials with the permutation $4321$. 
\end{example}

We may also consider the Demazure crystal $\B_{2413}(2,2,1,0)$, which is the subset shown on the left side of Fig.~\ref{fig:crystal-1202} of the irreducible $\gl_4$ crystal $\B(2,2,1,0)$ show in Fig.~\ref{fig:crystal-2210}. Note that it also has a unique highest weight that has weight $(2,2,1,0)$, but there are multiple lowest weight elements of different weights. Its character is $\key_{(1,2,0,2)}$.

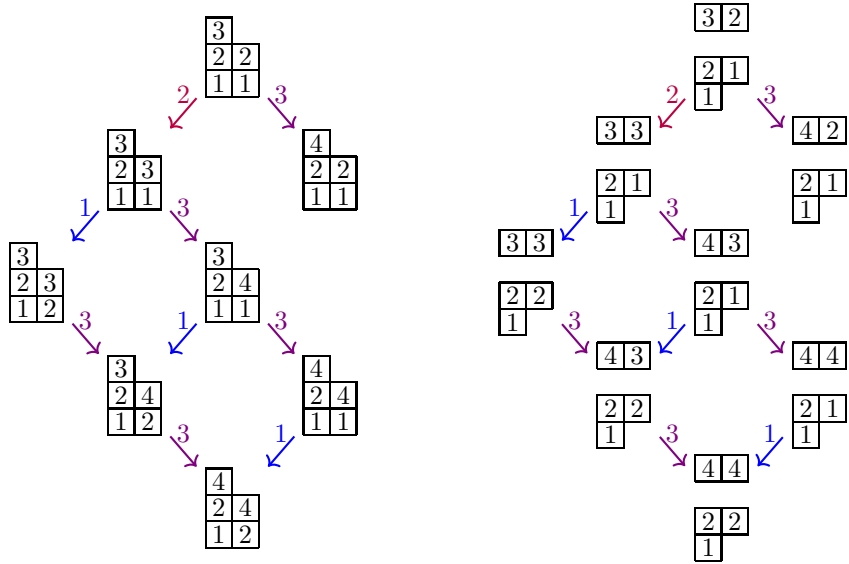
\begin{figure}[ht]
  \begin{center}
    \begin{tikzpicture}[xscale=1.3,yscale=1.5]
      \node at (3,7) (T00) {$\tableau{3 \\ 2 & 2 \\ 1 & 1}$};
      \node at (2,6) (T01) {$\tableau{3 \\ 2 & 3 \\ 1 & 1}$};
      \node at (4,6) (T11) {$\tableau{4 \\ 2 & 2 \\ 1 & 1}$};
      \node at (1,5) (Tb2) {$\tableau{3 \\ 2 & 3 \\ 1 & 2}$};
      \node at (3,5) (T22) {$\tableau{3 \\ 2 & 4 \\ 1 & 1}$};
      \node at (2,4) (T03) {$\tableau{3 \\ 2 & 4 \\ 1 & 2}$};
      \node at (4,4) (T33) {$\tableau{4 \\ 2 & 4 \\ 1 & 1}$};
      \node at (3,3) (T24) {$\tableau{4 \\ 2 & 4 \\ 1 & 2}$};
      \draw[thick,->,blue  ] (T01) -- (Tb2) node[midway,above] {$1$};
      \draw[thick,->,blue  ] (T22) -- (T03) node[midway,above] {$1$};
      \draw[thick,->,blue  ] (T33) -- (T24) node[midway,above] {$1$};
      \draw[thick,->,purple] (T00) -- (T01) node[midway,above] {$2$};
      \draw[thick,->,violet] (T00) -- (T11) node[midway,above] {$3$};
      \draw[thick,->,violet] (T01) -- (T22) node[midway,above] {$3$};
      \draw[thick,->,violet] (Tb2) -- (T03) node[midway,above] {$3$};
      \draw[thick,->,violet] (T22) -- (T33) node[midway,above] {$3$};
      \draw[thick,->,violet] (T03) -- (T24) node[midway,above] {$3$};
      \node at (8,7) (U00) {$\vline\tableau{3 & 2 \\ \\ 2 & 1 \\ 1}$};
      \node at (7,6) (U01) {$\vline\tableau{3 & 3 \\ \\ 2 & 1 \\ 1}$};
      \node at (9,6) (U11) {$\vline\tableau{4 & 2 \\ \\ 2 & 1 \\ 1}$};
      \node at (6,5) (Ub2) {$\vline\tableau{3 & 3 \\ \\ 2 & 2 \\ 1}$};
      \node at (8,5) (U22) {$\vline\tableau{4 & 3 \\ \\ 2 & 1 \\ 1}$};
      \node at (7,4) (U03) {$\vline\tableau{4 & 3 \\ \\ 2 & 2 \\ 1}$};
      \node at (9,4) (U33) {$\vline\tableau{4 & 4 \\ \\ 2 & 1 \\ 1}$};
      \node at (8,3) (U24) {$\vline\tableau{4 & 4 \\ \\ 2 & 2 \\ 1}$};
      \draw[thick,->,blue  ] (U01) -- (Ub2) node[midway,above] {$1$};
      \draw[thick,->,blue  ] (U22) -- (U03) node[midway,above] {$1$};
      \draw[thick,->,blue  ] (U33) -- (U24) node[midway,above] {$1$};
      \draw[thick,->,purple] (U00) -- (U01) node[midway,above] {$2$};
      \draw[thick,->,violet] (U00) -- (U11) node[midway,above] {$3$};
      \draw[thick,->,violet] (U01) -- (U22) node[midway,above] {$3$};
      \draw[thick,->,violet] (Ub2) -- (U03) node[midway,above] {$3$};
      \draw[thick,->,violet] (U22) -- (U33) node[midway,above] {$3$};
      \draw[thick,->,violet] (U03) -- (U24) node[midway,above] {$3$};
    \end{tikzpicture}
    \caption{\label{fig:crystal-1202}The Demazure $\gl_4$ crystal indexed highest weight $(2,2,1,0)$ and permutation $2413$ realized on $\SSYT_4(2,2,1)$ (left) and on $\SSKT(1,2,0,2)$ (right).}
  \end{center}
\end{figure}

Assaf and Schilling \cite{ASc18} defined an explicit Demazure crystal structure on semistandard key \emph{tableaux} \cite{Ass18}, the objects that correspond to Mason's semi-skyline augmented fillings \cite{Mas09}. As semistandard key tableaux are precisely the semistandard key tabloids with $\maj = 0$ \cite{Ass18}(Proposition~3.1). Their definitions will come as a special case of the more general structure we define on semistandard key tabloids, so we defer the details to Proposition~\ref{prop:ASc}. However, for comparison, the right side of Fig.~\ref{fig:crystal-1202} shows the Demazure crystal structure directly on semistandard key tabloids of shape $a=(1,2,0,2)$.

%
\section{Characterizations of Demazure crystals}
%
\label{sec:demazure}

One can prove that a given colored, directed graph with weighted vertices is the crystal of a $\gl_n$-representation by finding a weight-preserving bijection with semistandard Young tableaux that intertwines the crystal operators. To circumvent this difficulty of finding an explicit bijection, Stembridge \cite{Ste03} gave a local characterization of normal crystals for simply-laced types that allows one to determine directly if a given colored, directed graph is the crystal for some representation. Parallel to this, one can prove that a given subset of a normal crystal is a Demazure crystal by finding a weight-preserving injection into semistandard Young tableaux, or bijection to key tableaux, that intertwines the crystal operators. In this section, we present an alternative local characterization of Demazure subsets of normal crystals.

In \S\ref{sec:demazure-extremal}, we define \newword{extremal} subsets of normal crystals, which contain Demazure subsets as a special case. These extremal subsets are easy to find and characterize. In \S\ref{sec:demazure-local}, we extend the axioms for extremal subsets to a local characterization for Demazure subsets, giving a powerful tool for proving that a given structure is a Demazure crystal. In \S\ref{sec:demazure-lowest}, we characterize the \newword{Demazure lowest weight elements} for Demazure crystals. These important elements play a role analogous to highest weight elements for normal crystals in that they are the unique elements that encode the character of the crystal in their weights.

\subsection{Extremal subsets of crystals}
\label{sec:demazure-extremal}

Given a connected, normal crystal $\B(\lambda)$, recall that a weight vector is an \newword{extremal weight} if it is of the form $w \cdot \lambda$ for some permutation $w$. Similarly, we say that an element $b \in \B(\lambda)$ is \newword{extremal} if its weight $\wt(b)$ is an extremal weight. We begin by noting that extremal weight elements sit at the ends of their $i$-strings.

Due to their recurring appearance in the following section, recall from Definition~\ref{def:base} that $\varepsilon_i(b) = \max\{k \in \mathbb{Z}_{\geq 0} \mid e_i^k(b) \neq 0 \}$ and $\varphi_i(b) = \max\{k \in \mathbb{Z}_{\geq 0} \mid f_i^k(b) \neq 0 \}$. Henceforth, for any $x \in \B(\lambda)$ let $f_i^*(x):=f_i^{\varphi_i(x)}(x)$ and $e_i^*(x):= e_i^{\varepsilon_i(x)}(x)$ whenever $\varphi_i(x)$ and $\varepsilon_i(x)$ are nonzero, respectively.

\begin{proposition}
If $b\in\B(\lambda)$ is extremal, then for each $i$, either $\varphi_i(b)=0$ or $\varepsilon_i(b)=0$.
  \label{prop:extremal}
\end{proposition}

\begin{proof}
Suppose $wt(b)=s_i\cdot \lambda$. Then, from Definition~\ref{def:base}(2) it follows that if $f_i(u_\lambda) \neq 0$ then $\varphi_i(f_i(u_\lambda))=\varphi_i(u_\lambda)-1$.  Since $e_i(u_\lambda)=0$ for all $i$ then $f_i^{\varphi_i(u_\lambda)}(u_\lambda)=b$ and thus $\varepsilon_i(b)=\varphi_i(u)>0$ and $\varphi_i(b)=0$. Moreover if $j\neq i$ then $\varepsilon_j(b)=\varepsilon_j(u_\lambda)=0$. Since any permutation $w$ has reduced expressions in terms of simple transpositions and for any $b \in \B(\lambda)$ there exists a minimal path to $\lambda$ such that $e_{i_1}^*\cdots e_{i_n}^*(b)=u_\lambda$, then if $b$ is extremal with $wt(b)=w\cdot \lambda$, the result follows by inducting on $n = l(w)$.
\end{proof}

Given any normal crystal $\B(\lambda)$ and any \emph{subset} $X$ of $\B(\lambda)$, we consider the \emph{induced subgraph} on $X$ that includes all edges $x \overset{e}{\rightarrow} y$ whenever $x,y\in X$. Similarly, we allow $X$ to inherit the weight map from $\B(\lambda)$. We are especially interested in two special classes of subsets.

\begin{definition}
  Given a connected, normal crystal $\B(\lambda)$, a subset $X \subseteq \B(\lambda)$ is \newword{extremal} if
  \begin{enumerate}
  \item $u_{\lambda} \in X$, where $u_{\lambda}$ is the highest weight element of $\B(\lambda)$;
  \item for $x \in X$ and $1 \leq i < n$, if $e_i(x) \neq 0$, then $e_i(x) \in X$;
  \item for $x \in X$ and $1 \leq i < n$, if $f_i(x) \neq 0$ and $f_i(x) \not\in X$, then $e_i(x) \not\in X$.
  \end{enumerate}
  \label{def:extremal}
\end{definition}

Informally, an \newword{extremal subset} contains the highest weight element and contain either all elements of an $i$-string, no elements of an $i$-string, or only the top element of an $i$-string.

Though we will not require the full power of Stembridge's axioms here, in order to study extremal subsets we begin with some technical observations that follow from almost directly from these axioms.

\begin{proposition}\label{obv:sl2}
  Suppose $|i-j|=1$ and $x, z \in \B(\lambda)$ are such that $\varphi_i(x), \varphi_j(x)>0$, $\varepsilon_i(x)$ and $\varepsilon_j(x)$ are zero, and $z=f_i^*f_j^*f_i^*(x)$. Let $u=f_j^rf_i^{s+1}(x)$, $v=e_i^re_j^{s+1}(z)$, $u'=f_i^rf_j^{s+1}(x)$, and $v'=e_j^re_i^{s+1}(z)$. Then,
  \begin{enumerate}
  \item if $r>\varphi_j(x)+1+s$ then $u=v=0$
  \item if $r>\varphi_i(x)+1+s$ then $u'=v'=0$
  \item if $\varphi_i(x)-2\geq s\geq r \geq 0$ then $f_if_j(u)=f_jf_i(u)$, $e_ie_j(v)=e_je_i(v)$, and are nonzero.
  \item if $\varphi_j(x)-2\geq s\geq r \geq 0$ then $f_if_j(u')=f_jf_i(u')$, $e_ie_j(v')=e_je_i(v')$, and are nonzero.
  \end{enumerate}
\end{proposition}

\begin{proof}
  We recall the following definitions from Section 1 of \cite{Ste03}. Given any $x \in \B(\lambda)$ let $\epsilon(x,i):= \varphi_i(x)$ and $\delta(x,i):=\varepsilon_i(x)$ and consider the operators $\Delta_i\delta(x,j)=\delta(e_i(x),j)-\delta(x,j)$, $\Delta_i\epsilon(x,j)=\epsilon(e_i(x),j)-\epsilon(x,j)$ defined whenever $e_i(x) \neq 0$ and also $\nabla_i\delta(x,j)=\delta(x,j)-\delta(f_i(x),j)$ and $\nabla_i\epsilon(x,j)=\epsilon(x,j)-\epsilon(f_ix,j)$ defined whenever $f_i(x) \neq 0$.  Given any $\lambda \vdash n$ and $x \in \B(\lambda)$ by using axioms (P3)-(P6) in \cite{Ste03} for simply-laced crystals, we can directly deduce the following statements.
  \begin{enumerate}\setcounter{enumi}{-1}
  \item If $|i-j|=1$ and $e_i(x)=e_j(x)=0$ but $f_i(x),f_j(x) \neq 0$ then $\nabla_i\epsilon(x,j)=\nabla_j\epsilon(x,i)=-1$.
  \item If $|i-j|=1$ and $e_i(x),f_i(x),f_j(x) \neq 0$ then $\Delta_i\delta(f_j^s(x),j)=0$ whenever $f_j^s(x)\neq 0$. Likewise, if $f_i(x),e_i(x),e_j(x) \neq 0$ then $\nabla_i\epsilon(e_j^s(x),j)=0$ whenever $e_i^s(x) \neq 0$.
  \item If $|i-j|=1$ and $x$ satisfies $\Delta_i\delta(x,j)=\Delta_j\delta(x,i)=-1$ and $f_i(x),f_j(x) \neq 0$ then $\nabla_i\epsilon(x,j)=\nabla_j\epsilon(x,i)=-1$.
  \item If $|i-j|=1$ and $x$ satisfies $\Delta_i\delta(x,j)=\Delta_j\delta(x,i)=-1$ and $f_i(x) \neq 0$ then $\Delta_i\delta(f_jf_i(x),j)=\Delta_j\delta(f_jf_i(x),i)=-1$.
  \end{enumerate}
  Graphically, statement (1) can be envisioned as follows:

  \begin{center}
    $
    \hackcenter
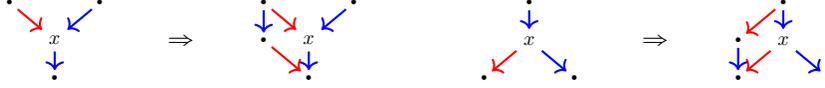
{\begin{tikzpicture}[scale=.5,yscale=1, xscale=1.2,every node/.style={scale=0.5}]
        \begin{scope}
          \node at (-1,-3)(d){$\bullet$};
          \node at (1,-3)(d'){$\bullet$};
          \node at (0,-4)(e')[scale=1.5]{$x$};
          \node at (0,-5)(f){$\bullet$};
          \draw[->,red,thick](d)--(e');
          \draw[->,blue,thick](d')--(e');
          \draw[->,blue,thick](e')--(f);
        \end{scope}
    \end{tikzpicture}}
    \qquad
    \Rightarrow \qquad
    \hackcenter{
      \begin{tikzpicture}[scale=.5,yscale=1, xscale=1.2,every node/.style={scale=0.5}]
        \begin{scope}
          \node at (-1,-3)(d){$\bullet$};
          \node at (1,-3)(d'){$\bullet$};
          \node at (-1,-4)(e){$\bullet$};
          \node at (0,-4)(e')[scale=1.5]{$x$};
          \node at (0,-5)(f){$\bullet$};
          \draw[->,red,thick](d)--(e');
          \draw[->,red,thick](e)--(f);
          \draw[->,blue,thick](d)--(e);
          \draw[->,blue,thick](d')--(e');
          \draw[->,blue,thick](e')--(f);
        \end{scope}
    \end{tikzpicture}}$
    \qquad\qquad
    $
    \hackcenter{\begin{tikzpicture}[scale=.5,yscale=-1, xscale=1.2,every node/.style={scale=0.5}]
        \begin{scope}
          \node at (-1,-3)(d){$\bullet$};
          \node at (1,-3)(d'){$\bullet$};
          \node at (0,-4)(e')[scale=1.5]{$x$};
          \node at (0,-5)(f){$\bullet$};
          \draw[<-,red,thick](d)--(e');
          \draw[<-,blue,thick](d')--(e');
          \draw[<-,blue,thick](e')--(f);
        \end{scope}
    \end{tikzpicture}}
    \qquad
    \Rightarrow \qquad
    \hackcenter{
      \begin{tikzpicture}[scale=.5,yscale=-1, xscale=1.2,every node/.style={scale=0.5}]
        \begin{scope}
          \node at (-1,-3)(d){$\bullet$};
          \node at (1,-3)(d'){$\bullet$};
          \node at (-1,-4)(e){$\bullet$};
          \node at (0,-4)(e')[scale=1.5]{$x$};
          \node at (0,-5)(f){$\bullet$};
          \draw[<-,red,thick](d)--(e');
          \draw[<-,red,thick](e)--(f);
          \draw[<-,blue,thick](d)--(e);
          \draw[<-,blue,thick](d')--(e');
          \draw[<-,blue,thick](e')--(f);
        \end{scope}
    \end{tikzpicture}}$
  \end{center}

  \noindent Whereas statements (2) and (3) say the following:
  \begin{center}
    $
    \hackcenter
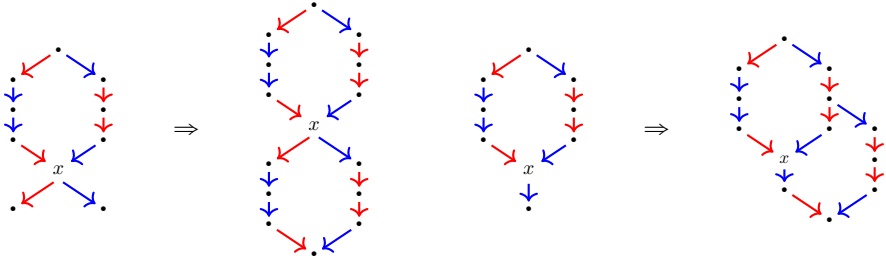
{\begin{tikzpicture}[scale=.5,yscale=.8, xscale=1.2,every node/.style={scale=0.5}]
        \begin{scope}
          \node at (0,0)(a){$\bullet$};
          \node at (-1,-1)(b){$\bullet$};
          \node at (1,-1)(b'){$\bullet$};
          \node at (-1,-2)(c){$\bullet$};
          \node at (1,-2)(c'){$\bullet$};
          \node at (-1,-3)(d){$\bullet$};
          \node at (1,-3)(d'){$\bullet$};
          \node at (0,-4)(e')[scale=1.5]{$x$};
          \draw[->,red,thick](a)--(b);
          \draw[->,red,thick](b')--(c');
          \draw[->,red,thick](c')--(d');
          \draw[->,red,thick](d)--(e');
          \draw[->,blue,thick](a)--(b');
          \draw[->,blue,thick](b)--(c);
          \draw[->,blue,thick](c)--(d);
          \draw[->,blue,thick](d')--(e');
        \end{scope}
        \begin{scope}[shift={(0,-4.3)}]
          \node at (0,0)(a){};
          \node at (-1,-1)(b){$\bullet$};
          \node at (1,-1)(b'){$\bullet$};
          \draw[->,red,thick](a)--(b);
          \draw[->,blue,thick](a)--(b');
        \end{scope}
    \end{tikzpicture}}
    \qquad
    \Rightarrow
    \qquad
    \hackcenter{\begin{tikzpicture}[scale=.5,yscale=.8, xscale=1.2,every node/.style={scale=0.5}]
        \begin{scope}
          \node at (0,0)(a){$\bullet$};
          \node at (-1,-1)(b){$\bullet$};
          \node at (1,-1)(b'){$\bullet$};
          \node at (-1,-2)(c){$\bullet$};
          \node at (1,-2)(c'){$\bullet$};
          \node at (-1,-3)(d){$\bullet$};
          \node at (1,-3)(d'){$\bullet$};
          \node at (0,-4)(e')[scale=1.5]{$x$};
          \draw[->,red,thick](a)--(b);
          \draw[->,red,thick](b')--(c');
          \draw[->,red,thick](c')--(d');
          \draw[->,red,thick](d)--(e');
          \draw[->,blue,thick](a)--(b');
          \draw[->,blue,thick](b)--(c);
          \draw[->,blue,thick](c)--(d);
          \draw[->,blue,thick](d')--(e');
        \end{scope}
        \begin{scope}[shift={(0,-4.3)}]
          \node at (0,0)(a){};
          \node at (-1,-1)(b){$\bullet$};
          \node at (1,-1)(b'){$\bullet$};
          \node at (-1,-2)(c){$\bullet$};
          \node at (1,-2)(c'){$\bullet$};
          \node at (-1,-3)(d){$\bullet$};
          \node at (1,-3)(d'){$\bullet$};
          \node at (0,-4)(e')[scale=1]{$\bullet$};
          \draw[->,red,thick](a)--(b);
          \draw[->,red,thick](b')--(c');
          \draw[->,red,thick](c')--(d');
          \draw[->,red,thick](d)--(e');
          \draw[->,blue,thick](a)--(b');
          \draw[->,blue,thick](b)--(c);
          \draw[->,blue,thick](c)--(d);
          \draw[->,blue,thick](d')--(e');
        \end{scope}
    \end{tikzpicture}}
    \qquad\qquad
    \hackcenter{\begin{tikzpicture}[scale=.5,yscale=.8, xscale=1.2,every node/.style={scale=0.5}]
        \begin{scope}
          \node at (0,0)(a){$\bullet$};
          \node at (-1,-1)(b){$\bullet$};
          \node at (1,-1)(b'){$\bullet$};
          \node at (-1,-2)(c){$\bullet$};
          \node at (1,-2)(c'){$\bullet$};
          \node at (-1,-3)(d){$\bullet$};
          \node at (1,-3)(d'){$\bullet$};
          \node at (0,-4)(e')[scale=1.5]{$x$};
          \draw[->,red,thick](a)--(b);
          \draw[->,red,thick](b')--(c');
          \draw[->,red,thick](c')--(d');
          \draw[->,red,thick](d)--(e');
          \draw[->,blue,thick](a)--(b');
          \draw[->,blue,thick](b)--(c);
          \draw[->,blue,thick](c)--(d);
          \draw[->,blue,thick](d')--(e');
        \end{scope}
        \begin{scope}[shift={(0,-4.3)}]
          \node at (0,0)(a){};
          \node at (0,-1)(b'){$\bullet$};
          \draw[->,blue,thick](a)--(b');
        \end{scope}
    \end{tikzpicture}}
    \qquad
    \Rightarrow
    \qquad
    \hackcenter{
      \begin{tikzpicture}[scale=.5,yscale=.8, xscale=1.2,every node/.style={scale=0.5}]
        \begin{scope}
          \node at (0,0)(a){$\bullet$};
          \node at (-1,-1)(b){$\bullet$};
          \node at (1,-1)(b'){$\bullet$};
          \node at (-1,-2)(c){$\bullet$};
          \node at (1,-2)(c'){$\bullet$};
          \node at (-1,-3)(d){$\bullet$};
          \node at (1,-3)(d'){$\bullet$};
          \node at (2,-3)(d''){$\bullet$};
          \node at (0,-4)(e')[scale=1.3]{$x$};
          \node at (2,-4)(e''){$\bullet$};
          \node at (0,-5)(f){$\bullet$};
          \node at (2,-5)(f'){$\bullet$};
          \node at (1,-6)(g){$\bullet$};
          \draw[->,red,thick](a)--(b);
          \draw[->,red,thick](b')--(c');
          \draw[->,red,thick](c')--(d');
          \draw[->,red,thick](d)--(e');
          \draw[->,red,thick](d'')--(e'');
          \draw[->,red,thick](e'')--(f');
          \draw[->,red,thick](f)--(g);
          \draw[->,blue,thick](a)--(b');
          \draw[->,blue,thick](b)--(c);
          \draw[->,blue,thick](c)--(d);
          \draw[->,blue,thick](d')--(e');
          \draw[->,blue,thick](c')--(d'');
          \draw[->,blue,thick](e')--(f);
          \draw[->,blue,thick](f')--(g);
        \end{scope}
      \end{tikzpicture}
    }$
  \end{center}

  Moreover, if $x \in \B(\lambda)$ satisfies $e_i(x)=e_j(x)=0$ and $\varphi_i(x),\varphi_j(x)>0$ then for $G_x$, the maximal connected subgraph of $\B(\lambda)$ generated by $x$ under the action of $f_i$ and $f_j$, the following also holds:
  \begin{enumerate}\addtocounter{enumi}{3}
  \item If $|i-j|\geq 2$ then $G_x$ has the property that $f_i^rf_j^s(z)=f_j^sf_i^r(z) \in \mathcal{B}(\lambda)$ for any $0\leq s \leq \varphi_j(x)$ and $0\leq r \leq \varphi_i(x)$
  \item If $|i-j|=1$ then $G_x$ has the property that:
    \begin{itemize}
    \item $f_j^{\varphi_i(x)+\varphi_j(x)}f_i^{\varphi_i(x)}(x) \neq f_i^{\varphi_i(x)+\varphi_j(x)}f_j^{\varphi_j(x)}(x)$,
    \item $e_if_j^{\varphi_i(x)+\varphi_j(x)}f_i^{\varphi_i(x)}(x)=0$ and $e_jf_i^{\varphi_i(x)+\varphi_j(x)}f_j^{\varphi_j(x)}(x)=0$,
    \item $f_i^{\varphi_j(x)}f_j^{\varphi_i(x)+\varphi_j(x)}f_i^{\varphi_i(x)}(x) = f_j^{\varphi_i(x)}f_i^{\varphi_i(x)+\varphi_j(x)}f_j^{\varphi_j(x)}(x)$.
    \end{itemize}
  \end{enumerate}

  Combining statements (0)-(5) above, the claims in the proposition can be immediately deduced.
\end{proof}

\begin{proposition}
  For $\lambda$ a partition of length $n$ and $w$ a permutation of $\mathcal{S}_n$, the Demazure crystal $\B_w(\lambda)$ is an extremal subset of $\B(\lambda)$.
  \label{prop:dem-crystal is extremal}
\end{proposition}

\begin{proof}
Suppose $w$ has length one so that $\B_w(\lambda)=  \mathfrak{D}_i(u_\lambda)$ for some $ 1\leq i <n$. Then, by definition, all conditions will hold. We proceed by induction on the length of $w$.

 Suppose $\B_{\nu(\lambda)}$ is extremal for any $\nu$ of length at most $m-1$.  If $w$ has length $m$ we may write $\B_w(\lambda)=\mathfrak{D}_j\mathfrak{D}_{\nu}(u_\lambda)$ where $\nu$ is a permutation of length $m-1$ and thus $\mathfrak{D}_{\nu}(u_\lambda)=\mathfrak{D}_{i_1}\mathfrak{D}_{i_2} \dots \mathfrak{D}_{i_{m-1}}(u_\lambda)$ for some reduced expression $s_{1_i} \dots s_{i_m}$ of $\nu$ is an extremal subset of $\B(\lambda)$.

 It is obvious that $u_\lambda \in \B_w(\lambda)$. Suppose $x \in \B_w(\lambda)$ and $e_i(x) \neq 0$. If $i=j$ then by definition it follows that $e_i(x) \in \B_w(\lambda)$. If $|i-j|\geq 2$ then since $e_j^*(x) \in \mathfrak{D}_{\nu}(u_\lambda)$, then by axiom (P5) in \cite{Ste03} $e_i(e_j^*(x)) \neq 0$ and since $\mathfrak{D}_{\nu}(u_\lambda)$ is extremal, then $e_i(e_j^*(x)) \in \mathfrak{D}_{\nu}(u_\lambda)$. However, since $e_i(e_j^*(x))=e_j^*(e_i(x))$ then $e_i(x) \in \mathfrak{D}_{\nu}(u_\lambda)$. If $|i-j|=1$, by axiom (P6) in \cite{Ste03} we have that $e_i(e_j^*(x)) \neq 0$ and thus by the induction hypothesis either $e_i(e_j^*(x))=e_j^*(e_i(x)) \in \mathfrak{D}_{\nu}(u_\lambda)$ or  $e_j^*e_i^{*}e_j^*(x) = e_i^*e_j^{*}e_i^*(x)\in \mathfrak{D}_{\nu}(u_\lambda)$. In the first case, it is clear that $e_i(x) \in \mathfrak{D}_j(\mathfrak{D}_{\nu}(u_\lambda))$. In the second case, this implies there is some reduced expression of $\nu$ satisfying $i_1=i$ and $i_2=j$ so that $\B_w(\lambda) = \mathfrak{D}_j\mathfrak{D}_i\mathfrak{D}_j\mathfrak{D}_{i_3}\dots\mathfrak{D}_{i_{m-1}}(u_\lambda)$. In particular, there is $ y \in \mathfrak{D}_{\nu}(u_\lambda)$ such that $e_i^*(y) =e_j^*e_i^{*}e_j^*(x)$, and since $e_j^{*}e_i^*(x) = y$ then $e_i^*(x) \in \mathfrak{D}_j(\mathfrak{D}_{\nu}(u_\lambda))$.

 Now suppose $x \in \B_w(\lambda)$, $e_i(x), f_i(x) \neq 0$ and $e_i(x) \in \B_w(\lambda)$. Clearly, if $i=j$ then $f_i(x) \in \B_w(\lambda)$. If $|i-j|\geq 2$ then since $e_j^*e_i(x)$ and $e_j^*(x) \in \mathfrak{D}_{\nu}(u_\lambda)$, it follows by axiom (P5) in \cite{Ste03} and the induction hypothesis that $e_j^*f_i(x)=f_i(e_j^*(x)) \neq 0$ and $f_i(e_j^*(x)) \in \mathfrak{D}_{\nu}(u_\lambda)$. Thus, $f_i(x) \in \B_w(\lambda)$.
The remaining case, when $|i-j|=1$, follows from axiom (P6) in \cite{Ste03} and the induction hypothesis by similar arguments to the ones above.
\end{proof}


Recall that highest weights uniquely characterize normal $\gl_n$ crystals. In contrast, for a fixed partition $\lambda$ of length at most $n$, every Demazure subcrystal $\B_w(\lambda)$ has the same highest weight, namely $\lambda$. Thus, we need additional information to begin to classify Demazure crystals.

\begin{definition}
  Given a subcrystal $X \subseteq \B$ of a normal $\gl_n$ crystal, an element $x \in X$ is a \newword{lowest weight element} of $X$ if  for every $i=1,2,\ldots,n-1$ either $f_i(x) = 0$ or $f_i(x) \not\in X$.
  \label{def:lw}
\end{definition}

The full crystal $\B(\lambda)$ has a unique lowest weight element, which has weight the reverse of $\lambda$. In contrast, a Demazure subcrystal $\B_w(\lambda)$ can, in general, have multiple lowest weight elements. For example, the Demazure crystal $\B_{2413}(2,2,1,0)$ shown in Fig.~\ref{fig:crystal-1202} has two lowest weight elements.

Define the \newword{dominance order} on weak compositions of length $n$ by $a \leq b$ if and only if for every $k=1,\ldots,n$, we have
\begin{equation}
  a_1 + a_2 + \cdots + a_k \leq b_1 + b_2 + \cdots + b_k .
\end{equation}
Using this, we justify our nomenclature for extremal subsets with the following.

\begin{proposition}
  Any extremal subset $X \subseteq \B(\lambda)$ is connected. Moreover, if $x \in X$ is a lowest weight element then either $x$ is extremal or there exists an extremal lowest weight element $z \in X$ such that $\wt(x)>wt(z)$ in dominance order on weak compositions.
  \label{prop:extremal-lw}
\end{proposition}

\begin{proof}
  Suppose $X \subset \B(\lambda)$ is extremal. Then for any $x \in X$ such that $x \neq u_\lambda$ there exists $1\leq i <n$ for which $e_i(x) \neq 0$. Since $X$ is extremal then $e_i^N(x) \in X$ for all $N>0$ for which $e_i^N(x) \neq 0$. Thus, there is some sequence $i_1,\dots,i_k$ such that $e_{i_1}^{*} e_{i_2}^{*} \dots e_{i_k}^{*}(x)=u_\lambda$. Since this holds for any $x \in X$, then $X$ is connected.

   Suppose, in addition, that $x$ is a lowest weight element of $X$ and thus, by Proposition~\ref{prop:extremal}, $\varphi_i(x)=0$ for some $i$.  If $x$ is not extremal then for any reduced path $e_{i_1}^{*} e_{i_2}^{*} \dots e_{i_k}^{*}(x)=u_\lambda$ there exists $1<t\leq k$ for which $f_{i_{t-1}} e^*_{i_t} \dots e^*_{i_k}(x)\neq 0$. By combining axioms (P5) and (P6) in \cite{Ste03}, one can see that there is a $y \in \B(\lambda)$ such that $y = f^*_{i_k}\dots f^*_{i_t}f_{i_{t-1}} e^*_{i_t} \dots e^*_{i_k}(x)$  with $\wt(x) \geq \wt(y)$. In particular, axiom (P5) of \cite{Ste03} and (3) and (4) of Proposition~\ref{obv:sl2} ensure that $y \in X$. Iterating this procedure for each $y$ will eventually terminate in an extremal lowest weight element $z \in X$ satisfying $\wt(x)\geq \wt(z)$. Thus, every lowest weight element of $X$ is either extremal or is higher in dominance order than some other extremal lowest weight element of $X$.
\end{proof}

\subsection{Local characterizations}
\label{sec:demazure-local}

While the unique highest weight does not uniquely characterize a Demazure crystal and lowest weights themselves are not unique, each Demazure crystal has a unique lowest weight element \emph{at the lowest level} of the crystal. Moreover, this element uniquely determines the Demazure crystal.

\begin{definition}
  Given a subcrystal $X \subseteq \B$ of a normal $\gl_n$ crystal, an element $z \in X$ is a \newword{Demazure lowest weight element} of $X$ if it is a lowest weight element and for every other lowest weight element $y \in X$, we have $\wt(y) \geq \wt(z)$.
  \label{def:Dlw}
\end{definition}

The following result follows from the triangularity of Demazure characters with respect to monomials and the fact that dominance order refines lexicographic order.

\begin{proposition}
  The Demazure crystal $\B_w(\lambda)$ has a unique Demazure lowest weight element $z$ with $\wt(z)~=~w \cdot \lambda$.
  \label{prop:Dlw}
\end{proposition}

Recall the \emph{length} of a permutation $w\in\mathcal{S}_n$, denoted by $\ell(w)$, is the minimum number of simple transpositions needed to express $w$. The \newword{weak Bruhat order} on $\mathcal{S}_n$ is defined by $u \preceq v$ whenever $v = s_{i_k} \cdots s_{i_1} u$ and $\ell(v) = k + \ell(u)$. This order translates to containment on Demazure crystals in the following sense.

\begin{proposition}
  Let $u,v$ be permutations that act faithfully on a dominant weight $\lambda$. Then $\B_u(\lambda) \subseteq \B_v(\lambda)$ if and only if $u \preceq v$.
  \label{prop:contain}
\end{proposition}

\begin{proof}
  Reduced expressions for a permutation $w$ are in one-to-one correspondence with maximal chains in the weak order from the identity to $w$. Thus if $u \preceq v$, then any reduced expression $s_{i_k} \cdots s_{i_1}$ for $u$ can be completed to a reduced expression $s_{j_l} \cdots s_{j_1} s_{i_k} \cdots s_{i_1}$ for $v$. The result now follows from Eq.~\eqref{e:D} and Eq.~\eqref{e:BwL}.
\end{proof}

In particular, combining Propositions~\ref{prop:Dlw} and \ref{prop:contain} gives the following.

\begin{corollary}
  If $u \preceq v$ in the weak order on permutations and both act faithfully on a dominant weight $\lambda$, then $u \cdot \lambda \geq v \cdot \lambda$ in dominance order on weak compositions.
\end{corollary}

We now refine our notion of extremal subsets to correspond to Demazure crystals.

\begin{definition}
  Given a connected, normal crystal $\B(\lambda)$, a subset $X \subseteq \B(\lambda)$ is \newword{Demazure} if it is extremal and for any extremal elements $x,y \in X$ the following conditions hold:
  \begin{enumerate}\addtocounter{enumi}{3}
  \item For $|i-j|\geq 2$, if $e_i^*(x)=e_j^*(y) \in X$, then $f_j(x)$ and $f_i(y)$ are nonzero and contained in $X$. Moreover, if $f_k(x)\neq 0$ for some $|k-i|=|k-j|=1$ then $f_k(y) \neq 0$ if and only if $f_kf_j^*(x) \neq 0$.
  \item For $|i-j|=1$,
   \begin{enumerate}
  \item if $e_j^*e_i^*(y)=x$ and $f_i(x) \neq 0$ then $f_i(x) \in X$.
  \item if $e_i^*(x)=e_j^*(y)$ then either $f_i(y)$ or $f_j(x) \in X$. If both $f_i(y)$ and $f_j(x) \in X$ then $f_i^*f_j^{*}(x)=f_j^*f_i^{*}(y) \in X$.
  \end{enumerate}
  \item For $|i-j|=1$, if $e_i^*(x)=e_j^* e_i^{*}(y)$ and $f_{i_n}^{*}\dots f_{i_1}^{*}(x) \in X$ for some path for which no reduced expression $s_{i_1}\dots s_{i_n}$ satisfies $s_{i_1}=j$, then $f_{i_n}^{*}\dots f_{i_1}^{*}(y) \in X$.
  \end{enumerate}
  \label{def:demazure}
\end{definition}

\begin{figure}[ht]
  \begin{center}
    \begin{tikzpicture}[scale=.8]
      \node at (-2,-2) (x) {$x$} ;
      \node at (0,0) (b) {$\bullet$} ;
      \node at (2,-2) (y) {$y$} ;
      \node at (-1,-3) (fx) {$f_{j}(x)$} ;
      \node at (1,-3) (fy) {$f_{i}(y)$} ;
      \draw[thick,color=blue  ,<-] (x) -- (-1.3,-1.3) node[midway,right] [scale=.75] {$f_i$} ;
      \draw[thick,color=blue  ,<-] (-.7,-.7) -- (b) node[midway,right] [scale=.75] {$f_i$} ;
      \draw[thick,color=blue,dotted] (-.7,-.7) -- (-1.3,-1.3) ;
      \draw[thick,color=purple,<-] (y) -- (1.3,-1.3)  node[midway,right] [scale=.75] {$f_j$} ;
      \draw[thick,color=purple,<-] (.7,-.7) -- (b) node[midway,right] [scale=.75] {$f_j$} ;
       \draw[thick,color=purple,dotted] (.7,-.7) -- (1.3,-1.3) ;
      \draw[dotted,thick,color=purple,->] (x) -- (fx) node[midway,left] [scale=.75] {$f_j$} ;
      \draw[dotted,thick,color=blue  ,->] (y) -- (fy) node[midway,right] [scale=.75] {$f_i$} ;
      \begin{scope}[shift={(6,0)}]
      \node at (0,0) (x) {$x$} ;
      \node at (-2,-2) (b) {$\bullet$} ;
      \node at (0,-4) (y) {$y$} ;
      \node at (1,-1) (fx) {$f_{j}(x)$} ;
      \draw[thick,color=purple,->] (x) -- (-.7,-.7) node[midway,right] [scale=.75] {$f_j$} ;
      \draw[thick,color=purple,->] (-1.3,-1.3) -- (b) node[midway,right] [scale=.75] {$f_j$} ;
       \draw[thick,color=purple,dotted] (-.7,-.7) -- (-1.3,-1.3) ;
      \draw[thick,color=blue  ,->] (b) -- (-1.3,-2.7)node[midway,left] [scale=.75] {$f_i$} ;
      \draw[thick,color=blue  ,->] (-.7,-3.3) -- (y) node[midway,left] [scale=.75] {$f_i$} ;
      \draw[dotted,thick,color=blue  ,->] (x) -- (fx) node[near start,right] [scale=.75] {$f_i$} ;
       \draw[thick,color=blue,dotted] (-.7,-3.3) -- (-1.3,-2.7) ;
      \end{scope}
    \end{tikzpicture}
    \caption{\label{fig:demax}An illustration of Demazure axioms (4) and (5b) (left) and (5a) (right).}
  \end{center}
\end{figure}
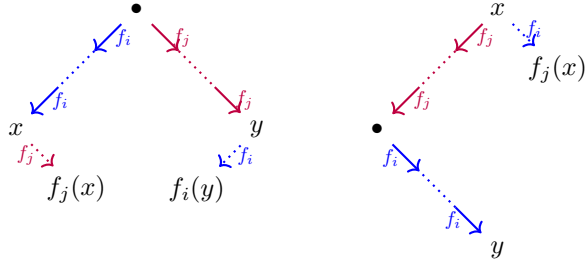

To begin to justify our definition, we have the following.

\begin{theorem}
  Any Demazure crystal $\B_w(\lambda) \subseteq \B(\lambda)$ is a Demazure subset.
  \label{thm:dem-well-def}
\end{theorem}

\begin{proof}
  By Proposition~\ref{prop:dem-crystal is extremal} any Demazure crystal $\B_w(\lambda)$ is an extremal subset of $\B(\lambda)$. Thus, without loss of generality suppose $x,y \in \B_w(\lambda)$ are extremal elements and that $z \in \B_w(\lambda)$ is such that $z = e^*_i(x)=e^*_j(y)$. By (4) and (5) in the proof of Proposition~\ref{obv:sl2} it follows that $z$ is also extremal.

  Condition (4) of Definition~\ref{def:demazure} follows from axioms (P5) and (P6) in \cite{Ste03}. In particular, since $z$ is extremal, then $e_j(x)=e_i(y)=0$. Thus, $\B_w(\lambda)=\mathfrak{D}_{w''}\mathfrak{D}_i\mathfrak{D}_j \mathfrak{D}_w'(\lambda)$ for some permutations $w',w''$ satisfying $w''s_is_jw'=w''s_js_iw'=w$. Since in $\B(\lambda)$, $f^*_i(y)$ and $f^*_j(x) \neq 0$, it immediately follows that $f^*_i(y)=f^*_j(x) \in \B_w(\lambda)$, as desired. The second statement in (4) follows directly from noting that if $|i-k|=|i-j|=1$ and $f_k(x)$ and either $f_k(y)$ or $f_kf_j^*(x)\in \B_w(\lambda)$ then necessarily $s_ks_js_i$ is a subword of $w$.

  Conditions (5a) and (5b) of Definition~\ref{def:demazure} follow from the following observations. If $z$ is an extremal element of $\B_w(\lambda)$ with $e_i(z)=e_j(z)=0$ and $f^*_i(z),f^*_j(z) \neq 0$, then for any reduced expression $s_{i_1}\dots s_{i_n}$ of $w$ there is a sub-expression $s_{i_k}\dots s_{i_n}$ for some $k$ such that $z \in \mathfrak{D}_{i_k}\dots \mathfrak{D}_{i_1}(u_\lambda)$ and $f_{i_a}(z)=0$ for all $1\leq a\leq k$. Suppose $s_{i_c}s_{i_b}s_{i_a}$ is a subword of $s_{i_k}\dots s_{i_n}$ with $k<c<b<a\leq n$ and denote by $G_z$ the subset of $\B_w(\lambda)$ generated by acting on $z$ by $f_i$ and $f_j$. Then one of the following situations must hold.

  \begin{enumerate}
  \item If there exists a subword for which $i_a=i, i_b=j$ and $i_c=i$, then $G_z$ is the subset with highest weight element $z$ and lowest weight element $f^*_if^*_jf^*_i(z)$ and whose edges given by all the $i$-strings and $j$-strings connecting these two vertices.
  \item If for any such subword, $i_a \neq i$ for any $a$ but $i_b=j$ and $i_c=i$ then $G_z$ is the subset with highest weight element $z$ and lowest weight elements $\lbrace f^*_jf_i^s(z)\rbrace_{s\geq0}$ and edges defined by condition (3) in Proposition~\ref{obv:sl2}.
  \item If for any subword, $i_a \neq i$ and $i_b \neq j$ for any $a$ and $b$ but $i_c=i$, then $G_z$ is the full $i$-string $\lbrace f_i^s(z)\rbrace_{s\geq0}$.
  \item If for any subword, $i_a\neq i, i_b\neq j$ and $ i_c\neq i$ for any $a,b,c$ then $G_z$ is the single vertex $\lbrace z\rbrace$ with no edges.
  \end{enumerate}

  Condition (6) in Definition~\ref{def:demazure} follows from the relations of the symmetric group and axioms (P5) and (P6) in \cite{Ste03}. Namely, if $x,y$ are extremal, $e^*_i(x)=e^*_je^*_i(y)$, and $\alpha$ is some path with reduced expression $s_{i_1} \dots s_{i_n}$ such that $f^*_{i_n}\dots f^*_{i_1}(x) \in \B_w(\lambda)$, it follows that, as before, $\B_w(\lambda)=\mathfrak{D}_{w''}\mathfrak{D}_i\mathfrak{D}_j \mathfrak{D}_w'(\lambda)$ where $x$ and $y$ are lowest weight elements of $\mathfrak{D}_i\mathfrak{D}_j \mathfrak{D}_w'(\lambda)$ and $\alpha$ is a sub-expression of $w''$. If no reduced expression for $\alpha$ satisfies $s_{i_1} \neq j$, then $(\alpha \cdot s_i)s_j \neq s_k (\alpha \cdot s_i)$ for any $k$ and thus there exists no paths in $\B_w(\lambda)$ satisfying $f^*_kf^*_{i_n}\dots f^*_{i_1}(x) = f^*_{i_n}\dots f^*_{i_1}(y)$. Moreover, by iterated applications of axioms (P5) and (P6) and keeping track of the weights after each application of the lowering operators, one can see that $f^*_{i_n}\dots f^*_{i_1}(y) \neq 0$. Combining this with the fact that $\alpha$ is a sub-expression of $w''$ implies that $f^*_{i_n}\dots f^*_{i_1}(y) \in \B_w(\lambda)$ as desired.
\end{proof}

In order to prove the converse of Theorem~\ref{thm:dem-well-def}, we begin by noting that every Demazure subset has a unique lowest weight at the lowest level.

In anticipation of the following proof, we recall that a crystal can be regarded as a partially ordered set with $a \preceq b$ if there exists a sequence of lowering operators $f_{i_1},\ldots,f_{i_k}$ such that $a = f_{i_1}\cdots f_{i_k}(b)$. Regarded as such, a connected $\mathfrak{gl}_n$-crystal is a \newword{lattice}, meaning each pair of elements $a,b$ has a unique \newword{join} (least upper bound), denoted by $a \vee b$, and a unique \newword{meet} (greatest lower bound), denoted by $a \wedge b$.

\begin{lemma}\label{lem:Dem-UniqueLowestWeight}
  For any Demazure subset $X \subseteq \B(\lambda)$ there exists a unique global lowest weight element $Z$ satisfying $\wt(x)>wt(Z)$ in dominance order for any other lowest weight element $x \in X$. In particular, $Z$ is extremal and hence if $x$ is also extremal then $wt(x) \prec wt(Z)$ in Bruhat order.
\end{lemma}

\begin{proof}
  Consider the set of lowest weight elements of $X$. By Proposition~\ref{prop:extremal-lw} it suffices to consider only those weights which are extremal. Suppose $X$ does not have a unique global lowest weight element. Since extremal weights form a poset under Bruhat order, then there must exist extremal lowest weights $x,y \in X$ for which $\wt(b) \preceq \wt(x)=\wt(y)$ for any other extremal element $b \in X$. Consider $\wt(x)\wedge \wt(y)$ in the Bruhat graph of the extremal elements of $X$. A straightforward application of axioms (P5) and (P6) in \cite{Ste03} shows that the element $u \in X$ satisfying $\wt(u)=\wt(x) \wedge \wt(y)$ must also be extremal. Let $s_{i_1}\dots s_{i_n}$ and $s_{j_1}\dots s_{j_n}$ be reduced expressions for the paths from $x$ and $y$ to $u$, so that $e^*_{i_1}\dots e^*_{i_n}(x)=e^*_{j_1}\dots e^*_{j_n}(y)=u$.

  \begin{itemize}
  \item[Case 1:] Assume for any reduced expressions $s_{i_1}\dots s_{i_n}$ and $s_{j_1}\dots s_{j_n}$ the relation $|i_1-j_1|=1$ always holds. Suppose there exists no paths for which $i_2\neq j_1$ and $j_2=i_1$.
    \begin{itemize}
    \item If $|i_1-i_2|\geq 2$ then it follows that $|i_2-j_1|=1$. If $|j_2-j_1| \geq 2$ then $|j_2-i_1|=|j_2-i_2|=1$. Since $|j_1-i_1|=|j_1-i_1|=1$, then $j_1=j_2$ which is clearly impossible since $s_{j_1}\dots s_{j_n}$ is reduced.

    \item If $|i_1-i_2|=1$, since $j_1 \neq i_2$ if $|j_1-j_2|\geq 2$ it follows that $j_2=i_2$. Hence, there exists a reduced expression for the path from $u$ to $y$ satisfying $j'_1=i_2$, which contradicts our assumptions. If $|j_1-j_2|=1$, by condition (5b) from Definition~\ref{def:demazure}, either $f^*_{j_1}f^*_{i_1}(u) \in X$ or $f^*_{i_1}f^*_{j_1}(u) \in X$. In either case, since $|j_1-i_2|=|i_1-j_2|=2$ this again implies there is reduced expression satisfying $j'_1=i_2$.
    \end{itemize}

    Thus, if $|i_1-j_1|=1$ for all possible reduced expressions, then there is at least one such expression for which $j_2=i_1$ or $i_2=j_1$. So then, without loss of generality, suppose $i_2=j_1$.
    \begin{enumerate}
    \item If $j_2 \neq i_2$ for any such path then, by condition (6) in Definition~\ref{def:demazure}, there is a $y' \in X$ satisfying $e^*_{i_1}e^*_{j_1} \dots e^*_{j_n}(y')=u$. However, this implies that $\wt(y)\prec wt(y')$ which contradicts $y$ being a global lowest weight element of $X$ (see Fig.~\ref{fig:Case1.1}).
    
      \begin{figure}[ht]
      \begin{tikzpicture}[scale=.75]
        \node at (0,0)(a) {\textbf{\Large{\textit{u}}}};
        \node at (-1,-1)(b) {$\bullet$};
        \node at (1,-1)(c) {$\bullet$};
        \node at (-2,-2)(d) {$\bullet$};
        \node at (-3,-3)(e) {\textbf{\Large{\textit{x}}}};
        \node at (3,-3)(f) {\textbf{\Large{\textit{y}}}};
        \node at (-2,-4)(g) {\textbf{\Large{\textit{y$'$}}}};
        \draw[thick,->] (a) -- (b) node[midway,above] {$i_1$};
        \draw[thick,->] (a) -- (c) node[midway,above] {$j_1$};
        \draw[thick,->] (b) -- (d) node[midway,above] {$j_1$};
        \draw[thick,->,dashed,thick,red]  (c) -- (f)  node[midway,right]{$\alpha$};
        \draw[thick,->,dotted] (d) -- (e) ;
        \draw[thick,->,dashed, thick,red] (d) -- (g) node[midway,right]{$\alpha$};
      \end{tikzpicture}
      \caption{Case 1.1 for the proof of Lemma~\ref{lem:Dem-UniqueLowestWeight}}\label{fig:Case1.1}
      \end{figure}
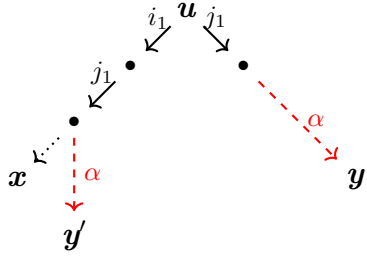
    \item If $j_2 = i_2$ for some path, since $i_2=j_1$, then by condition (5) of Definition~\ref{def:demazure} we must have $f^*_{i_1}f^*_{j_1}f^*_{i_1}(u) = f^*_{j_1}f^*_{i_1}f^*_{j_1}(u) \in X$. If $j_3=j_1$, then $\wt(f^*_{j_1}f^*_{i_1}(u)) = \wt(x)\wedge \wt(y)=\wt(u)$, which is impossible. Clearly, if $i_3=i_1$ an analogous contradiction arises. Thus, $i_3 \neq i_1$ and $j_3 \neq j_1$. So let $u'=f^*_{j_1}f^*_{i_1}(u)$ and $u''=f^*_{i_1}f^*_{j_1}(u)$ (see Fig.~\ref{fig:Case1.2}).
      \begin{enumerate}
      \item Suppose every reduced expression of the paths from $u'$ to $x$ or $u''$ to $y$ satisfies $|i_3-i_1|=1$ or $|j_3-j_1|=1$, respectively.
        \begin{enumerate}
        \item If $f^*_{i_3}f^*_{i_1}(u')$ or $f^*_{j_3}f^*_{j_1}(u'') \in X$, since either $|j_3-i_1|=2$ or $|i_3-j_1|=2$, by (4) in Definition~\ref{def:demazure} it follows that $f^*_{i_3}f^*_{i_1}(u)$ or $f^*_{j_3}f^*_{j_1}(u) \in X$. Moreover, by condition (6) this implies that either $f^*_{i_3}(u'') = f^*_{i_3}f^*_{i_1}f^*_{j_1}(u) \in X$ or $f^*_{j_3}(u') = f^*_{j_3}f^*_{j_1}f^*_{i_1}(u) \in X$ and so, $f^*_{i_3}f^*_{j_1}(u'')= f^*_{i_3}f^*_{i_1}(u')$ or $f^*_{j_3}f^*_{i_1}(u')= f^*_{j_3}f^*_{j_1}(u'')$ are also in $X$. Thus, we may iterate Case 1 with $u'$ or $u''$ in place of $u$.

        \item If either $f^*_{i_3}f^*_{i_1}(u')$ or $f^*_{j_3}f^*_{j_1}(u'') \in X$, then we can iterate Case 1 with $u'$ or $u''$ in place of $u$, respectively.

        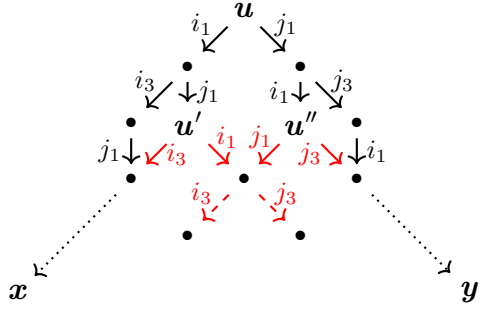
\begin{figure}[ht]
        \begin{center}
        \begin{tikzpicture}[scale=.75]
          \node at (0,0)(a) {\textbf{\Large{\textit{u}}}};
          \node at (-1,-1)(b) {$\bullet$};
          \node at (1,-1)(c) {$\bullet$};
          \node at (-2,-2)(d) {$\bullet$};
          \node at (-1,-2)(e) {\textbf{\Large{\textit{u}$'$}}};
          \node at (1,-2)(f) {\textbf{\Large{\textit{u}$''$}}};
          \node at (2,-2)(g) {$\bullet$};
          \node at (-2,-3)(h) {$\bullet$};
          \node at (0,-3)(i) {$\bullet$};
          \node at (2,-3)(j) {$\bullet$};
          \node at (-1,-4)(k) {$\bullet$};
          \node at (1,-4)(l) {$\bullet$};
          \node at (-4,-5)(m) {\textbf{\Large{\textit{x}}}};
          \node at (4,-5)(n) {\textbf{\Large{\textit{y}}}};
          \draw[thick,->] (a) -- (b) node at (-.75,-.25) {$i_1$};
          \draw[thick,->] (b) -- (e) node[midway,right] {$j_1$};
          \draw[thick,->,red] (e) -- (i) node at (-.3,-2.25) {$i_1$};

          \draw[thick,->] (a) -- (c)  node at (.75,-.25) {$j_1$};
          \draw[thick,->] (c) -- (f) node[midway,left] {$i_1$};
          \draw[thick,->,red] (f) -- (i) node at (.3,-2.25) {$j_1$};

          \draw[thick,->] (b) -- (d) node at (-1.75,-1.25) {$i_3$};
          \draw[thick,->,red] (e) -- (h) node[midway,right] {$i_3$};
          \draw[thick,->] (d) -- (h) node[midway,left] {$j_1$};

          \draw[thick,->] (c) -- (g) node at (1.75,-1.25) {$j_3$};
          \draw[thick,->,red] (f) -- (j) node[midway,left] {$j_3$};
          \draw[thick,->] (g) -- (j) node[midway,right] {$i_1$};

          \draw[thick,->,dashed,red] (i) -- (k) node at (-.75,-3.25) {$i_3$};
          \draw[thick,->,dashed,red] (i) -- (l) node at (.75,-3.25) {$j_3$};

          \draw[thick,->, dotted] (h) -- (m) ;
          \draw[thick,->,dotted] (j) -- (n) ;
        \end{tikzpicture}
        \caption{Case 1.2 for the proof of Lemma~\ref{lem:Dem-UniqueLowestWeight}}\label{fig:Case1.2}
        \end{center}
        \end{figure}
         \end{enumerate}
      \item Suppose there exist reduced expressions for the paths from $u'$ to $x$ or $u''$ to $y$ satisfying $|i_3-i_1|\geq 2$ or $|j_3-j_1|\geq 2$. In this case, we proceed to Case 2 with $u'$ or $u''$ in place of $u$, respectively.
      \end{enumerate}
    \end{enumerate}
  \item[Case 2:] Suppose there exist reduced expressions $s_{i_1}\dots s_{i_n}$ and $s_{j_1}\dots s_{j_n}$ for the paths from $x$ and $y$ to $u$ satisfying the relation $|i_1-j_1|\geq 2$.
    \begin{enumerate}
    \item Suppose $i_k\neq j_1$ for any $k$. If $|i_k-j_1|\geq 2$ then by (4) in Definition~\ref{def:demazure} this would imply that $f_{j_1}(x) \in X$, which contradicts $y$ being a lowest weight. Thus, there must exist some maximal $r$ for which $|i_k-j_1|\geq 2$ for all $1\leq k<r$ but $|i_r-j_1|=1$. Set $u':= f^*_{i_{r-1}}\dots f^*_{i_i}(u)$ (see Fig.~\ref{fig:Case2.1}).
      \begin{enumerate}
      \item If $f^*_{i_r}f^*_{j_1}(u') \in X$ then any path from it cannot terminate in $y$, since this would contradict $u$ satisfying $\wt(u)=\wt(x)\wedge \wt(y)$. Thus, the longest possible path out of $u'$ that passes through $f^*_{i_r}f^*_{j_1}(u')$ must be shorter than the path from $u'$ to $y$ and thus we iterate Case 1 with $u'$ in place of $u$.

      \item If $f^*_{i_r}f^*_{j_1}(u') \not\in X$, then by (5) in Definition~\ref{def:demazure} $f^*_{j_1}f^*_{i_r}(u') \in X$. Since $j_1\neq i_{r+1}$ then by considering $|i_{r+1}-j_1|$ we can iterate Case 1 or Case 2 as needed, with $u'$ in place of $u$. Clearly, if $j_k \neq i_1$ for any $k$ an analogous result holds.

       \begin{figure}[ht]
       \begin{center}
      \begin{tikzpicture}[scale=1]
        \node at (0,0)(a) {\textbf{\Large{\textit{u}}}};
        \node at (-1,-1)(b) {$\bullet$};
        \node at (1,-1)(c) {$\bullet$};
        \node at (-2,-2)(d) {$\bullet$};
        \node at (0,-2)(e) {$\bullet$};
        \node at (-3,-3)(g) {\textbf{\Large{\textit{u$'$}}}};
        \node at (-1,-3)(h) {$\bullet$};
        \node at (-4,-4)(k) {$\bullet$};
         \node at (-4,-5)(p'') {$\bullet$};
        \node at (-2,-4)(l) {$\bullet$};
        \node at (-5,-5)(p) {$\bullet$};
        \node at (-2,-5)(p') {$\bullet$};
        \node at (-6,-6)(x) {\textbf{\Large{\textit{x}}}};
        \node at (3,-6)(y) {\textbf{\Large{\textit{y}}}};

        \draw[thick,->] (a) -- (b) node[midway,above] {$i_1$};
        \draw[thick,->,dotted] (b) -- (d) ;
        \draw[thick,->] (d) -- (g)  node at (-2.8,-2.3) {$i_{r-1}$};
        \draw[thick,->,red] (g) -- (k) node at (-3.8,-3.3) {$i_{r}$};
        \draw[thick,->] (k) -- (p) node at (-4.8,-4.3) {$i_{r+1}$};
        \draw[thick,->,dotted] (p) -- (x) ;
        \draw[thick,->] (c) -- (e) node[midway,above] {$i_1$};
        \draw[thick,->,dotted] (e) -- (h) ;
        \draw[thick,->] (h) -- (l)  node at (-1.8,-3.3) {$i_{r-1}$};
        %
        \draw[thick,->] (a) -- (c) node[midway,above] {$j_1$};
        \draw[thick,->,dotted] (c) -- (y) ;
        \draw[thick,->] (b) -- (e) node[midway,above] {$j_1$};
        \draw[thick,->] (d) -- (h) node[midway,above] {$j_1$};
        %
        \draw[thick,->,red] (g) -- (l) node[midway,above] {$j_1$};
        \draw[thick,->,dotted,red ] (l) -- (p') node[midway,left] {$i_{r}$};
        \draw[thick,->,dotted,red ] (k) -- (p'') node[midway,right] {$j_1$};
      \end{tikzpicture}
        \caption{Case 2.1 for the proof of Lemma~\ref{lem:Dem-UniqueLowestWeight}}\label{fig:Case2.1}
        \end{center}
        \end{figure}
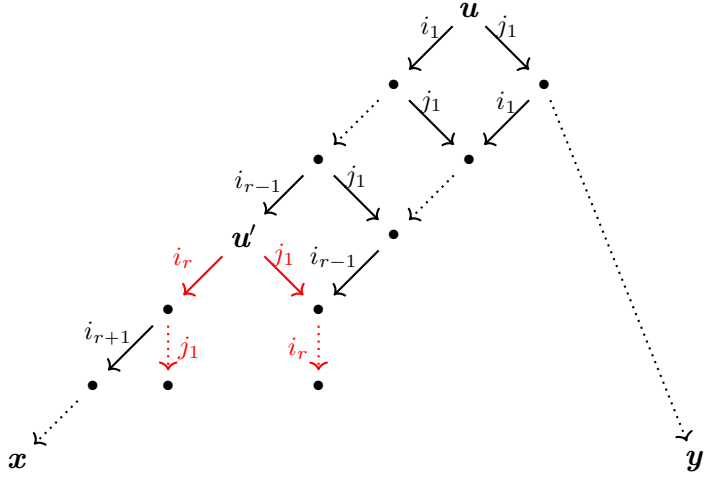
      \end{enumerate}

    \item Now suppose every reduced expression satisfies $i_r = j_1$ for some $r$ and $j_s=i_1$ for some $s$. Then we can choose paths such that $r$ and $s$ are minimized. As an immediate consequence $|j_1-i_{r-1}|=|j_{s-1}-i_1|=1$, else the minimality of $r$ and $s$ would be violated. Moreover, if $k<r$ is the largest index for which $|i_k-i_{k-1}|\geq 2$ then it follows that $|i_{k-1}-i_t|\geq 2$ for all $k<t\leq r$. However, this implies there exists some reduced expression $s_{i'_1}\dots s_{i'_n}$ for the same path for which $i'_{r-1}=j_1$ which contradicts the minimality of $r$. Thus, $i_k=i_{k-1}+1$ and $j_{k'}=j_{k'-1}-1$ (or vice versa) for all $1 \leq k \leq r$ and $1\leq k'\leq s$. Since $i_r=j_1$ and $j_s=i_1$, then without loss of generality if we assume $i_1<j_1$ then $i_1=j_1-r+1$ and $j_s=j_1-s+1$ imply $s=r$ and $i_k=j_1-r+k$ for all $1\leq k\leq r$. In particular, $|j_1-i_k|=|j_1-(j_1-r+k)|=|r-k|\geq 2$ whenever $k\leq r-2$. Likewise, $|j_2-i_k|\geq 2$ for $k\leq r-3$.
    \smallskip
    
    Thus, if set $u':=f^*_{i_{r-2}}\dots f^*_{i_1}(u)$ (see figure below) then by condition (4) of Definition~\ref{def:demazure} it follows that $f^*_{j_{r-1}}f^*_{j_1}(u') \in X$. It is clear the analogous situation holds for $u'':=f^*_{j_{r-2}}\dots f^*_{j_1}(u)$. In particular, we can iteratively apply condition (4) from Definition~\ref{def:demazure} and obtain a sequence of elements  in $X$ that lie higher in Bruhat order than $u'$ and $u''$ (see figure below). Moreover, since $i_{r-1}=j_1-1$ and $j_{r-1}=i_1-1$ then we can proceed to Case 1 by replacing $u$ with $u'$ and $u''$, respectively. (See Fig.~\ref{fig:Case2.2})
    
      \begin{figure}[ht]
      \begin{center}
      \begin{tikzpicture}[scale=1]
        \node at (0,0)(a) {\textbf{\Large{\textit{u}}}};
        \node at (-1,-1)(b) {$\bullet$};
        \node at (1,-1)(c) {$\bullet$};
        \node at (-2,-2)(d) {$\bullet$};
        \node at (0,-2)(e) {$\bullet$};
        \node at (2,-2)(f) {$\bullet$};
        \node at (-3,-3)(g) {\textbf{\Large{\textit{u$'$}}}};
        \node at (-1,-3)(h) {$\bullet$};
        \node at (1,-3)(i) {$\bullet$};
        \node at (3,-3)(j) {\textbf{\Large{\textit{u}$''$}}};
        \node at (-4,-4)(k) {$\bullet$};
        \node at (-2,-4)(l) {$\bullet$};
        \node at (0,-4)(m) {$\bullet$};
        \node at (2,-4)(n) {$\bullet$};
        \node at (4,-4)(o) {$\bullet$};
        \node at (-5,-5)(p) {$\bullet$};
        \node at (-2,-5)(p') {$\bullet$};
        \node at (2,-5)(q') {$\bullet$};
        \node at (5,-5)(q) {$\bullet$};
        \node at (-6,-6)(x) {\textbf{\Large{\textit{x}}}};
        \node at (6,-6)(y) {\textbf{\Large{\textit{y}}}};

        \draw[thick,->] (a) -- (b) node[midway,above] {$i_1$};
        \draw[thick,->,dotted] (b) -- (d) ;
        \draw[thick,->] (d) -- (g)  node at (-2.8,-2.3) {$i_{r-2}$};
        \draw[thick,->,red] (g) -- (k) node at (-3.8,-3.3) {$i_{r-1}$};
        \draw[thick,->] (k) -- (p) node[midway,above] {$i_r$};
        \draw[thick,->,dotted] (p) -- (x) ;
        \draw[thick,->] (c) -- (e) node[midway,above] {$i_1$};
        \draw[thick,->,dotted] (e) -- (h) ;
        \draw[thick,->] (h) -- (l)  node at (-1.8,-3.3) {$i_{r-2}$};
        \draw[thick,->] (f) -- (i) node[midway,above] {$i_1$};
        \draw[thick,->,dotted] (i) -- (m);
        \draw[thick,->,red] (j) -- (n) node[midway,above] {$i_1$};
        \draw[thick,->] (a) -- (c) node[midway,above] {$j_1$};
        \draw[thick,->,dotted] (c) -- (f) ;
        \draw[thick,->] (f) -- (j) node at (2.8,-2.3) {$j_{r-2}$};
        \draw[thick,->,red] (j) -- (o) node at (3.8,-3.3) {$j_{r-1}$};
        \draw[thick,->] (o) -- (q) node[midway,above] {$j_r$};
        \draw[thick,->,dotted] (q) -- (y) ;
        \draw[thick,->] (b) -- (e) node[midway,above] {$j_1$};
        \draw[thick,->,dotted] (e) -- (i) ;
        \draw[thick,->] (i) -- (n) node at (1.8,-3.3) {$j_{r-2}$};
        \draw[thick,->] (d) -- (h) node[midway,above] {$j_1$};
        \draw[thick,->,dotted] (h) -- (m) ;
        \draw[thick,->,red] (g) -- (l) node[midway,above] {$j_1$};
        \draw[thick,->,red ] (l) -- (p') node[midway,left] {$i_{r-1}$};
        \draw[thick,->, red] (n) -- (q') node[midway,right] {$j_{r-1}$};
      \end{tikzpicture}
        \caption{Case 2.2 for the proof of Lemma~\ref{lem:Dem-UniqueLowestWeight}}\label{fig:Case2.2}
        \end{center}
        \end{figure}
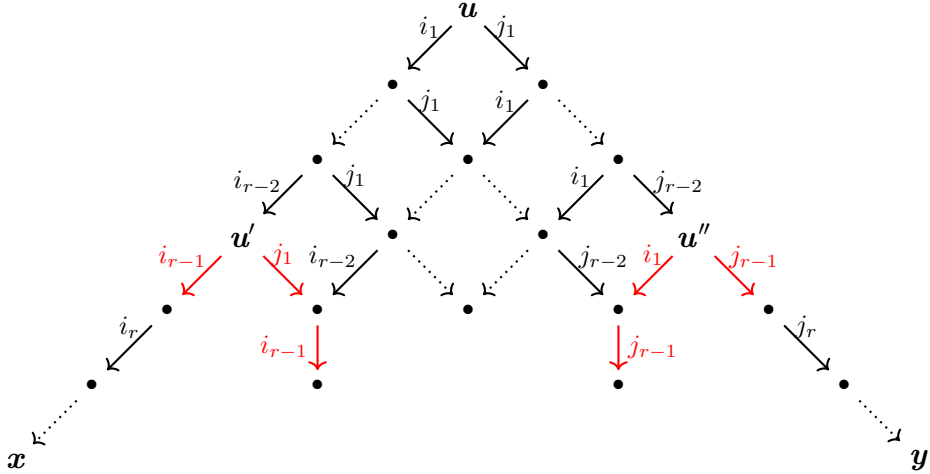
    \end{enumerate}
  \end{itemize}

  Thus, by iterative applications of Case 1 and 2 we can eventually find an element $u'$ on the path between $u$ and $x$ or $u$ and $y$ which satisfies Case 1.1 and yields the desired contradiction. That is, if $x,y \in X$ are extremal lowest weights with $\wt(x)=\wt(y)$, then there is a $z \in X$ also extremal satisfying $\wt(x)\prec wt(z)$ in Bruhat order. Since $X$ is finite, then this implies $X$ has a unique global lowest weight element.
\end{proof}

Finally, we have our main result of this section.

\begin{theorem}
  For any Demazure subset $X \subseteq \B(\lambda)$ there exists $w \in S_n$ such that $X= \B_w(\lambda)$. That is, any Demazure subset of a normal crystal is a Demazure crystal.
  \label{thm:main-dem}
\end{theorem}

\begin{proof}
By Lemma~\ref{lem:Dem-UniqueLowestWeight} $X$ contains a unique global lowest weight satisfying $wt(Z)<wt(b)$ in dominance order for any $b \in X$ and $wt(x) \prec wt(Z)$ in Bruhat order for any other extremal lowest weight $x \in X$. Hence, if $y \in X$ is any extremal weight with $wt(y) = \sigma \cdot \lambda$ then since $wt(y) \prec wt(Z)$ in Bruhat order, then if $wt(Z) = w \cdot \lambda$ it follows that $\sigma \prec w$. Since by Definition~\ref{def:demazure} $X$ is an extremal subset and thus closed under the action of $e_i$, it follows that $X = \B_w(\lambda)$, the Demazure crystal with highest weight $u_\lambda$ and Demazure lowest weight $Z$.
\end{proof}

\subsection{Demazure lowest weights}
\label{sec:demazure-lowest}

Highest weight elements are extremely powerful since they uniquely characterize normal crystals, immediately determine the character, and are easily found from a vertex of the crystal by applying any sequence of raising operators. In contrast, Demazure lowest weight elements satisfy the first two conditions but lack the essential property of being easy to find by applying arbitrary sequences of lowering operators.

In order to find the Demazure lowest weights of a Demazure crystal algorithmically, we consider certain sequences of lowering operators that may be applied to elements of extremal subsets.

\begin{definition}
  Let $X\subseteq\B(\lambda)$ be an extremal subset of a normal crystal. For $i \leq j$ and $b \in X$, define an operator $F_{[i,j]}$ on $X$ by
  \begin{equation}
    F_{[i,j]} (b) = f_{i}^{r_i} f_{i+1}^{r_{i+1}} \cdots f_{j}^{r_j} (b)
    \label{e:fardown}
  \end{equation}
  where $r_k = \varphi_k(F_{[k+1,j]}(b))$ if $f_k(F_{[k+1,j]}(b))\in X$ and otherwise $r_k=0$. We say that $F_{[i,j]}$ acts \newword{faithfully} on $b$ if $r_k>0$ for each $i \leq k \leq j$.
\label{def:fardown}
\end{definition}

In other words, $F_{[i,j]}$ applies lowering operators sequentially from $f_j$ down to $f_i$ with each applied as many times as possible without annihilating the element or leaving $X$, and this is \emph{faithful} if it can be done with each $f_k^{r_k}$ acting nontrivially.

We use these composite lowering operators to find the Demazure lowest weight element from the highest weight element by the following algorithm.

\begin{definition}
  Given a connected Demazure crystal $\B$, define an element $Z\in\B$ by the following procedure. Set $X_0$ to be the (unique) highest weight element of $\mathcal{B}$, and for $k>0$, do the following
  \begin{enumerate}
  \item if $X_{k-1}$ is a lowest weight, then set $Z=X_{k-1}$ and stop;
  \item otherwise, set $X_k = F_{[i_k,j_k]}(X_{k-1})$ where
    \begin{enumerate}
    \item $i_k$ is minimal among all $i$ for which there exists $j \geq i$ such that $F_{[i,j]}$ acts faithfully on $X_{k-1}$, and
    \item $j_k$ is maximal among all $j\geq i_k$ for which $F_{[i_k,j]}$ acts faithfully on $X_{k-1}$.
    \end{enumerate}
  \end{enumerate}
  \label{def:dem-low}
\end{definition}

In \cite{Ass-R}(Proposition~2.4), Assaf shows that every permutation $w$ has a unique reduced word $\pi$ characterized by the properties that, when writing $\pi=(\pi^{(k)}|\cdots|\pi^{(1)})$ such that each subword $\pi^{(i)}$ is an increasing subsequences of maximal length,
\begin{enumerate}
\item each such subsequence $\pi^{(i)}$ is an interval of integers, and
\item the smallest letters in each maximal length increasing subsequence decrease from left to right, i.e. $\pi^{(i)}_1 > \pi^{(i-1)}_1$.
\end{enumerate}
Such a word is called \newword{super-Yamanouchi} \cite{Ass-R}(Definition~2.3). For example, the word
  \[ ( \overbrace{5 , 6, 7}^{\pi^{(4)}} \mid \overbrace{4 , 5 }^{\pi^{(3)}} \mid \overbrace{3, 4, 5, 6}^{\pi^{(2)}} \mid \overbrace{1 , 2, 3}^{\pi^{(1)}}  ) \]
is the super-Yamanouchi reduced word for the permutation $41758236$. 

\begin{lemma}
  Let $\B_w(\lambda)$ be a Demazure crystal, and let $u_{\lambda}$ be the highest weight element. Let $\pi$ be the super-Yamanouchi reduced word for $w$, and write $\pi=(\pi^{(k)}|\cdots|\pi^{(1)})$ for the decomposition of $\pi$ into increasing subsequence of maximal length. Then for all $i\leq k$, $F_{\pi^{(i)}}$ acts faithfully on $F_{\pi^{(i-1)}}\circ\cdots\circ F_{\pi^{(1)}}(u_{\lambda})$.
  \label{lem:faithful}
\end{lemma}

\begin{proof}
  From Proposition~\ref{prop:extremal} it follows that for $X \subset \B(\lambda)$ an extremal subset, if $b\in X$ is an extremal element and $F_{[i,j]}$ acts faithfully on $b$ within $X$, then $F_{[i,j]}(b) \in X$ is an extremal element. The result now follows from Definition~\ref{def:BwL} and Eq.~\eqref{e:Dw} since $\pi$ is a reduced word for $w$.
\end{proof}

In fact, the sequence $F_{\pi^{(k)}}\circ\cdots\circ F_{\pi^{(1)}}(u_{\lambda})$ is precisely the result of Definition~\ref{def:dem-low}. Thus we always have a canonical sequence of lowering operators that will terminate at the Demazure lowest weight.

\begin{theorem}
  For $\B_w$ a connected Demazure crystal, the element $Z\in\B_w$ determined by Definition~\ref{def:dem-low} is the Demazure lowest weight of $\B_w$.
  \label{thm:dem-low}
\end{theorem}

\begin{proof}
  Consider the Demazure crystal $\B_w(\lambda)$. Let $\pi=(\pi^{(k)}|\cdots|\pi^{(1)})$ denote the decomposition of the super-Yamanouchi reduced word for $w$ into increasing subsequences of maximal length. We proceed by induction on $k$. If $k=1$, then $\pi$ is an increasing interval, say $\pi=(i,i+1,\ldots,j)$ for some $i \leq j$. In this case, $f_h(u_{\lambda}) \not\in \B_w(\lambda)$ for any $h\not\in[i,j]$. Therefore the only intervals $[i^{\prime},j^{\prime}]$ for which $F_{[i^{\prime},j^{\prime}]}$ acts faithfully are contained in $[i,j]$, so by Lemma~\ref{lem:faithful}, since $F_{[i,j]}$ acts faithfully on $u_{\lambda}$, the result of Definition~\ref{def:dem-low} is $Z = F_{\pi}(u_{\lambda})$ as desired.

  Now let $k>1$ and assume the result whenever the super-Yamanouchi reduced word for $w$ has fewer than $k$ increasing intervals. Let $v$ be the permutation with reduced word $\pi^{(k-1)}\cdots\pi^{(1)}$. Then $\pi^{(k-1)}\cdots\pi^{(1)}$ is super-Yamanouchi, and $v$ acts faithfully on $\lambda$. Therefore by induction, the element $Z^{\prime} \in \B_v(\lambda)$ constructed by Definition~\ref{def:dem-low} for $\B_v(\lambda)$ is given by $Z^{\prime} = F_{\pi^{(k-1)}}\circ\cdots\circ F_{\pi^{(1)}}(u_{\lambda})$. Since $w = \pi^{(k)} v$, we have $v \prec w$ in weak Bruhat order and so, by Proposition~\ref{prop:contain}, we have $\B_v(\lambda) \subset \B_w(\lambda)$. In particular, $Z^{\prime} \in \B_w(\lambda)$.

  From the characterization of super-Yamanouchi words, it follows that for each $1\leq i < k$, if $\pi_j^{(i)}$ denotes the $j^{th}$ entry of $\pi^{(i)}$, then
  \begin{equation}
    \pi^{(i)}_1 < \min_{1\leq j} \left\lbrace\pi^{(s)}_j \; | \; i+1\leq s\leq k \right\rbrace.
    \label{e:super-Y}
  \end{equation}
  Therefore $Z^{\prime}$ is the element constructed by the first $k-1$ iterations of Definition~\ref{def:dem-low}(2) since each $\pi^{(i)}_1$ is minimal within $\B_w(\lambda)$ as well. By Lemma~\ref{lem:faithful}, $F_{\pi^{(k)}}$ acts faithfully on $Z^{\prime}$, and since $f_h(Z^{\prime})\not\in \B_w(\lambda)$ for $h\not\in\pi^{(k)}$, the final step of Definition~\ref{def:dem-low}(2) will result in $Z = F_{\pi^{(k)}}(Z^{\prime}) = F_{\pi}(u_{\lambda})$ as desired.
\end{proof}

Parallel to the expansion in Eq.~\eqref{e:char-hw}, we have the following tractable character formula.

\begin{corollary}
  For $\B$ any Demazure crystal, we have
  \begin{equation}
    \mathrm{ch}(\B) = \sum_{\substack{u \in \B \\ u \ \text{highest weight}}} \key_{\wt(Z(u))}.
    \label{e:char-dlw}
  \end{equation}
\end{corollary}

%
\section{Demazure crystal on key tabloids}
%
\label{sec:tabloid}

We now apply the tools and techniques of crystal theory to the specialized nonsymmetric Macdonald polynomials. In particular, we will define crystal operators on semistandard key tabloids that generate a Demazure crystal, thereby giving a new combinatorial proof of Theorem~\ref{thm:nskostka} along with a tractable formula for the coefficients that arise in the Demazure expansion of a specialized nonsymmetric Macdonald polynomial.

In \S\ref{sec:tabloid-edges}, we define explicit raising and lowering operators on semistandard key tabloids that are inverse to one another and change the weight in the prescribed way. In \S\ref{sec:tabloid-rect}, we shift our paradigm to Kohnert's diagram model for Demazure characters in order to obtain an injection from semistandard key tabloids to semistandard Young tableaux that intertwines with the crystal operators. Then, in \S\ref{sec:tabloid-lowest}, we use the tools developed in \S\ref{sec:demazure} to prove our operators define a Demazure crystal by showing that their image under the map to semistandard Young tableaux is in fact a Demazure subset.

\subsection{Crystal operators on key tabloids}
\label{sec:tabloid-edges}

Generalizing the crystal constructions on Young tableaux, we give a new proof of Theorem~\ref{thm:nskostka} by constructing an explicit Demazure crystal on semi-standard key tabloids. To begin, we define a pairing rule that will determine the lengths of the $i$-strings.

\begin{definition}
  For $T \in \SSKD(a)$ and $1 \leq i <n$ an integer, define the \newword{i-pairing} of the cells of $T$ with entries $i$ or $i+1$ as follows: $i$-pair together $i$ and $i+1$ whenever they occur in the same column, and then iteratively $i$-pair an unpaired $i+1$ with an unpaired $i$ to its left whenever all entries $i$ or $i+1$ that lie between them are already $i$-paired.
  \label{def:pair}
\end{definition}

\begin{example}
  For $a=(0,0,0,0,0,10,12,0,8,3)$, Fig.~\ref{fig:pairing} shows the $2$-pairing for the given semistandard key tabloid. The paired entries are connected with red lines and the only unpaired $3$ is circled.
  \label{ex:pairing}
\end{example}

\begin{figure}[ht]
  \begin{center}
    \begin{tikzpicture}[scale=.85,every node/.style={scale=0.85}]
      \draw[thick](-3.25,2.25) to (-3.25,-1.25);
      \node at (-1.085,1.7){\begin{ytableau}
          \color{white!30!blue}2&2&2\\
          4&3&3&\color{red}3&\color{white!30!blue}2&\color{white!30!blue}2&2&2\\
      \end{ytableau}};
      \node at (0,0){\begin{ytableau}
          7&7&6&6&5&4&3&3&\color{red}3&3&\color{red}3&\color{violet}3\\
          1&1&1&1&1&1&1&1&1&2\\
      \end{ytableau}};
      \draw[red,thick](-2.45,1.6)--(-2.45,1.85);
      \draw[red,thick](-1.9,1.6)--(-1.9,1.85);
      \draw[red,thick](.25,0.45)--(0.25,1.25);
      \draw[red,thick](0.8,0.45)--(.8,1.25);
      \draw[red,thick](1.9,-.1)--(1.9,.15);
      \draw[red, very thick](-3,2.1)--(-3,2.4)--(-1.4,2.4)--(-1.4,1.6);
      \begin{scope}[shift={(2.75,-.5)}]
        \draw[red, very thick](-3,2.1)--(-3,2.4)--(-1.4,2.4)--(-1.4,.95);
      \end{scope}
      \begin{scope}[shift={(2.2,-.25)}]
        \draw[red, very thick](-3,1.9)--(-3,2.4)--(.2,2.4)--(.2,.7);
      \end{scope}
      \draw (3,0.27) circle (.22cm);
    \end{tikzpicture}
  \end{center}
  \caption{\label{fig:pairing}An example of the $2$-pairing on a semistandard key tabloid.}
\end{figure}
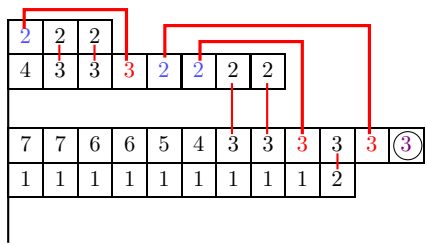

\begin{definition}  \label{def:raise-tabloid}
  Given any $T \in \SSKD(a)$ and an integer $1\leq i <n$ define the \newword{raising operator} $e_i$ on $\SSKD(a)$ whose action on $T$ is as follows:
  \begin{itemize}
  \item if $T$ does not have any cells containing an unpaired $i+1$ then $e_i(T)=0$.
  \item otherwise, $e_i$ changes the rightmost unpaired $i+1$ to $i$ and
    \begin{itemize}
    \item swaps the entries $i$ and $i+1$ in each of the consecutive columns left of this entry that have an $i+1$ in the same row and an $i$ above, and
    \item swaps the entries $i$ and $i+1$ in each of the consecutive columns right of this entry that have an $i+1$ in the same row and an $i$ below.
    \end{itemize}
  \end{itemize}
\end{definition}

\begin{example}\label{ex:raising-oper}
  Let $a = (0,0,4,0,6,2,2)$. Then $e_4$ acts on cells with entries $4$ and $5$ as shown in Fig.~\ref{fig:raise-ex}. At each step, the circled entries contain the unpaired $5$'s and the red highlighted entries are the columns on which $e_4$ will act. As can be seen, each application of $e_4$ decreases the number of unpaired $5$'s by one. When no more remain, then $e_4$ acts by zero.
\end{example}

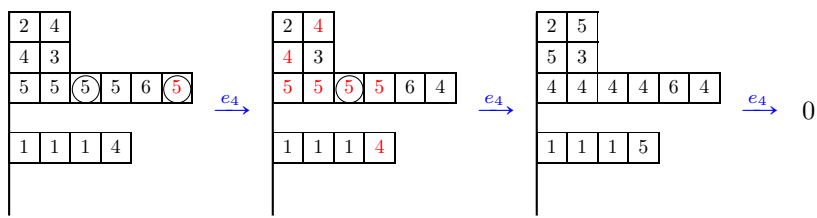
\begin{figure}[ht]
  \begin{center}
    \[
  \hackcenter{\begin{tikzpicture}[scale=1.5,every node/.style={scale=1.5}]
      \draw[thick] (-.54,-1.4)--(-.54,0.4);
      \node at (.27,0)[scale=.5]{$\begin{ytableau}
          2& 4\\
          4&3\\
          5&5&5&5&6&\color{red}5\\
        \end{ytableau}$};
      \node at (0, -.8)[scale=.5]{$\begin{ytableau}
          1&1&1&4\\
        \end{ytableau}$};
      \draw (.14,-.28) circle (1.2mm);
      \draw (.95,-.28) circle (1.2mm);
  \end{tikzpicture}}
  \;\;{\color{blue}\xrightarrow{e_4}}\;\;
  \hackcenter{\begin{tikzpicture}[scale=1.5,every node/.style={scale=1.5}]
      \draw[thick] (-.54,-1.4)--(-.54,.4);
      \node at (.27,0)[scale=.5]{$\begin{ytableau}
          2& \color{red}4\\
          \color{red}4&3\\
          \color{red}5&\color{red}5&\color{red}5&\color{red}5&6&4\\
        \end{ytableau}$};
      \node at (0, -.8)[scale=.5]{$\begin{ytableau}
          1&1&1&\color{red}4\\
        \end{ytableau}$};
      \draw (.14,-.28) circle (1.2mm);
  \end{tikzpicture}}
  \;\;{\color{blue}\xrightarrow{e_4}}\;\;
  \hackcenter{\begin{tikzpicture}[scale=1.5,every node/.style={scale=1.5}]
      \draw[thick] (-.54,-1.4)--(-.54,.4);
      \node at (.27,0)[scale=.5]{$\begin{ytableau}
          2& 5\\
          5&3\\
          4&4&4&4&6&4\\
        \end{ytableau}$};
      \node at (0, -.8)[scale=.5]{$\begin{ytableau}
          1&1&1&5\\
        \end{ytableau}$};
  \end{tikzpicture}}
  \;\;{\color{blue}\xrightarrow{e_4}}\;\;
  \hackcenter{\begin{tikzpicture}
      \node at (7,3){0};
  \end{tikzpicture}}\]
  \end{center}
  \caption{\label{fig:raise-ex}An example of the raising operators applied to a semistandard key tabloid.}
\end{figure}

Note for any $T \in \SSKD$ on which $e_i$ acts non-trivially, the consecutive sequence of columns that will be affected by its action will have entries $i$ and $i+1$ distributed in a specific way, as illustrated in Fig.~\ref{fig:raising-distribution}. The following result collects useful facts about these affected columns.

\begin{proposition}\label{prop:raising-distribution}
Let $T \in \SSKD(a)$ and $i$ an integer such that $e_i(T) \neq 0$. Then the columns of $T$ that are modified non-trivially by the action of $e_i$ will satisfy the following properties:
\begin{itemize}
\item[i)] The column of the rightmost unpaired $i+1$ will have no cells with value equal to $i$.
\item[ii)] For any two consecutive columns both of which contain a cell equal to $i$, the $i$ in the right column must be in the same row or higher than the $i$ in the left column.
\item[iii)] The column immediately left of the leftmost affected column cannot have a cell containing an unpaired $i$.
\item[iv)] The $i+1$ in the leftmost affected column cannot have a cell with value $i$ immediately to its left.
\item[v)] The column immediately right of the rightmost affect column cannot have any cells with an unpaired $i+1$.
\item[vi)] The $i+1$ in the rightmost affected column cannot have a cell with value $i+1$ immediately to its right.
\end{itemize}
\end{proposition}

\begin{proof}
We prove each point separately.
\begin{itemize}
\item[i)] This follows from the definition of $i$-paring, since if such an $i$ existed then the leftmost unpaired $i+1$ would be paired.
\item[ii)] If this were not the case both cells containing $i$ would be attacking, which is impossible since $T \in \SSKD(a)$.
\item[iii)] If such a cell existed then by Definition~\ref{def:pair} it would $i$-pair with the rightmost unpaired $i+1$, which is a contradiction.
\item[iv)] Consider the cell containing an $i+1$ in the leftmost affected column. Then if the cell immediately to its right has value $i$, then by (iii) this $i$ must be paired with some $i+1$ located in the same column and below it. If the $i$ is in a row strictly shorter than the row of the $i+1$ below it with which it is paired, then the cell right of this $i+1$ must also contain an $i+1$. This implies the leftmost affected column has two cells with value $i+1$ which is impossible. If instead, the $i$ is in a row weakly longer than the row of the $i+1$ below it with which it is paired, then the cell right of this $i+1$ must again have value $i+1$. Thus, the cell right of the leftmost affected $i+1$ cannot have value $i$.
\item[v)] This is obvious, since clearly there can exist no unpaired $i+1$'s right of the rightmost unpaired $i+1$.

\item[vi)] Consider the cell with value $i+1$ in the rightmost affected column of $T$ and suppose the cell immediately to its right also contained an $i+1$. By (v) this cell must be paired with an $i$ located in the same column and above it. If the row of the $i$ in this column were weakly longer than the row of the $i+1$ below it, then the cell right of the $i$ would also have value $i$. This is impossible since the $i+1$ in the rightmost affected column is either unpaired or paired with an $i$ below it. Thus, the row of $i$ must strictly shorter than the row of $i+1$ and thus the cell to the right of $i+1$ must also have value $i+1$. Since this new $i+1$ also lies right of the rightmost unpaired $i+1$ we can iterate the previous argument for each consecutive column to the right of the rightmost affected column and conclude that the rightmost affected $i+1$ has a cell in the same row and in some column to its right with an unpaired $i+1$ which is clearly a contraction.
\end{itemize}
\end{proof}

\begin{figure}[ht]
\begin{center}
\begin{tikzpicture}[scale=.6,every node/.style={scale=.7}]
\node at (-1,6) (a)[scale=1.5] {$i$};
\node at (-4,4)(b)[scale=1.5] {$i$};
\node at (-3,4) (c) [scale=1.5]{$i$};
\node at (-6,3) (d) [scale=1.5]{$i$};
\node at (-7,2) (e) [scale=1.5]{$i$};
\node at (-8,0) {$\cdots$};
\node at (-7,0) (f) {$i+1$};
\node at (-6,0) (g) {$i+1$};
\node at (-5,0) {$\cdots$};
\node at (-4,0) (h) {$i+1$};
\node at (-3,0) (i) {$i+1$};
\node at (-2,0){$\cdots$};
\node at (-1,0) (j) {$i+1$};
\node at (0,0) (k) {$\color{red}i+1$};
\node at (1,0) (l) {$i+1$};
\node at (2,0){$\cdots$};
\node at (3,0) (m) {$i+1$};
\node at (4,0) (n) {$i+1$};
\node at (5,0){$\cdots$};
\node at (6,0) (o) {$i+1$};
\node at (7,0) (p) {$i+1$};
\node at (8,0) {$\cdots$};
\node at (1,-6)(r)[scale=1.5]{$i$};
\node at (3,-5)(s)[scale=1.5]{$i$};
\node at (4,-3)(t)[scale=1.5]{$i$};
\node at (6,-2)(u)[scale=1.5]{$i$};
\node at (7,-1)(v)[scale=1.5]{$i$};
\csquare{(a)}
\csquare{(b)}
\csquare{(c)}
\csquare{(d)}
\csquare{(e)}
\csquare{(f)}
\csquare{(g)}
\csquare{(h)}
\csquare{(i)}
\csquare{(j)}
\csquare{(k)}
\csquare{(l)}
\csquare{(m)}
\csquare{(n)}
\csquare{(o)}
\csquare{(p)}
\csquare{(r)}
\csquare{(s)}
\csquare{(t)}
\csquare{(u)}
\csquare{(v)}
\draw[dotted] (-1,5.25)--(-1,.75);
\draw[dotted] (-3,3.25)--(-3,.75);
\draw[dotted] (-4,3.25)--(-4,.75);
\draw[dotted] (-6,2.25)--(-6,.75);
\draw[dotted] (-7,1.25)--(-7,.75);
\draw[red,dotted](0,6.5)--(0,.75);
\draw[red,dotted](0,-6.5)--(0,-.75);
\draw[dotted] (1,-5.25)--(1,-.75);
\draw[dotted] (3,-4.25)--(3,-.75);
\draw[dotted] (4,-2.25)--(4,-.75);
\draw[dotted] (6,-1.25)--(6,-.75);
\end{tikzpicture}
  \caption{\label{fig:raising-distribution}The layout of the consecutive columns of $T \in \SSKD(a)$ affected by $e_i$. The rightmost unpaired $i+1$ of $T$ is highlighted in red and contains no cells equal to $i$ in the same column. For any other two adjacent columns, the row of the $i$ in the left column must be at the same height or lower than the row of the $i$ in the right column. }
  \end{center}
\end{figure}
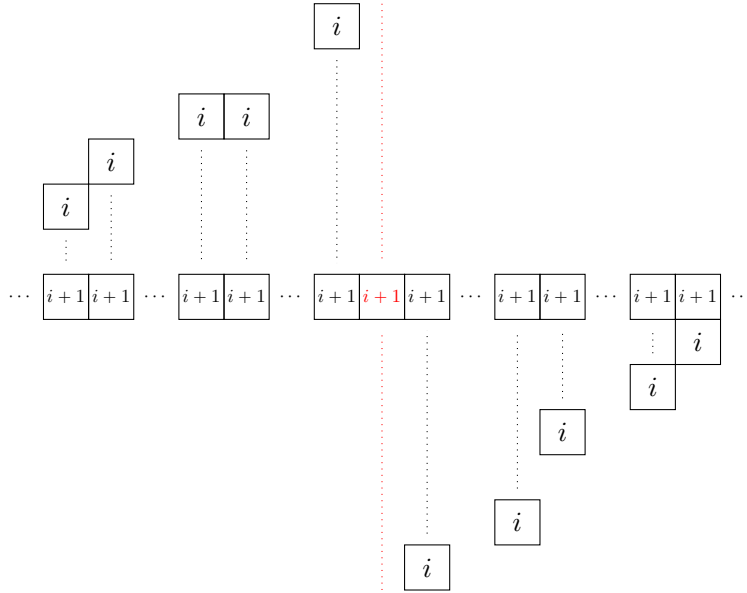

Unlike the crystal operators on semistandard Young tableaux, the raising operator $e_i$ on semistandard key tabloids can invert the relative order of $i$ and $i+1$ within a given column, motivating the following definition.

Given any $T \in \SSKD(a)$ and integer $1 \leq i<n$ such that $e_i(T) \neq 0$, we say $e_i$ \emph{flips} $T$ if its action on $T$ changes an $i$ above $i+1$ in some column to an $i+1$ above $i$.  For instance, in Example~\ref{ex:raising-oper} the first application of $e_4$ does not flip the tabloid but the second application of $e_4$ does.

From Definition~\ref{def:raise-tabloid}, we see that for any $T \in \SSKD(a)$, the action of $e_i$ on $T$ is restricted to cells with $i$ and $i+1$'s in a consecutive set of columns left and right of the rightmost unpaired $i+1$. As will be demonstrated in the following two lemmas, the very specific distribution these columns satisfy (see Fig.~\ref{fig:raising-distribution}) imposes certain restrictions on the lengths of the rows containing $i+1$ and $i$ in each column.

\begin{lemma}\label{lem:A}
Let $i$ be an integer $1 \leq i <n$ and suppose $T \in \SSKD(a)$  has a cell with an unpaired $i+1$ in column $c$ and row $r$, for some $c,r \geq 1$.

If $T$ contains a sequence of consecutive columns immediately right of column $c$ such that each column has an $i+1$ in row $r$ and an $i$ in some row below it, then row $r$ is weakly longer than the rows of all the $i$'s contained in the sequence.
\end{lemma}

\begin{proof}
Denote by $\lbrace c+s \rbrace_{1\leq s \leq m}$ the maximal sequence of consecutive columns right of column $c$ satisfying the conditions above, and for each column $c+s$, denote by $r_s$ the row containing $i$. We proceed by induction on $s$.

First, suppose $r$ is strictly shorter than $r_1$. Since any other value would create a co-inversion triple, then the cell below $i+1$ in column $c$ and left of $i$ in column $c+1$ must contain an $i$. This contradicts the $i+1$ in column $c$ being unpaired. Hence, $r$ must be weakly longer than $r_1$.

Now, suppose there is an $s>1$ such that $r$ is weakly longer than $r_{s-1}$ but is strictly shorter than $r_s$. Then the cell left of the $i$ in row $r_s$ and below the $i+1$ in column $c+s-1$ must have value $i$. Since no column may have two cells of equal value, this implies $r_s=r_{s-1}$. Thus, $r$ is weakly longer than $r_s$ which contradicts the assumptions.
\end{proof}

\begin{lemma}\label{lem:B}
Let $i$ be an integer $1 \leq i <n$ and suppose $T \in \SSKD(a)$  has a cell with an unpaired $i+1$ in column $c$ and row $r$, for some $c,r \geq 1$.

If $T$ contains a sequence of consecutive columns immediately left of column $c$ such that each column has an $i+1$ in row $r$ and an $i$ in some row above it, then row $r$ is strictly longer than the rows of all the $i$'s contained in the sequence.
\end{lemma}

\begin{proof}
Once again, denote by $\lbrace c-s \rbrace_{1\leq s \leq n}$ the maximal sequence of consecutive columns left of column $c$ satisfying the assumptions above, and for each column $c-s$, denote by $r_s$ the row containing $i$. We proceed by induction on $-s$.

Suppose $r$ is weakly shorter than $r_n$. Then, the cell immediately left of the $i$ contained in column $c-n$ must have value $i$. However, since $n$ is maximal, if such a cell  existed it would be paired with the $i+1$ contained in column $c$, which is a contradiction.

If there exists $s<n$ such that $r$ is strictly longer than $r_s$ but weakly shorter than $r_{s-1}$, then the cell immediately left of the $i$ contained in column $c-{s-1}$ has entry $i$. This implies $r_s=r_{s-1}$, so that $r_{s-1}$ is both strictly longer and weakly shorter than $r$, which is nonsense.
\end{proof}

The specific distribution of the cells containing $i$ and $i+1$ exemplified in the previous two lemmas is the key to proving that the raising operators $e_i$ are well defined on $\SSKD(a)$. In order to do so we show that for any $T \in \SSKD(a)$ on which $e_i$ acts non-trivially, $e_i(T)$ has no attacking cells nor any co-inversion triples and has the same major index as $T$.

\begin{lemma}\label{lem:raising-nonattacking}
  Let $T \in \SSKD(a)$ and $1\leq i <n$ be an integer such that $e_i(T) \neq 0$. Then $e_i(T)$ has no attacking cells.
\end{lemma}

\begin{proof}
  First consider the case of two cells located in the same column. If the column contains both an $i$ and an $i+1$, then $e_i$ will act by swapping these entries, and hence its image will never have a column with two cells of the same value. Likewise, if the column contains an unpaired $i+1$, then this column cannot contain a cell with value $i$. Thus, when $e_i$ sends this $i+1$ to $i$ it will not be attacking a cell in the same column, so we need only consider the case of attacking cells in adjacent columns  with the cell on the left strictly higher than the cell on the right.

  If the left cell has value $i$ and the right cell has value $i+1$ (see the left diagram in Fig.~\ref{fig:Nonattacking}), then $e_i$ acts non-trivially on these columns only if $i$ is paired with an $i+1$ below it and immediately left of the $i+1$ in the right cell. Since $e_i$ will act by exchanging these entries, the image of these columns will remain non-attacking. 
  
Suppose instead the left cell has value $i+1$ and the right cell has value $i$ (see the middle diagram in Fig.~\ref{fig:Nonattacking}). If the top row is weakly longer then the entry immediately right of the left cell must have value $i+1$. Thus, if $e_i$ acts non-trivially on these columns it must swap all $i$'s and $i+1$'s and consequently does not create any attacking cells. 

If the top row is strictly shorter, then the entry immediately left of the right cell must have value $i$ (see the right diagram in Fig.~\ref{fig:Nonattacking}). Moreover, if $e_i$ acted only on the right column then the right cell containing an $i$ would be paired with an $i+1$ above it, which occurs only if there exists an unpaired $i+1$ in some column to its left. However, by the definition of $e_i$, this implies the $i+1$ above the right cell and the $i+1$ in the left cell lie in the same row. By Lemma~\ref{lem:A} this is impossible since this would mean the top row is both weakly longer and strictly shorter than the bottom row. If $e_i$ acted on the left column then once again by Lemma~\ref{lem:A} this leads to a contradiction regarding the relative lengths of the top and bottom row.

\end{proof}

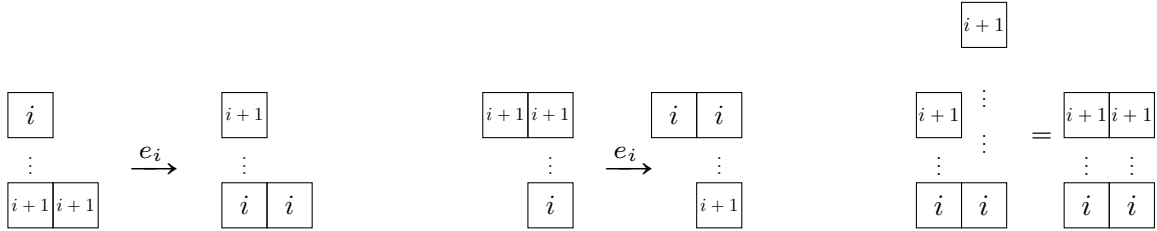
\begin{figure}[ht]
\begin{center}
\begin{tikzpicture}[scale=.6,every node/.style={scale=.7}]
\node at (0,0) (a) {$i+1$};
\node at (0,1){$\vdots$};
\node at (0,2)(b)[scale=1.5]{$i$};
\node at (1,0)(c){$i+1$};
\csquare{(a)}
\csquare{(b)}
\csquare{(c)}
\end{tikzpicture}\;
\begin{tikzpicture}[scale=.6,every node/.style={scale=.7}]
\node at (-2,1) [scale=2]{$\xrightarrow{e_i}$};
\node at (0,0) (a) [scale=1.5]{$i$};
\node at (0,1){$\vdots$};
\node at (0,2)(b){$i+1$};
\node at (1,0)(c)[scale=1.5]{$i$};
\csquare{(a)}
\csquare{(b)}
\csquare{(c)}
\end{tikzpicture}
\qquad\qquad\qquad
\begin{tikzpicture}[scale=.6,every node/.style={scale=.7}]
\node at (0,0) (a)[scale=1.5]{$i$};
\node at (0,1){$\vdots$};
\node at (0,2)(b){$i+1$};
\node at (-1,2)(c){$i+1$};
\csquare{(a)}
\csquare{(b)}
\csquare{(c)}
\end{tikzpicture}\;
\begin{tikzpicture}[scale=.6,every node/.style={scale=.7}]
\node at (-2,1) [scale=2]{$\xrightarrow{e_i}$};
\node at (0,0) (a){$i+1$};
\node at (0,1){$\vdots$};
\node at (0,2)(b)[scale=1.5]{$i$};
\node at (-1,2)(c)[scale=1.5]{$i$};
\csquare{(a)}
\csquare{(b)}
\csquare{(c)}
\end{tikzpicture}
\qquad \qquad\qquad
\begin{tikzpicture}[scale=.6,every node/.style={scale=.75}]
\node at (0,0) (a) [scale=1.5]{$i$};
\node at (0,1){$\vdots$};
\node at (0,2)(b){$i+1$};
\node at (1,0)(c)[scale=1.5]{$i$};
\node at (1,4)(d){$i+1$};
\node at (1,1.5){$\vdots$};
\node at (1,2.5){$\vdots$};
\csquare{(a)}
\csquare{(b)}
\csquare{(c)}
\csquare{(d)}
\end{tikzpicture}
\begin{tikzpicture}[scale=.6,every node/.style={scale=.75}]
\node at (-1,1.5)[scale=1.5]{$=$};
\node at (0,0) (a) [scale=1.5]{$i$};
\node at (0,1){$\vdots$};
\node at (0,2)(b){$i+1$};
\node at (1,0)(c)[scale=1.5]{$i$};
\node at (1,2)(d){$i+1$};
\node at (1,1){$\vdots$};
\csquare{(a)}
\csquare{(b)}
\csquare{(c)}
\csquare{(d)}
\end{tikzpicture}
\caption{The different cases for the proof of Lemma~\ref{lem:raising-nonattacking}.}\label{fig:Nonattacking}
\end{center}
\end{figure}

\begin{lemma}\label{lem:raising-majpreserving}
  Let $T \in \SSKD(a)$ and $1 \leq i <n$ and integer such that $e_i(T) \neq 0$. Then $\maj\left( e_i(T)\right)=\maj(T)$.
\end{lemma}

\begin{proof}
  We will show that the set of cells $c$ for which the entry is greater than that to its right is preserved by $e_i$, which implies preservation of major index.

  Consider a fixed row of $T$. If no cell in that row changes from $T$ to $e_i(T)$, then the major index is trivially maintained. Suppose then that cells $b_1,\ldots,b_k$, $k\geq 1$, change in passing from $T$ to $e_i(T)$. By Proposition~\ref{prop:raising-distribution}, these entries must lie in consecutive columns, say with $b_{j}$ immediately left of $b_{j+1}$. Let $a$ denote the cell immediately left of $b_1$ and let $c$ denote the cell immediately right of $b_k$.

  By the definition of $e_i$, all entries in $b_1,\ldots,b_k$ of $T$ must coincide, and $e_i$ will toggle the values between $i$ and $i+1$. Thus there is no decent among $b_1,\ldots,b_k$ before or after applying $e_i$. Therefore the only cases to be checked are the potential descent from $a$ to $b_1$ and from $b_k$ to $c$.

  If $a<i$, then $a < i,i+1$ creating a descent in both $T$ and $e_i(T)$, and if $a\geq i+1$, then $a \geq i,i+1$ avoiding a descent in both $T$ and $e_i(T)$. By Proposition~\ref{prop:raising-distribution}(iii), $a\neq i$, so this resolves all cases for $a$.

  Similarly, if $c\leq i$, then $i,i+1 \geq c$ avoiding a descent in both $T$ and $e_i(T)$, and if $c> i+1$, then $i,i+1 < c$ creating a descent in both $T$ and $e_i(T)$. By Proposition~\ref{prop:raising-distribution}(iv), $c\neq i+1$, so this resolves all cases for $c$.
\end{proof}

\begin{lemma}\label{lem:raising-inversiontriples}
Let $T \in \SSKD(a)$ and $1 \leq i <n$ an integer such that $e_i(T) \neq 0$. Then $\coinv(e_i(T))=0$.
\end{lemma}

\begin{proof}
Suppose $T \in \SSKD(a)$ has two consecutive rows forming a triple.  If only one of the cells contains an entry equal to $i$ or $i+1$ but the other two cells do not then clearly the action of $e_i$ on these columns will not modify the existing orientation and so the image of $T$ under $e_i$ will not contain co-inversion triples.

Now suppose the triple has two cells with value $i$ or $i+1$. By \cite{HHL08}(Lemma~3.6.3), a co-inversion triple will never contain two cells with equal value, and so it suffices to consider triples containing one cell with value $i$, one cell with value $i+1$, and one cell with value $x \neq i, i+1$. Hence, let $T \in \SSKD(a)$ be such that $e_i(T)$ contains a co-inversion triple with only one cell not equal to $i$ or $i+1$. We will show that such a $T$ cannot exist by considering each possible co-inversion triple of this form and deriving a contradiction.

Assume $e_i(T)$ contains a triple of \emph{type I}, thus the bottom row is strictly longer than the top row of the triple.

\begin{itemize}
\item Suppose $i$ lies in the top cell, $i+1$ in the bottom left cell, and $x$ in the bottom right cell (see figure below). Then the only possible pre-image interchanges the $i$ and $i+1$. However, if in $T$ $i+1$ lies above $i$ in the same column then by Lemma~\ref{lem:A} the top row must be weakly longer than the bottom row, which is a contradiction.
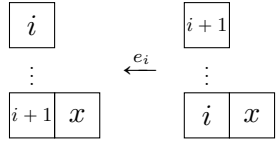
\[
\hackcenter{\begin{tikzpicture}[scale=.6,every node/.style={scale=.75}]
\node at (0,0) (a) {$i+1$};
\node at (0,1){$\vdots$};
\node at (0,2)(b)[scale=1.5]{$i$};
\node at (1,0)(c)[scale=1.5]{$x$};
\csquare{(a)}
\csquare{(b)}
\csquare{(c)}
\end{tikzpicture}}\;\;
\xleftarrow{e_i}\;\;
\hackcenter{\begin{tikzpicture}[scale=.6,every node/.style={scale=.75}]
\node at (0,0) (a) [scale=1.5]{$i$};
\node at (0,1){$\vdots$};
\node at (0,2)(b){$i+1$};
\node at (1,0)(c)[scale=1.5]{$x$};
\csquare{(a)}
\csquare{(b)}
\csquare{(c)}
\end{tikzpicture}}\]

\item Suppose $i+1$ lies in the top cell, $x$ in the bottom left cell, and $i$ in the bottom right cell. If in $T$ the cells above and right of $x$ were either both $i$ or both $i+1$, they would be attacking. Thus, $T$ must have an $i$ in the cell above $x$ and an $i+1$ in the cell right of $x$. This implies that in $T$ the cell containing $i$ is paired with an $i+1$ below it in the same column. By the definition of $e_i$, the $i+1$'s in the adjacent columns must lie in the same row so $x=i+1$. This is impossible since by assumption $x \neq i+1$.
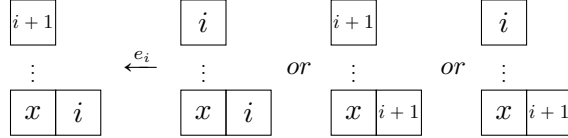
\[
\hackcenter{\begin{tikzpicture}[scale=.6,every node/.style={scale=.75}]
\node at (0,0) (a) [scale=1.5]{$x$};
\node at (0,1){$\vdots$};
\node at (0,2)(b){$i+1$};
\node at (1,0)(c)[scale=1.5]{$i$};
\csquare{(a)}
\csquare{(b)}
\csquare{(c)}
\end{tikzpicture}}\;\;
\xleftarrow{e_i}\;\;
\hackcenter{\begin{tikzpicture}[scale=.6,every node/.style={scale=.75}]
\node at (0,0) (a) [scale=1.5]{$x$};
\node at (0,1){$\vdots$};
\node at (0,2)(b)[scale=1.5]{$i$};
\node at (1,0)(c)[scale=1.5]{$i$};
\csquare{(a)}
\csquare{(b)}
\csquare{(c)}
\end{tikzpicture}}\;\;
or
\;\;
\hackcenter{\begin{tikzpicture}[scale=.6,every node/.style={scale=.75}]
\node at (0,0) (a) [scale=1.5]{$x$};
\node at (0,1){$\vdots$};
\node at (0,2)(b){$i+1$};
\node at (1,0)(c){$i+1$};
\csquare{(a)}
\csquare{(b)}
\csquare{(c)}
\end{tikzpicture}}\;\;
or
\;\;
\hackcenter{\begin{tikzpicture}[scale=.6,every node/.style={scale=.75}]
\node at (0,0) (a) [scale=1.5]{$x$};
\node at (0,1){$\vdots$};
\node at (0,2)(b)[scale=1.5]{$i$};
\node at (1,0)(c){$i+1$};
\csquare{(a)}
\csquare{(b)}
\csquare{(c)}
\end{tikzpicture}}\]

\item Suppose $x$ lies in the top cell, $i$ in the bottom left cell, and $i+1$ in the bottom right cell. If $T$ contained an $i$ in both cells of the bottom row, then this would imply the right column of the triple was the leftmost affected column of $T$. However, by Proposition~\ref{prop:raising-distribution} (iii) there cannot be an unpaired $i$ left of the leftmost affected column, hence this situation is impossible.If instead, $T$ contained an $i$ in the left bottom cell and an $i+1$ in the right bottom cell, then by Definition~\ref{def:raise-tabloid} $e_i$ would not act on both columns. Thus this cannot be the preimage of $e_i(T)$.  Finally, if both cells on the bottom row of $T$ had value $i+1$ then the left column of the triple would be the rightmost affected column. By Proposition~\ref{prop:raising-distribution} (vi) this is impossible. Thus, for any $T \in \SSKD(a)$ its image $e_i(T)$ will never contain a co-inversion triple of this form.
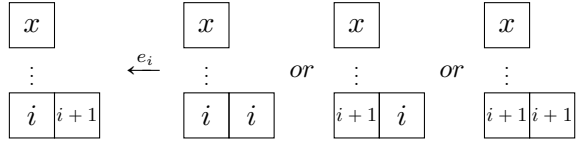
\[
\hackcenter{\begin{tikzpicture}[scale=.6,every node/.style={scale=.75}]
\node at (0,0) (a) [scale=1.5]{$i$};
\node at (0,1){$\vdots$};
\node at (0,2)(b)[scale=1.5]{$x$};
\node at (1,0)(c){$i+1$};
\csquare{(a)}
\csquare{(b)}
\csquare{(c)}
\end{tikzpicture}}\;\;
\xleftarrow{e_i}\;\;
\hackcenter{\begin{tikzpicture}[scale=.6,every node/.style={scale=.75}]
\node at (0,0) (a)[scale=1.5]{$i$};
\node at (0,1){$\vdots$};
\node at (0,2)(b)[scale=1.5]{$x$};
\node at (1,0)(c)[scale=1.5]{$i$};
\csquare{(a)}
\csquare{(b)}
\csquare{(c)}
\end{tikzpicture}}\;\;
or
\;\;
\hackcenter{\begin{tikzpicture}[scale=.6,every node/.style={scale=.75}]
\node at (0,0) (a){$i+1$};
\node at (0,1){$\vdots$};
\node at (0,2)(b)[scale=1.5]{$x$};
\node at (1,0)(c)[scale=1.5]{$i$};
\csquare{(a)}
\csquare{(b)}
\csquare{(c)}
\end{tikzpicture}}\;\;
or
\;\;
\hackcenter{\begin{tikzpicture}[scale=.6,every node/.style={scale=.75}]
\node at (0,0) (a) {$i+1$};
\node at (0,1){$\vdots$};
\node at (0,2)(b)[scale=1.5]{$x$};
\node at (1,0)(c){$i+1$};
\csquare{(a)}
\csquare{(b)}
\csquare{(c)}
\end{tikzpicture}}\]

\end{itemize}

Now assume $e_i(T)$ contains a triple of \emph{type II}, and thus the top row is weakly longer than the bottom row of the triple.

\begin{itemize}
\item Suppose $i$ lies in the bottom cell, $i+1$ in the top left cell, and $x$ in the top right cell. If in $T$ the cells in the triple left and below $x$ both had value $i$ or $i+1$, then these cells would be attacking. Thus, $T$ must have an $i$ right of $x$ in the top row and an $i+1$ below $x$ in the bottom row. By Definition~\ref{def:raise-tabloid}, the $i$ in the top left cell must be $i$-paired with an $i+1$ below it and left of the $i+1$ below $x$. By Lemma~\ref{lem:B} since $i$ lies above $i+1$ then the bottom row must be strictly longer. This is a contradiction since by definition type II triples must have the top row weakly longer than the bottom row.
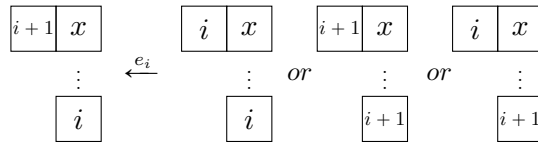
\[
\hackcenter{\begin{tikzpicture}[scale=.6,every node/.style={scale=.75}]
\node at (0,0) (a)[scale=1.5]{$i$};
\node at (0,1){$\vdots$};
\node at (0,2)(b)[scale=1.5]{$x$};
\node at (-1,2)(c){$i+1$};
\csquare{(a)}
\csquare{(b)}
\csquare{(c)}
\end{tikzpicture}}
\;\;
\xleftarrow{e_i}\;\;
\hackcenter{\begin{tikzpicture}[scale=.6,every node/.style={scale=.75}]
\node at (0,0) (a)[scale=1.5]{$i$};
\node at (0,1){$\vdots$};
\node at (0,2)(b)[scale=1.5]{$x$};
\node at (-1,2)(c)[scale=1.5]{$i$};
\csquare{(a)}
\csquare{(b)}
\csquare{(c)}
\end{tikzpicture}}
\;\; or
\hackcenter{\begin{tikzpicture}[scale=.6,every node/.style={scale=.75}]
\node at (0,0) (a){$i+1$};
\node at (0,1){$\vdots$};
\node at (0,2)(b)[scale=1.5]{$x$};
\node at (-1,2)(c){$i+1$};
\csquare{(a)}
\csquare{(b)}
\csquare{(c)}
\end{tikzpicture}}
\;\; or
\hackcenter{\begin{tikzpicture}[scale=.6,every node/.style={scale=.75}]
\node at (0,0) (a){$i+1$};
\node at (0,1){$\vdots$};
\node at (0,2)(b)[scale=1.5]{$x$};
\node at (-1,2)(c)[scale=1.5]{$i$};
\csquare{(a)}
\csquare{(b)}
\csquare{(c)}
\end{tikzpicture}}
\]

\item Suppose $x$ lies in the bottom cell, $i$ in the top left cell, and $i+1$ in the top right cell. If $T$ had an $i+1$ in the top left cell and an  $i$ in the top right cell, then by Definition~\ref{def:raise-tabloid} and Proposition~\ref{prop:raising-distribution}(v) $e_i$ would only act on left column, so this cannot be the preimage. If $T$ contained an $i$ in both cells of the top row then the right column of the triple would be the leftmost affected column of $T$. By Proposition~\ref{prop:raising-distribution} (iii) this cannot occur since the left column of the triple cannot have an unpaired $i$. Thus in $T$, both cells in the top row of the triple must have value $i+1$. However, this implies the left column of the triple is the rightmost affected column of $T$, so by Proposition~\ref{prop:raising-distribution} (vi) the column to its right cannot contain an unpaired $i+1$. Thus this co-inversion triple cannot be a part of the image under $e_i$ for any $T \in \SSKD(a)$.
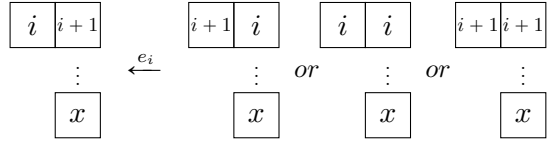
\[
\hackcenter{\begin{tikzpicture}[scale=.6,every node/.style={scale=.75}]
\node at (0,0) (a)[scale=1.5]{$x$};
\node at (0,1){$\vdots$};
\node at (0,2)(b){$i+1$};
\node at (-1,2)(c)[scale=1.5]{$i$};
\csquare{(a)}
\csquare{(b)}
\csquare{(c)}
\end{tikzpicture}}
\;\;
\xleftarrow{e_i}\;\;
\hackcenter{\begin{tikzpicture}[scale=.6,every node/.style={scale=.75}]
\node at (0,0) (a)[scale=1.5]{$x$};
\node at (0,1){$\vdots$};
\node at (0,2)(b)[scale=1.5]{$i$};
\node at (-1,2)(c){$i+1$};
\csquare{(a)}
\csquare{(b)}
\csquare{(c)}
\end{tikzpicture}}
\;\; or
\hackcenter{\begin{tikzpicture}[scale=.6,every node/.style={scale=.75}]
\node at (0,0) (a)[scale=1.5]{$x$};
\node at (0,1){$\vdots$};
\node at (0,2)(b)[scale=1.5]{$i$};
\node at (-1,2)(c)[scale=1.5]{$i$};
\csquare{(a)}
\csquare{(b)}
\csquare{(c)}
\end{tikzpicture}}
\;\; or
\hackcenter{\begin{tikzpicture}[scale=.6,every node/.style={scale=.75}]
\node at (0,0) (a)[scale=1.5]{$x$};
\node at (0,1){$\vdots$};
\node at (0,2)(b){$i+1$};
\node at (-1,2)(c){$i+1$};
\csquare{(a)}
\csquare{(b)}
\csquare{(c)}
\end{tikzpicture}}
\]
\item Lastly, suppose $x$ lies in the top left cell, $i$ in the top right cell, and $i+1$ in the bottom cell. Then the only possible preimage has an $i+1$ in the top right cell and an $i$ in the cell below it. Since the top row is weakly longer than the bottom row, this implies $x=i+1$ which is a contradiction.
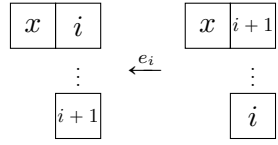
\[
\hackcenter{\begin{tikzpicture}[scale=.6,every node/.style={scale=.75}]
\node at (0,0) (a){$i+1$};
\node at (0,1){$\vdots$};
\node at (0,2)(b)[scale=1.5]{$i$};
\node at (-1,2)(c)[scale=1.5]{$x$};
\csquare{(a)}
\csquare{(b)}
\csquare{(c)}
\end{tikzpicture}}
\;\;
\xleftarrow{e_i}\;\;
\hackcenter{\begin{tikzpicture}[scale=.6,every node/.style={scale=.75}]
\node at (0,0) (a)[scale=1.5]{$i$};
\node at (0,1){$\vdots$};
\node at (0,2)(b){$i+1$};
\node at (-1,2)(c)[scale=1.5]{$x$};
\csquare{(a)}
\csquare{(b)}
\csquare{(c)}
\end{tikzpicture}}\]

\end{itemize}
\end{proof}

From these lemmas, we establish the following theorem.

\begin{theorem}
 For any integer $1 \leq i <n$, the raising operators $e_i : \SSKD(a) \rightarrow \SSKD(a) \cup \{0\}$ are well-defined $\maj$-preserving maps on $\SSKD(a)$.
  \label{thm:well-defined}
\end{theorem}

\begin{proof}
Recall that if $T$ is a semistandard key tabloid no cells of $T$ can form attacking pairs or co-inversion triples. By Lemmas~\ref{lem:raising-nonattacking} and \ref{lem:raising-inversiontriples} we have that for any $T \in \SSKD(a)$ with $e_i(T) \neq 0$ then $e_i(T)$ has no attacking cells and $\coinv(T)=0$. Thus $e_i(T) \in \SSKD(a)$. Moreover, by Lemma~\ref{lem:raising-majpreserving} we also have $\maj(e_i(T))=\maj(T)$ so $e_i$ is indeed $\maj$-preserving.
\end{proof}

In an entirely analogous manner for integers $1\leq i<n$ we can define \emph{lowering operators}, denoted by $f_i$, satisfying $f_i(T^{\prime})=T$ if and only if $e_i(T) = T^{\prime}$ for all $T, T^{\prime} \in \SSKD(a)$.

\begin{definition} \label{def:lower-key}
Given any $T \in \SSKD(a)$ and an integer $1\leq i<n$, define the \newword{lowering operator} $f_i$ on $\SSKD(a)$ whose action on $T$ is as follows:
\begin{itemize}
\item Set $f_i(T)=0$ whenever
\begin{itemize}
\item $T$ does not have any cells containing an unpaired $i$ or
\item the leftmost unpaired $i$ is in row $i$ and all columns to its left have an $i$ in the same row with an $i+1$ above them.
\end{itemize}
\item otherwise, $f_i$ changes the leftmost unpaired $i$ to $i+1$ and
\begin{itemize}
\item swaps the entries $i$ and $i+1$ in each of the consecutive columns left of this entry that have an $i$ in the same row and an $i+1$ above, and
\item swaps the entries $i$ and $i+1$ in each of the consecutive columns right of this entry that have an $i$ in the same row and an $i+1$ below.
\end{itemize}
\end{itemize}
\end{definition}

Whenever $T$ has an unpaired $i$ yet $f_i(T)=0$, we say that $T$ is subject to \newword{Demazure death}.

 We similarly say that $f_i$ \emph{flips} $T$ if within some column if $f_i$ changes $i+1$ above $i$ to become $i$ above $i+1$.

\begin{remark} \label{rem:complete-i-string}
Given any $T \in \SSKD(a)$ we note that if $f_i(T) \neq 0$ then the first column cannot have an unpaired $i$ in row $i$ (since $f_i$ will always act on the leftmost unpaired $i$). So if $N_i$ is the number of unpaired cells with value $i$ of $T$ then $f_i^s(T) \neq 0$ for all $1\leq s \leq N_i$ but $f_i^{N_i+1}(T)=0$. Thus for any $T$ with $e_i(T)=0$ either $f_i(T)=0$ or $f_i$ will act nontrivially exactly $N_i$ times.
\end{remark}

We now prove that these raising and lowering operators are inverse to one another when nonzero.

\begin{theorem}
  For $S,T\in\SSKD(a)$, we have $e_i(S) = T$ if and only if $f_i(T) = S$.
\end{theorem}

\begin{proof}
Suppose $T \in \SSKD(a)$ and $i$ and integer such that $e_i(T) \neq 0$. Since $T$ contains an unpaired $i+1$, necessarily there is no unpaired $i$ left of the rightmost unpaired $i+1$. Consequently, the leftmost unpaired $i$ in $e_i(T)$ will be precisely the image of the rightmost unpaired $i+1$ in $T$. By Proposition~\ref{prop:raising-distribution}
(iv) it follows that the affected columns left of the rightmost unpaired $i+1$ in $T$ will be the same affected columns left of the leftmost unpaired $i$ in $e_i(T)$. Furthermore, if the $i+1$ in the rightmost affected column of $T$ has a cell immediately right with value $i$, then this $i$ will either be unpaired or paired with an $i+1$ above it and will thus be unaffected by the action of $e_i$ and $f_i$. Hence the columns right of the rightmost unpaired $i+1$ in $T$ affected by $e_i$ are the same as columns right of the leftmost unpaired $i$ in $e_i(T)$ affected by $f_i$. Thus $f_i(e_i(T))=T$.

Likewise, if $S \in \SSKD(a)$ with $f_i(S) \neq 0$ then since $f_i$ acts non-trivially only if there is no unpaired $i+1$ right of the leftmost unpaired $i$, then the rightmost unpaired $i+1$ in $f_i(S)$ must be the image of the leftmost unpaired $i$ in $S$. Since in the rightmost of column of $S$ affected by $f_i$ the cell containing an $i$ cannot contain an $i+1$ immediately to its left, then the columns in $S$ affected by $f_i$ right of the leftmost unpaired $i$ are the same as the columns in $f_i$ right of the rightmost unpaired $i+1$  in $f_i(S)$ affected by $e_i$. Furthermore, if in the leftmost column of $S$ affected by $f_i$ the cell containing an $i$ had a cell with an $i+1$ immediately to its left, then this $i+1$ cannot have an $i$ above it since that would create an attacking cell. Thus, the columns of $S$ affected by $f_i$ left of the leftmost unpaired $i$ are the same as the columns of $f_i(S)$ affected by $e_i$ left of the rightmost unpaired $i+1$. Hence, $e_i(f_i(S))=S$.
\end{proof}

\subsection{Rectification of key tabloids}
\label{sec:tabloid-rect}

Assaf and Schilling \cite{ASc18} defined an explicit Demazure crystal structure on semistandard key \emph{tableaux} \cite{Ass18}, the objects that correspond to Mason's semi-skyline augmented fillings \cite{Mas09}. As semistandard key tableaux are precisely the semistandard key tabloids with $\maj = 0$ \cite{Ass18}(Proposition~3.1), we can consider our operators restricted to this case, and in so doing we recover the constructions of Assaf and Schilling \cite{ASc18}.

\begin{proposition}
  The raising operators in Definition~\ref{def:raise-tabloid} restricted to semistandard key tableaux agree with the raising operators in \cite{ASc18}(Definition~3.7).
  \label{prop:ASc}
\end{proposition}

\begin{proof}
  Using notation and terminology from \cite{ASc18}, the condition $m_i(w(T))<0$ is precisely the statement that there exists unpaired $i$'s left of the rightmost unpaired $i+1$ and $m_i(w(T))=0$ is the statement that no unpaired $i+1's$ exists. Moreover, in an $\SSKT$ there can never be columns as in Lemma~\ref{lem:B} since the entry immediately left of the $i$ in column $c+1$ would have to be greater than $i+1$, and consequently, the entry to the left of it would need to be smaller than $i+1$, which contradicts rows being weakly decreasing. Thus, the entry where $q$ is maximal as in  \cite{ASc18} is precisely the $i+1$ in the leftmost affected column $c-n$. Since by definition in \cite{ASc18}, $e_i$ swaps $i$'s and $i+1$'s in all columns weakly right of this entry, then the columns flipped by the original definition of \cite{ASc18} and the one given here are exactly the same.
\end{proof}

In particular, by \cite{ASc18}(Theorem~3.14), the raising operators on semistandard key tableaux give a Demazure crystal. We aim to show this holds for semistandard key tabloids as well by comparing the latter with the former. To achieve this, we shift our paradigm from tabloids to \emph{diagrams}, arbitrary collections of unit cells in the first quadrant, based Kohnert's \cite{Koh91} elegant combinatorial algorithm for computing a Demazure character.

\begin{definition}[\cite{Koh91}]
  A \newword{Kohnert move} on a diagram selects the rightmost cell of a given row and moves the cell to the first available position below, jumping over other cells in its way as needed.
\label{def:move}
\end{definition}

Given a weak composition $a$, the \newword{key diagram} $a$ as the set of left justified cells with $a_i$ in row $i$, indexed in cartesian coordinates. Fig.~\ref{fig:kohnert} shows all diagrams that can be obtained via Kohnert moves from the key diagram of $(0,3,2)$.

\begin{figure}[ht]
  \begin{center}
    \begin{tikzpicture}[scale=.75, every node/.style={scale=.75}, xscale=2.3,yscale=1.3]
      \node at (0,5) (A) {$\vline\cirtab{ \  & \  \\ \  & \  & \  \\ & }$};
      \node at (1,4) (B) {$\vline\cirtab{ \  & \  \\ \  & \  \\ & & \ }$};
      \node at (2,3) (C) {$\vline\cirtab{ \  & \  \\ \  \\ & \  & \  }$};
      \node at (3,2) (D) {$\vline\cirtab{ \  & \  \\ \\ \  & \  & \  }$};
      \node at (1,2) (E) {$\vline\cirtab{ \  \\ \  & \  \\ & \  & \  }$};
      \node at (2,1) (F) {$\vline\cirtab{ \  \\ & \  \\ \  & \  & \  }$};
      \node at (1,0) (G) {$\vline\cirtab{ \\ \  & \  \\ \  & \  & \  }$};
      \node at (0,3) (H) {$\vline\cirtab{ \  \\ \  & \  & \  \\ & \  }$};
      \node at (0,1) (I) {$\vline\cirtab{ \\ \  & \  & \  \\ \  & \  }$};
      \draw[thin] (A) -- (H) ;
      \draw[thin] (A) -- (B) ;
      \draw[thin] (H) -- (I) ;
      \draw[thin] (H) -- (E) ;
      \draw[thin] (B) -- (E) ;
      \draw[thin] (B) -- (C) ;
      \draw[thin] (I) -- (G) ;
      \draw[thin] (E) -- (G) ;
      \draw[thin] (C) -- (E) ;
      \draw[thin] (C) -- (D) ;
      \draw[thin] (D) -- (F) ;
      \draw[thin] (F) -- (G) ;
    \end{tikzpicture}
    \caption{\label{fig:kohnert}Iterative construction of Kohnert diagrams for $(0,3,2)$, where an edge down indicates the lower diagram can be obtained from the higher via a single Kohnert move.}
  \end{center}
\end{figure}
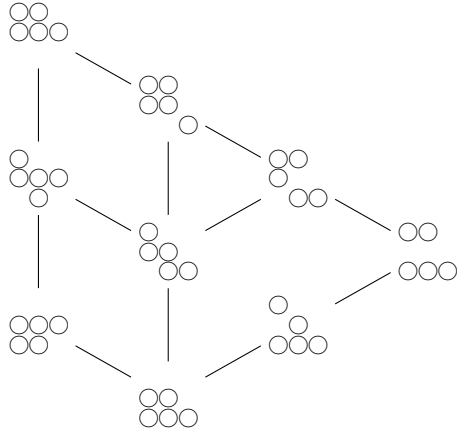

Denote the set of diagrams that can be obtained by Kohnert moves from the diagram of $a$ by $\KD(a)$. Note that there might be multiple ways to obtain a diagram from different Kohnert moves of a given diagram, but each resulting diagram is included in the set exactly once.

\begin{theorem}[\cite{Koh91}]
  The Demazure character $\key_a$ is given by
  \begin{equation}
    \key_{a} = \sum_{D \in \KD(a)} x_1^{\wt(D)_1} \cdots x_n^{\wt(D)_n} ,
    \label{e:key}
  \end{equation}
  where $\wt(D)$ is the weak compositions whose $i$th part is the number of cells in the $i$th row of $D$.
  \label{thm:key-SSKT}
\end{theorem}

The poset structure on Kohnert diagrams that arise for a key diagram is not a crystal structure, and Kohnert moves do not generally correspond to crystal moves. However, the Demazure crystal structure from \cite{ASc18} has a natural analog on Kohnert diagrams through the correspondence between diagrams and tableaux based on \cite{Ass-W}(Definition~3.14).

\begin{definition}
  The \newword{diagram map} $\D$ sends a nonattacking filling to a diagram by letting $\D(T)$ be the diagram with a cell in row $r$ and column $c$ if and only if $T$ has a cell with entry $r$ in column $c$.
  \label{def:diagram-map}
\end{definition}

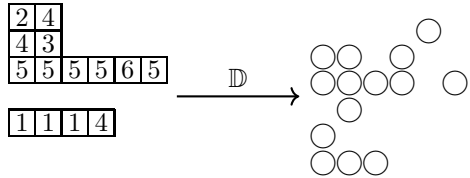
\begin{figure}[ht]
  \begin{center}
    \begin{tikzpicture}[xscale=4,yscale=1]
      \node at (1,1)   (D1) {$\vline\tableau{2 & 4 \\ 4 & 3 \\ 5 & 5 & 5 & 5 & 6 & 5 \\ \\ 1 & 1 & 1 & 4 \\ \\ & }$};
      \node at (2,1)   (D2) {$\vline\cirtab{ & & & & \ \\ \ & \ & & \ \\ \ & \ & \ & \ & & \ \\ & \ \\ \ \\ \ & \ & \  }$};
      \draw[thick,,->] (D1) -- (D2) node[midway,above] {$\D$} ;
    \end{tikzpicture}
  \end{center}
  \caption{\label{fig:diagram}An example of the diagram map $\D$ on a semistandard key tabloid.}
\end{figure}

Assaf \cite{Ass-W}(Theorem~3.15) shows that the diagram map is a bijection between Kohnert diagrams for $a$ and semistandard key tableaux for $a$. We translate the crystal operators under this bijection as follows.

\begin{definition}\label{def:Dpair}
  Given any diagram $D$ with $n\geq 1$ rows and integer $1\leq i < n$, define the \newword{vertical i-pairing} of $D$ as follows: $i$-pair any boxes in rows $i$ and $i+1$ that are located in the same column and then iteratively vertically $i$-pair any unpaired boxes in row $i+1$ with the leftmost unpaired box in row $i$ located in a column to its left whenever all the boxes in rows $i$ and $i+1$ in the columns between them are already vertically $i$-paired.
\end{definition}

\begin{example}
  The vertical $2$-pairing for the diagram in Fig.~\ref{fig:pairing-diagram} is indicated by shading in blue the cells in row $2$ that are $2$-paired with a cell in row $3$ strictly to its right, and the latter cells are shaded in red. The purple cell is the rightmost unpaired. Note that this diagram is precisely $\D(T)$ for $T$ the diagram in Fig.~\ref{fig:pairing}, and the two pairing rules correspond.
\end{example}

\begin{figure}[ht]
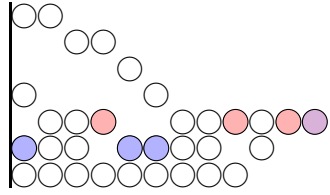

  \begin{center}
    \begin{displaymath}
      \vline\cirtab{ \ & \ \\ & & \ & \ \\ & & & & \ \\ \ & & & & & \ \\
        & \ & \ & \rball & & & \ & \ & \rball & \ & \rball & \vball \\
        \bball & \ & \ & & \bball & \bball & \ & \ & & \ \\
        \ & \ & \ & \ & \ & \ & \ & \ & \ }
    \end{displaymath}
  \end{center}
  \caption{\label{fig:pairing-diagram}An illustration of the pairing rule on diagrams.}
\end{figure}

\begin{definition}\label{def:Draise}
Given any integer $n\geq 0$ and any diagram $D$ with at most $n$ rows, for any integer $1\leq i < n$ define the \newword{raising operator} $\De_i$ on the space of diagrams as the operator that pushes the rightmost vertically unpaired box in row $i+1$ of $D$ down to row $i$. If $D$ has no vertically unpaired boxes in row $i+1$ then $\De_i(D)=0$.
\end{definition}

\begin{example}
  The leftmost diagram in Fig.~\ref{fig:raise-diagram} has two vertically unpaired cells in row $5$ (indicated as $\rball$), and so $\De_4$ acts by lowering these cells (resulting in $\bball$) from the right until none remains. These diagrams are precisely the images of the semistandard key tabloids in Fig.~\ref{fig:raise-ex} under the diagram map.
\end{example}

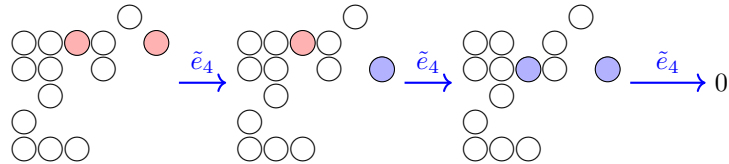
\begin{figure}[ht]
  \begin{center}
    \begin{tikzpicture}[xscale=3,yscale=3]
      \node at (0,1)   (D0) {$\vline\cirtab{ & & & & \ \\ \ & \ & \rball & \ & & \rball \\ \ & \ & & \ & & \\ & \ \\ \ \\ \ & \ & \ }$};
      \node at (1,1)   (D1) {$\vline\cirtab{ & & & & \ \\ \ & \ & \rball & \ \\ \ & \ & & \ & & \bball \\ & \ \\ \ \\ \ & \ & \  }$};
      \node at (2,1)   (D2) {$\vline\cirtab{ & & & & \ \\ \ & \ & & \ \\ \ & \ & \bball & \ & & \bball \\ & \ \\ \ \\ \ & \ & \  }$};
      \node at (2.8,1)   (D3) {$0$};
      \draw[thick,color=blue  ,->] (D0) -- (D1) node[midway,above] {$\De_4$} ;
      \draw[thick,color=blue  ,->] (D1) -- (D2) node[midway,above] {$\De_4$} ;
      \draw[thick,color=blue  ,->] (D2) -- (D3) node[midway,above] {$\De_4$} ;
    \end{tikzpicture}
  \end{center}
  \caption{\label{fig:raise-diagram}The images of the semistandard key tabloids in Fig.~9/Example~4.2 under the diagram map.}
\end{figure}

\begin{proposition}
  Let $T \in \SSKD(a)$ and suppose $e_i(T) \neq 0$. Then $\D(e_i(T)) = \De_i(\D(T))$. That is, the raising operators on tabloids and on diagrams coincide.
  \label{prop:commute}
\end{proposition}

\begin{proof}
  Suppose $T \in \SSKD(a)$ and $e_i(T)\neq 0$. Then $e_i$ will send the rightmost $i+1$ in $T$ to $i$ and flip all the entries with values $i$ and $i+1$ in certain affected columns of $T$. Since $\D$ sorts the cells in any fixed column of $T$ based on the value of their entries and not the row in which they are situated, then if $T, T' \in \SSKD(a)$ differ only by a flip of an $i$ and $i+1$ in a given column then $\D(T)=\D(T')$. Now, a quick comparison of Definitions~\ref{def:pair} and \ref{def:Dpair} makes it apparent that the rightmost unpaired $i+1$ in $T$ corresponds to the rightmost unpaired box in $\D(T)$. Thus, the column on which $\D$ acts corresponds to the only affected column of $T$ on which $e_i$ does not act by a flip but by sending the rightmost unpaired $i+1$ to an $i$. Thus, $\D(e_i(T))=\De_i(\D(T))$.
\end{proof}

In order to make use of the known Demazure crystal structure on Kohnert diagrams for a key diagram, we introduce a \newword{rectification map} that sends an arbitrary diagram to a Kohnert diagram for some key diagram. On the level of tabloids, rectification sends a semistandard key tabloid to a semistandard key tableau.

The key to understanding the rectification map is the following characterization stated in \cite{AS18}(Lemma~2.2).

\begin{lemma}[\cite{AS18}]
  A diagram $D$ can be obtained via a series of Kohnert moves on a key diagram if and only if for every position $(r,c)\in\mathbb{N} \times \mathbb{N}$ with $c>1$, we have
  \begin{equation}
    \#\{ (s,c-1) \in D \mid s \geq r \} \geq \#\{ (s,c) \in D \mid s \geq r \}.
    \label{e:rectified}
  \end{equation}
  \label{lem:diagram}
\end{lemma}

Recall the crystal flip map $\mathcal{F}$ from Definition~\ref{def:crystalflip} and consider the map from diagrams satisfying Lemma~\ref{lem:diagram} to semistandard Young tableaux that gives a partial inverse of the diagram map, based on \cite{Ass-W}(Definition~3.14).

\begin{definition}
  For fixed $n$, define the \newword{tableau map} $\T$ on diagrams $D$ with no cells above row $n$ satisfying Lemma~\ref{lem:diagram} as follows. Place entry $n-r+1$ in each cell of row $r$; drop and sort the cells of each columns to be bottom-justified and to increase from bottom to top; apply the crystal flip map $\mathcal{F}$. 
  \label{def:tableau-map}
\end{definition}

Note the first two steps of Definition~\ref{def:tableau-map} are equivalent to the \emph{column sorting map} of \cite{ASc18}(Definition~3.5), and so by \cite{ASc18}(Proposition~3.6), the result is a semistandard Young tableau. Therefore the crystal flip applies, making the tableau map of Definition~\ref{def:tableau-map} well-defined. For an example, see Fig.~\ref{fig:tabmap}.

\begin{figure}[ht]
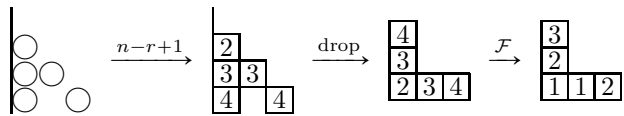

  \begin{displaymath}
    \hackcenter{\vline \cirtab{ \\   \ & &   \\ \ & \ & \\  \ &  & \  \\}}  \hspace{1ex} \xrightarrow{n-r+1}      \hspace{1ex} 
    \hackcenter{\vline \tableau{ \\   2 & &   \\ 3 & 3 & \\  4 &  & 4  \\}} \hspace{1ex} \xrightarrow{\text{drop}}\hspace{1ex} 
    \hackcenter{\vline \tableau{  4 \\ 3\\ 2&3&4\\}}                        \hspace{1ex} \xrightarrow{\mathcal{F}}\hspace{1ex} 
    \hackcenter{\vline \tableau{  3 \\ 2\\ 1&1&2\\}}
  \end{displaymath}
  \caption{\label{fig:tabmap}An illustration of the tableau map $\T$ with $n=4$.}
\end{figure}


An equivalent reformulation of Lemma~\ref{lem:diagram} is that the image under the tableau map of any diagram satisfying Eq.~\eqref{e:rectified} is a semistandard Young tableau of partition shape.
\smallskip

Recall from \S~\ref{sec:crystal-tableaux} that $\Ye$ denotes the raising operators on semistandard Young tableaux.

\begin{proposition}
  For any diagram $D$ satisfying Eq.~\eqref{e:rectified}, $\T(D)$ is a semistandard Young tableau of partition shape. Moreover, if $\tilde{e}_i(D) \neq 0$, then $\T(\tilde{e}_i(D)) = \Ye_i(\T(D))$.
  \label{prop:tableau-map}
\end{proposition}

\begin{proof}
  Generalizing \cite{AS18}(Definition~4.5) to the semistandard case, the column sorting map $\phi$ that lets cells of a semistandard key tableau fall vertically until it has partition shape, reverses entries by $i \mapsto n-i+1$, and sorts columns is an injection from $\SSKT(a)$ to $\SSYT(\lambda)$, where $\lambda$ is the partition rearrangement of $a$, by \cite{ASc18}(Proposition~3.6). By \cite{Ass-W}(Theorem~3.15), the diagram map $\D$ is a bijection between Kohnert diagrams for $a$ and semistandard key tableaux for $a$. Then the column sorting map may be factored as $\T\circ\D$, and so by Lemma~\ref{lem:diagram}, since any diagram $D$ satisfying Eq.~\eqref{e:rectified} may be identified with a semistandard key tableau of some shape $a$, we have $\T(D)$ is a semistandard Young tableau of shape the partition rearrangement of $a$. By \cite{ASc18}(Lemma~3.9), we have $\phi(e_i(T)) = \Yf_{n-i}(\phi(T))$, so composing with the crystal-reversing flip map gives the desired interwining.
\end{proof}

We utilize the characterization in Lemma~\ref{lem:diagram} to define a map from semistandard key tabloids to semistandard key tableaux and ultimately, by Proposition~\ref{prop:tableau-map}, to semistandard Young tableaux.

\begin{definition}\label{def:horiz-pair}
  Given any diagram $D$ with $n\geq 1$ columns and integer $1\leq i < n$, define the \newword{horizontal i-pairing} of $D$ as follows: $i$-pair any boxes in columns $i$ and $i+1$ that are located in the same row and then iteratively $i$-pair any unpaired box in column $i+1$ with the topmost unpaired box in column $i$ located in a row above it whenever all the boxes in columns $i$ and $i+1$ in the rows between them are already horizontally $i$-paired.
\end{definition}

\begin{remark}\label{rem:i-pairing}Notice that horizontal $i$-pairing is nothing more than a ``transposed" version of vertical $i$-pairing in Definition~\ref{def:Dpair} with the concept of rows and columns exchanged. That is, the concepts are equal under the mapping that sends a cell in position $(r,c)$ to $(c,r)$.
\end{remark}

\begin{definition}\label{def:Rect-raise}
Given any integer $n\geq 0$ and any diagram $D$ with at most $n$ columns, for any integer $1\leq i < n$, define the \newword{rectification operator} $\E_i$ on the space of diagrams as the operator which pushes the bottom-most horizontally unpaired box in column $i+1$ of $D$ left to column $i$. If $D$ has no unpaired boxes in column $i+1$ then $\E_i$ acts by zero.
\end{definition}

Unlike the raising operators, we will use the rectification operators in a prescribed way to map a given diagram to one that can be obtained from a composition diagram by a sequence of Kohnert moves.

\begin{definition}\label{def:rectification}
  Given a diagram $D$, define the \newword{rectification of $D$}, denoted by $\rectify(D)$, as follows. If $\E_i(D) = 0$ for all $i\geq 1$, then set $\rectify(D)=D$. Otherwise, finding the minimal column index $i \geq 1$ such that $\E_i(D)\neq 0$, replacing $D$ with $\E_i(D)$, and repeat.
\end{definition}


Key diagrams are a special case of diagrams that can arise, and we remark with the result below that they often correspond to extremal elements.

\begin{proposition}\label{prop:leftjustified}
  Given any key diagram $D$ with $\De_i(D) =0$ and $\Df_i(D)\neq 0$ for some $i$, then $\Df_i^*(D)$ is also a key diagram.
\end{proposition}

\begin{proof}
If $\De_i(D)=0$ and $\Df_i(D)\neq 0$ for some $i$ then necessarily $\wt(D)_i>\wt(D)_{i+1}$. Since $D$ is a key diagram, it is left justified so row $i$ must be strictly longer than row $i+1$. Thus all cells in row $i$ in columns with index greater than $\wt(D)_{i+1}$ will be vertically unpaired. If $k=\wt(D)_i-\wt(D)_{i+1}$, since $\Df_i$ acts on the leftmost vertically unpaired box in row $i$, then $\Df_i$ will act nontrivially on $D$ exactly $k$ times, pushing each of the cells in row $i$ and columns $\wt(D)_{i+1}+1$ through $\wt(D)_i$ up to row $i+1$ sequentially from left to right. Since $\Df_i$ will not affect any rows with index $j \neq i,i+1$ then $\Df_i^k(D)= \Df_i^*(D)$ will have $\wt(D)_i$ cells in row $i+1$ and $\wt(D)_{i+1}$ cells in row $i$ and, consequently, be left justified.
\end{proof}

The following lemma is the precursor to showing that rectification commutes with the crystal operators by showing that the pairing structures are respected.

\begin{lemma}\label{lem:PairingPreserve}
Rectification operators preserve vertical $i$-pairing. That is, given any diagram $D$ with $n$ columns, if $N_i(D)$ is the number of cells in row $i+1$ that are not vertically $i$-paired, then $N_i(D) = N_i$(\emph{$\E_c$}(D)) for any $1\leq c < n$.

Likewise, raising operators on diagrams preserve horizontal $i$-pairing. That is, given any diagram $D$ with $m$ rows, if $M_i(D)$ is the number of cells in column $i+1$ that are not horizontally $i$-paired, then $M_i(D) = M_i$(\emph{$\De_r$}(D)) for any $1\leq r < m$.
\end{lemma}

\begin{proof}
Since rectification operators act by pushing a cell one space to its left,  it suffices to consider all possible arrangements of cells in positions $(i,c),(i,c+1),(i+1,c)$, and $(i+1,c+1)$ on which $\E_c$ acts non-trivially and show that in each case the total number of vertically $i$-paired boxes is invariant under $\E_c$ for any $1 \leq c < n$.

If only one position in the arrangement is nonempty and $\E_c$ acts nontrivially, it is clear that pushing this cell left to column $c$ will not modify any existing vertical $i$-pairings.

If the arrangement consists of two nonempty cells, then either both cells lie in column $c+1$ or they lie in positions $(i+1,c+1)$ and $(i,c)$. If the first, then by Definition~\ref{def:horiz-pair} the bottom cell is horizontally paired only if the top cell is also, in which case $\E_c$ pushes the cell in position $(i,c+1)$ left which leaves its vertical $i$-pairing with the cell in $(i+1,c+1)$ unaltered.  The second situation follows from a dual argument to the first.
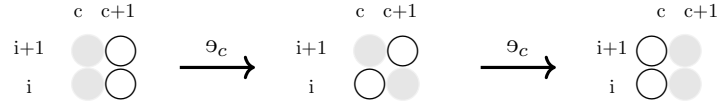
\[
\begin{tikzpicture}
\node at (-1,0.25)[scale=.75]{i+1};
\node at (-1,-0.25)[scale=.75]{i};
\node at (0,.75)[scale=.75]{c \; c+1};
\node at (0,0)[scale=1.25]{$\cirtab{ \lball & \ \\ \lball & \ }$};
\node at (1.5,.25)[scale=1]{$\E_c$};
\draw[->, very thick] (1,0)--(2,0);
\begin{scope}[shift={(.75,0)}]
\node at (2,0.25)[scale=.75]{i+1};
\node at (2,-0.25)[scale=.75]{i};
\node at (3,.75)[scale=.75]{c \; c+1};
 \node at (3,0)[scale=1.25]{$\cirtab{ \lball & \ \\ \ & \lball }$};
 \end{scope}
 \node at (5.5,.25)[scale=1]{$\E_c$};
 \draw[->, very thick] (5,0)--(6,0);
 \begin{scope}[shift={(7.75,0)}]
 \node at (-1,0.25)[scale=.75]{i+1};
\node at (-1,-0.25)[scale=.75]{i};
\node at (0,.75)[scale=.75]{c \; c+1};
\node at (-.25,0)[scale=1.25]{$\cirtab{  \ & \lball \\  \ & \lball}$};
\end{scope}
 \end{tikzpicture}
\]
If the arrangement consists of three nonempty cells then either the cell in position $(i,c+1)$ or in position $(i+1,c+1)$ can be pushed left. Although the specific cells that are paired with each other changes, the number of cells that were vertically paired or unpaired does not.  Thus, $N_i(D) = N_i(\E_c(D))$ for any $1 \leq c < n$.
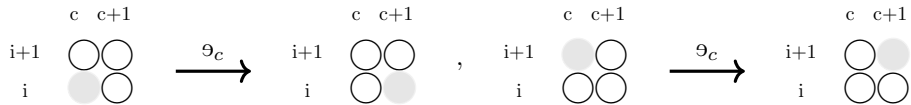
\[
\hackcenter{
\begin{tikzpicture}
\node at (-1,0.25)[scale=.75]{i+1};
\node at (-1,-0.25)[scale=.75]{i};
\node at (0,.75)[scale=.75]{c \; c+1};
\node at (0,0)[scale=1.25]{$\cirtab{  \ & \ \\  \lball & \ \\ }$};
\node at (1.5,.25)[scale=1]{$\E_c$};
\draw[->, very thick] (1,0)--(2,0);
\begin{scope}[shift={(.75,0)}]
\node at (2,0.25)[scale=.75]{i+1};
\node at (2,-0.25)[scale=.75]{i};
\node at (3,.75)[scale=.75]{c \; c+1};
 \node at (3,0)[scale=1.25]{$\cirtab{  \ & \ \\  \ & \lball}$};
 \end{scope}
 \end{tikzpicture}}
 \;\;\;\; , \;\;\;\;
\hackcenter{ \begin{tikzpicture}
\node at (-1,0.25)[scale=.75]{i+1};
\node at (-1,-0.25)[scale=.75]{i};
\node at (0,.75)[scale=.75]{c \; c+1};
\node at (0,0)[scale=1.25]{$\cirtab{  \lball & \ \\  \ & \ \\}$};
\node at (1.5,.25)[scale=1]{$\E_c$};
\draw[->, very thick] (1,0)--(2,0);
\begin{scope}[shift={(.75,0)}]
\node at (2,0.25)[scale=.75]{i+1};
\node at (2,-0.25)[scale=.75]{i};
\node at (3,.75)[scale=.75]{c \; c+1};
 \node at (3,0)[scale=1.25]{$\cirtab{  \ & \lball \\  \ & \ \\}$};
 \end{scope}
 \end{tikzpicture}}
\]

By Remark~\ref{rem:i-pairing} the statement that for any $i$ raising operators preserve the cells that are not horizontally $i$-paired follows identically from the work above by exchanging rows and columns.
\end{proof}

The following theorem is the key to establishing a Demazure crystal structure for nonsymmetric Macdonald polynomials, essentially by pulling the structure of semi-standard key tabloids back to semi-standard key tableaux.

\begin{theorem}
  The rectification operators and the raising operators on diagrams commute. That is, given any diagram $D$ for which $\E_c(D) \neq 0$ then for any row index $r \geq 1$, $\De_r(D) \neq 0$ if and only if $\De_r(\E_c(D)) \neq 0$. Likewise, if $\De_r(D) \neq 0$ then for any column index $c \geq 1$, $\E_c(D) \neq 0$ if and only if $\E_c(\De_r(D)) \neq 0$. In this case, we have $\E_c(\De_r(D)) = \De_r(\E_c(D))$ for all values of $r$ and $c$ for which $\E_c(D) \neq 0$ and $\De_r(D) \neq 0$.
  \label{thm:commute}
\end{theorem}

\begin{proof}
Suppose $\E_c$ acts on a diagram $D$ by pushing cell $(r+1,c+1)$ to $(r+1,c)$ and $\De_{r'}$ acts of $D$ by pushing cell $(r'+1,c'+1)$ to $(r',c'+1)$. Since the raising operators and rectification operators preserve the horizontal and vertical pairings, respectively, then the statements that $\De_r(D) \neq 0$ if and only if $\De_r(\E_c(D)) \neq 0$ and $\E_c(D) \neq 0$ if and only if $\E_c(\De_r(D)) \neq 0$ follow immediately from Lemma~\ref{lem:PairingPreserve}.

Since $\E_c$ only affects cells $(r+1,c+1)$ and $(r+1,c)$, and $\De_{r'}$ only affects cells $(r'+1,c'+1)$ and $(r', c'+1)$, and by Lemma~\ref{lem:PairingPreserve} their actions do not modify the respective pairings, then it is clear that if these four cells do not overlap in any way then the raising and rectification operators will commute. Thus, it suffices to check the following two cases.
\begin{itemize}
\item[Case 1:] Suppose $r=r'$ and $c=c'$ (see left diagram in Fig.~\ref{fig:commute}). That is, $\E_c$ and $\De_r$ both act on $D$ by pushing cell $(r+1,c+1)$ left and down, respectively. Since $\E_c(D) \neq 0$ and $\De_r(D) \neq 0$ then positions $(r,c+1)$ and $(r+1,c)$ must both be empty. In particular, this implies that $(r,c)$ must also be empty, since otherwise the cell in $(r+1,c+1)$ would pair vertically with it. Thus, $\E_c(D)$ sends $(r+1,c+1)$ to $(r+1,c)$. Since $(r,c)$ is empty and all cells in row $r$ to the left of column $c$ were vertically $r$-paired, then $(r+1,c)$ is actually the rightmost vertically unpaired cell in row $r+1$. Thus, $\De_r(\E_c(D))$ pushes $(r+1,c)$ down to position $(r,c)$. If instead we act of $D$ by $\De_r$ first, then the cell in $(r+1,c+1)$ is first sent to $(r,c+1)$. Once again, since $(r,c)$ is empty and all cells in column $c$ in rows higher than $r$ are horizontally $c$-paired then the cell in $(r,c+1)$ is the bottom-most cell in column $c+1$ that is horizontally unpaired. Thus, $E_c$ pushes $(r,c+1)$ left to position $(r,c)$. Since all other cells of $D$ remain in the same exact positions after applying $E_c$ and $\De_r$, then we see that in this situation $\E_c\De_r(D)=\De_r\E_c(D)$.

\item[Case 2:] Suppose instead that $r=r'+1$ and $c=c'+1$ (see right diagram in Fig.~\ref{fig:commute}). That is, $\E_c$ sends $(r,c+1)$ to $(r,c)$ and $\De_r$ sends $(r,c+1)$ to $(r,c)$. Since $\E_c(D) \neq 0$ then there must be a cell in position $(r+1,c+1)$. Otherwise, $(r,c+1)$ would be horizontally $c$-paired with $(r+1,c)$ which contradicts our assumptions. Now, $\E_c$ acts on $D$ by pushing $(r,c+1)$ left to position $(r,c)$. Since in $D$ the cell in $(r+1,c)$ was the rightmost vertically unpaired cell in row $r+1$, then every cell in row $r$ in a column left of $c$ must be vertically $r$-paired. Hence, in $\E_c(D)$ the rightmost unpaired cell in row $r+1$ is located in position $(r+1,c+1)$. Consequently, $\De_r$ acts on $\E_c(D)$ by pushing $(r+1,c+1)$ down to position $(r,c+1)$. If instead $\De_r$ acts first, then $(r+1,c)$ is pushed down to $(r,c)$. This time, since every cell in column $c$ and row higher than $r+1$ is horizontally $c$-paired, then the bottom most horizontally unpaired cell in column $c+1$ of $\De_r(D)$ is in position $(r+1,c+1)$. Thus, $\E_c$ acts on $\De_r(D)$ by pushing this cell left to $(r+1,c)$. Since in both compositions the resulting diagrams have cells in positions $(r+1,c), (r,c)$ and $(r,c+1)$ and no cell in position $(r+1,c+1)$, then as before $\E_c\De_r(D)=\De_r\E_c(D)$.

\end{itemize}
\end{proof}

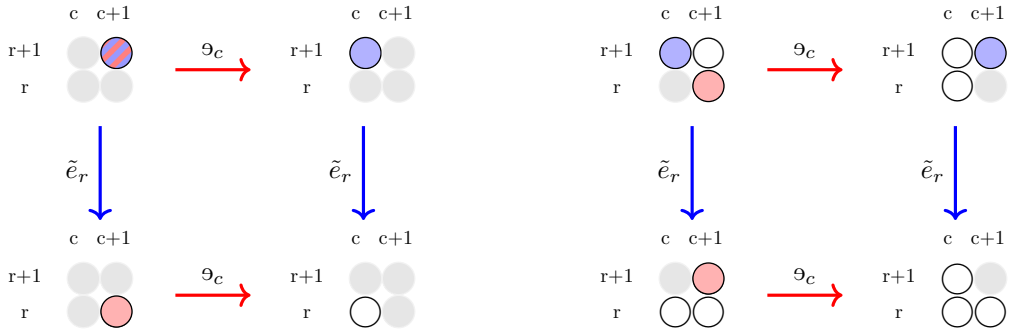
\begin{figure}[ht]
\begin{center}
\begin{tikzpicture}
\node at (-1,0.25)[scale=.75]{r+1};
\node at (-1,-0.25)[scale=.75]{r};
\node at (0,.75)[scale=.75]{c \; c+1};
\node at (0,0)[scale=1.25]{$\cirtab{ \lball & \sball \\ \lball & \lball }$};
\node at (1.5,.25)[scale=1]{$\E_c$};
\draw[->, red, very thick] (1,0)--(2,0);
\begin{scope}[shift={(.75,0)}]
\node at (2,0.25)[scale=.75]{r+1};
\node at (2,-0.25)[scale=.75]{r};
\node at (3,.75)[scale=.75]{c \; c+1};
 \node at (3,0)[scale=1.25]{$\cirtab{ \bball & \lball \\ \lball&\lball }$};
 \end{scope}
 \draw[->, blue, very thick] (0,-.75)--(0,-2);
\node at (-.3,-1.35)[scale=1]{$\De_r$};
\draw[->,blue, very thick] (3.5,-.75)--(3.5,-2);
\node at (3.2,-1.35)[scale=1]{$\De_r$};
 \begin{scope}[shift={(0,-3)}]
\node at (-1,0.25)[scale=.75]{r+1};
\node at (-1,-0.25)[scale=.75]{r};
\node at (0,.75)[scale=.75]{c \; c+1};
\node at (0,0)[scale=1.25]{$\cirtab{ \lball& \lball \\ \lball & \rball }$};
\node at (1.5,.25)[scale=1]{$\E_c$};
\draw[->,red, very thick] (1,0)--(2,0);
\begin{scope}[shift={(.75,0)}]
\node at (2,0.25)[scale=.75]{r+1};
\node at (2,-0.25)[scale=.75]{r};
\node at (3,.75)[scale=.75]{c \; c+1};
 \node at (3,0)[scale=1.25]{$\cirtab{ \lball &\lball \\ \ & \lball }$};
 \end{scope}
 \end{scope}
 \end{tikzpicture}
 \qquad \qquad\qquad
 \begin{tikzpicture}
\node at (-1,0.25)[scale=.75]{r+1};
\node at (-1,-0.25)[scale=.75]{r};
\node at (0,.75)[scale=.75]{c \; c+1};
\node at (0,0)[scale=1.25]{$\cirtab{ \bball & \ \\ \lball & \rball }$};
\node at (1.5,.25)[scale=1]{$\E_c$};
\draw[->, red, very thick] (1,0)--(2,0);
\begin{scope}[shift={(.75,0)}]
\node at (2,0.25)[scale=.75]{r+1};
\node at (2,-0.25)[scale=.75]{r};
\node at (3,.75)[scale=.75]{c \; c+1};
 \node at (3,0)[scale=1.25]{$\cirtab{ \ & \bball \\ \ & \lball }$};
 \end{scope}
 \draw[->, blue, very thick] (0,-.75)--(0,-2);
\node at (-.3,-1.35)[scale=1]{$\De_r$};
\draw[->,blue, very thick] (3.5,-.75)--(3.5,-2);
\node at (3.2,-1.35)[scale=1]{$\De_r$};
 \begin{scope}[shift={(0,-3)}]
\node at (-1,0.25)[scale=.75]{r+1};
\node at (-1,-0.25)[scale=.75]{r};
\node at (0,.75)[scale=.75]{c \; c+1};
\node at (0,0)[scale=1.25]{$\cirtab{ \lball & \rball \\ \ & \ }$};
\node at (1.5,.25)[scale=1]{$\E_c$};
\draw[->,red, very thick] (1,0)--(2,0);
\begin{scope}[shift={(.75,0)}]
\node at (2,0.25)[scale=.75]{r+1};
\node at (2,-0.25)[scale=.75]{r};
\node at (3,.75)[scale=.75]{c \; c+1};
 \node at (3,0)[scale=1.25]{$\cirtab{ \ & \lball \\ \ & \ }$};
 \end{scope}
 \end{scope}
 \end{tikzpicture}
 \caption{Diagrams depicting the operations described in Case 1 (left) and Case 2 (right) of the proof of Theorem~\ref{thm:commute}. At each step the cells on which the raising and rectification operators act are colored blue and red, respectively. The light gray circles are meant to denote empty spaces and are included only to clarify the relative position of the filled cell.} \label{fig:commute}
 \end{center}
 \end{figure}

\begin{corollary}
  The raising operators commute with rectification.
  That is, given any diagram $D$ and any row index $r \geq 1$, $\De_r(D) \neq 0$ if and only if $\De_r(\rectify(D)) \neq 0$. In this situation, we have $\rectify(\De_r(D)) = \De_r(\rectify(D))$.
  \label{cor:commute}
\end{corollary}

\begin{proof}
Suppose $\De_r$ acts on $D$ by moving the cell in position $(r+1,c)$ to position $(r,c)$.  If $\rectify(D)=D$ then it suffices to show that $\rectify(\De_r(D))=\De_r(D)$. In particular, by Lemma~\ref{lem:PairingPreserve} we know that raising operators preserve horizontal pairing. Thus, if $\E_c(D)=0$ for all $c\geq 1$, then $\E_c(\De_r(D))=0$ for all $c\geq 1$. Consequently, $\rectify(\De_r(D))=\De_r(D)$.

If instead $\rectify(D)\neq D$, then there must exist some $c\geq 1$ such that $\E_c(D) \neq 0$. By Theorem~\ref{thm:commute} we know that whenever $\E_c(D) \neq 0$ and $\De_r(D) \neq 0$ then these operators commute. Thus, it immediately follows that $\rectify(\De_r(D)) = \De_r(\rectify(D))$, as desired.
\end{proof}

For the sake of conciseness, we introduce the following notation.

\begin{definition}
  The \newword{embedding map} $\embed: \SSKD \to \SSYT$ is the composition of the maps $\T \circ \rectify \circ \D$.
  \label{def:ram}
\end{definition}

With this in hand, we combine the previous results and show that $\embed$ is in fact a crystal homomorphism from $\SSKD(a)$ into $\SSYT$, that is, $\embed$ preserves the crystal structures. 

\begin{corollary}
  Let $a$ be a weak composition of length $n$, and let $\mathcal{C}\subseteq\SSKD(a)$ be any subset closed under the raising and lowering operators on semistandard key tabloids. Then there exists a partition $\lambda$ such that $\embed(\mathcal{C})\subseteq\SSYT_n(\lambda)$. Moreover, for any $T \in \SSKD(a)$ such that $e_i(T) \neq 0$, we have
  \[\embed(e_i(T))= \Ye_i(\embed(T)).\]
  \label{cor:commute2}
\end{corollary}

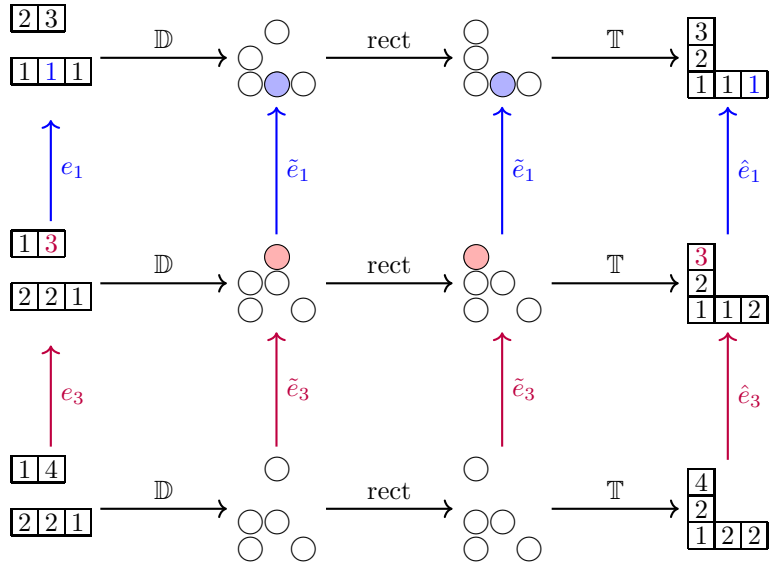
\begin{figure}[ht]
\begin{tikzpicture}
\node at (0,0) (a){\vline\tableau{2& 3 \\ & \\ 1&\color{blue}1&1\\ & }};
\node at (0,-3) (b){\vline\tableau{1& \color{purple}3 \\ & \\ 2&2&1\\ & }};
\node at (0,-6) (c){\vline\tableau{1& 4 \\ & \\ 2&2&1\\ & }};
\draw[thick, blue, <-] (a)--(b) node[midway,right] {$e_1$};
\draw[thick, purple, <- ](b)--(c) node[midway,right] {$e_3$};

\node at (3,0) (a2){\vline\cirtab{& \ \\ \ \\ \ & \bball &\ \\}};
\node at (3,-3) (b2){\vline\cirtab{& \rball \\ \ & \ \\ \ & & \ \\}};
\node at (3,-6) (c2){\vline\cirtab{& \ \\ \\ \ & \ \\ \ & & \ \\ }};
\draw[thick, blue, <-] (a2)--(b2) node[midway,right] {$\De_1$};
\draw[thick, purple, <-](b2)--(c2) node[midway,right] {$\De_3$};

\node at (6,0) (a3){\vline\cirtab{ \ \\ \ \\ \ & \bball &\ \\}};
\node at (6,-3) (b3){\vline\cirtab{ \rball \\ \ & \ \\ \ & & \ \\}};
\node at (6,-6) (c3){\vline\cirtab{ \ \\ \\ \ & \ \\ \ & & \ \\ }};
\draw[thick, blue, <-] (a3)--(b3) node[midway,right] {$\De_1$};
\draw[thick, purple, <-](b3)--(c3) node[midway,right] {$\De_3$};

\node at (9,0) (a4){\tableau{3\\2\\ 1&1& \color{blue}1}};
\node at (9,-3) (b4){\tableau{\color{purple}3 \\2\\ 1&1&2}};
\node at (9,-6) (c4){\tableau{4\\2\\ 1&2&2}};
\draw[thick, blue, <-] (a4)--(b4) node[midway,right] {$\Ye_1$};
\draw[thick, purple, <-](b4)--(c4) node[midway,right] {$\Ye_3$};

\draw[thick, ->] (a)--(a2) node[midway,above] {$\D$};
\draw[thick, ->] (b)--(b2) node[midway,above] {$\D$};
\draw[thick, ->] (c)--(c2) node[midway,above] {$\D$};

\draw[thick, ->] (a2)--(a3) node[midway,above] {$\rectify$};
\draw[thick, ->] (b2)--(b3) node[midway,above] {$\rectify$};
\draw[thick, ->] (c2)--(c3) node[midway,above] {$\rectify$};

\draw[thick, ->] (a3)--(a4) node[midway,above] {$\T$};
\draw[thick, ->] (b3)--(b4) node[midway,above] {$\T$};
\draw[thick, ->] (c3)--(c4) node[midway,above] {$\T$};


\end{tikzpicture}
\caption{\label{fig:rectify}Examples of the map $\embed$ given by the rectification algorithm from semi-standard key tabloids, to diagrams, to Kohnert diagrams (via rectification), to semi-standard Young tableaux.}
\end{figure}

In particular, each connected component of the graph determined by the raising operators on semistandard key tabloids is a subset of a normal crystal.

\begin{example} In Fig.~\ref{fig:rectify} we see a detailed example of the embedding map $\embed$ acting on some of the elements of the Demazure crystal $\B_{4321}(3,1,1,0)$. The colored entries and balls denote the unpaired cells on which the raising operators act. For a detailed example of the entire Demazure crystal $\B_{4321}(3,1,1,0)$ we refer the reader to Fig.~\ref{fig:embedExample} in the Appendix.
\end{example}

\subsection{Demazure property}
\label{sec:tabloid-lowest}

We leverage the tools developed in Section~\ref{sec:demazure} to show that rectification is, in fact, a crystal isomorphism between the graph determined by raising operators on semistandard key tabloids and the Demazure crystal on semistandard key tableaux.

To begin, we must show the graph is extremal, as in Definition~\ref{def:extremal}. In particular, we must show each component contains the necessary highest weight element. To that end, we have the following.

\begin{lemma}
  If $T \in \SSKD(a)$ is such that $e_i(T)=0$ for all $i$, then $\rectify(\D(T))$ is a key diagram with partition weight.
  \label{lem:hwPartition}
\end{lemma}

\begin{proof}
If $e_i(T)=0$ for all $i$ then clearly all cells with entries $i+i$ are $i$-paired and so $\wt(T)_{i+1}\leq \wt(T)_i$ for all $i$.  Thus, $\wt(T)$ is a partition.  Consequently, it suffices to show that $\rectify(\D(T))$ is left justified.

By Lemma~\ref{lem:PairingPreserve} rectification preserves the number of vertical $i$-pairs, therefore all the cells in row $i+1$ of $\rectify(\D(T))$ must be vertically $i$-paired. That is to say, for any cell in row $i+1$ there is a cell in row $i$ located either below it or in a column to its left.  In particular, in column one this implies all cells must lie in a consecutive block of rows with indexes $1\leq \dots \leq R_{1_1}$. This in turn forces the cells in column two to lie in at most two blocks of consecutive rows with indexes $1\leq \dots \leq R_{2_1}$ and $R_{1_1}+1\leq \dots \leq R_{2_2}$ satisfying $R_{2_1}\leq R_{1_1}$. Iterating this procedure we find that column $c+1$ must have cells in at most $c+1$ blocks of consecutive rows with indexes $1\leq \dots \leq R_{(c+1)_1}$, $R_{c_1}+1\leq \dots \leq R_{(c+1)_2}$, $\dots$, $R_{c_{c}}+1 \leq \dots \leq R_{(c+1)_{(c+1)}}$ satisfying $R_{(c+1)_j} \leq R_{c_j}$ for all $1\leq j \leq c$ (see Fig.~\ref{fig:HWrectification}).

Since rectification acts on the column with the lowest possible index and on the lowest row of the affected column, then if $\D(T)$ has cells in column two that have row index higher than $R_{1_1}$ then the first rectification operator to act will be $\E_1$. Specifically, since column one has no cells in rows higher than $R_{1_1}$ then every cell in column two in rows $R_{1_1}+1\leq \dots \leq R_{2_2}$ is not horizontally 1-paired. Thus, $\E_1$ will act $R_{2_2}-R_{1_1}$ times on $\D(T)$ and push all the cells of column two in rows $R_{1_1}+1\leq \dots \leq R_{2_2}$ left to column one, left justifying the first two columns in the process. Since all the cells in column two now lie in consecutive rows $1 \leq \dots \leq R_{2_1}$,  rectification will now act by applying $\E_2$ to $\E_1^{R_{2_2}-R_{1_1}}(\D(T))$ exactly $(R_{3_2}-R_{2_1}) + (R_{3_3}-R_{2_2})$ times and left justifying columns two and three. Iterating this procedure if we set $m_s(c+1):= \sum_{i=1}^{s} (R_{(c+1)_{i+1}}-R_{c_i})$ and $\reflectbox{R}_c:= \E_1^{m_1(c+1)}\E_s^{m_2(c+1)} \dots \E_c^{m_c(c+1)}$, then rectification will act on $\D(T)$ in the following manner:
\[\rectify(\D(T))  = \reflectbox{R}_{M-1} \reflectbox{R}_{M-2} \dots \reflectbox{R}_{1} (\D(T)),\]
where $M$ is the number of columns of $\D(T)$. Thus for each $1\leq c \leq M-1$, the diagram $\reflectbox{R}_{c} \reflectbox{R}_{M-2} \dots \reflectbox{R}_{1} (\D(T))$ has the same cells as $\D(T)$ but with the first $c+1$ columns left justified and all columns to the right of column $c+1$ identical to those of $\D(T)$.  Hence, rectification will  sequentially left justify the first $c$ columns of $\D(T)$ with $c$ increasing one step at a time and so $\D(T)$  is rectified to a key diagram of partition weight.
\end{proof}

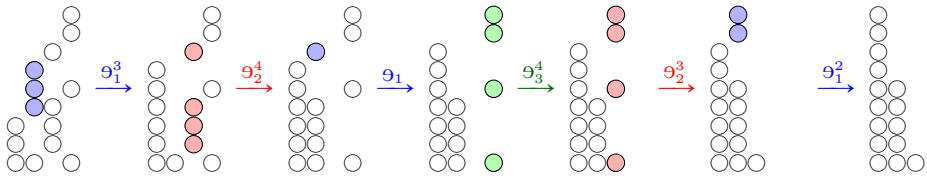
\begin{figure}[ht]
\[
\hackcenter{
\begin{tikzpicture}[scale=.7,every node/.style={scale=.7}]
\node at (0,0){$\vline \cirtab{ &&&\ \\ &&& \ \\ && \ \\ & \bball \\ & \bball & & \ \\ & \bball & \ \\ \ && \ \\ \ & & \ \\ \ & \ & & \ \\}$};
\end{tikzpicture}}
\xrightarrow{\color{blue}\E_1^3}
\hackcenter{
\begin{tikzpicture}[scale=.7,every node/.style={scale=.7}]
\node at (0,0){$\vline \cirtab{ &&&\ \\ &&& \ \\ && \rball \\  \ & \\  \ & & & \ \\  \ & & \rball \\ \ && \rball \\ \ & & \rball \\ \ & \ & & \ \\}$};
\end{tikzpicture}}
\xrightarrow{\color{red}\E_2^4}
\hackcenter{
\begin{tikzpicture}[scale=.7,every node/.style={scale=.7}]
\node at (0,0){$\vline \cirtab{ &&&\ \\ &&& \ \\ & \bball & \\  \ & \\  \ & & & \ \\  \ &  \ & \\ \ & \ & \\ \ &  \ & \\ \ & \ & & \ \\}$};
\end{tikzpicture}}
\xrightarrow{\color{blue}\E_1}
\hackcenter{
\begin{tikzpicture}[scale=.7,every node/.style={scale=.7}]
\node at (0,0){$\vline \cirtab{ &&& \gball \\ &&& \gball \\  \ & & \\  \ & \\  \ & & & \gball \\  \ &  \ & \\ \ & \ & \\ \ &  \ & \\ \ & \ & & \gball \\}$};
\end{tikzpicture}}
\xrightarrow{\color{black!60!green}\E_3^4}
\hackcenter{
\begin{tikzpicture}[scale=.7,every node/.style={scale=.7}]
\node at (0,0){$\vline \cirtab{ && \rball & \\ && \rball & \\  \ & & \\  \ & \\  \ & & \rball & \\  \ &  \ & \\ \ & \ & \\ \ &  \ & \\ \ & \ &  \rball &\\}$};
\end{tikzpicture}}
\xrightarrow{\color{red}\E_2^3}
\hackcenter{
\begin{tikzpicture}[scale=.7,every node/.style={scale=.7}]
\node at (0,0){$\vline \cirtab{ & \bball & & \\ & \bball & & \\  \ & & \\  \ & \\  \ &  \ & & \\  \ &  \ & \\ \ & \ & \\ \ &  \ & \\ \ & \ & \ & &\\}$};
\end{tikzpicture}}
\xrightarrow{\color{blue}\E_1^2}
\hackcenter{
\begin{tikzpicture}[scale=.7,every node/.style={scale=.7}]
\node at (0,0){$\vline \cirtab{ \ & & & \\  \ & & & \\  \ & & \\  \ & \\  \ &  \ & & \\  \ &  \ & \\ \ & \ & \\ \ &  \ & \\ \ & \ & \ & &\\}$};
\end{tikzpicture}}\]
\caption{Example of rectification of a diagram where $\De_i$ acts trivially for all $i$.}\label{fig:HWrectification}
\end{figure}

\begin{example}
  In Fig.~\ref{fig:HWrectification} we can see how a diagram $Y$ satisfying $\De_i(Y)=0$ for all $i$ is rectified to a partition diagram. In particular, $\reflectbox{R}_1=\E_1^3$, $\reflectbox{R}_2=\E_1\E_2^4$, $\reflectbox{R}_3=\E_1^2\E_2^3\E_3^4$ so that $\rectify(\D(Y))= \reflectbox{R}_3\reflectbox{R}_2\reflectbox{R}_1(Y)$.
\end{example}

\begin{lemma}
  Each connected component of the graph on $\SSKD(a)$ defined by the raising operators $e_i$ has a unique tabloid $Y$ such that $e_i(Y)=0$ for all $i$.
  \label{lem:hwUnique}
\end{lemma}

\begin{proof}
Assume $Y$ we not unique. Since the graph is connected there is a vertex $X$ and index $i$ such that $e_i(X)=Y$. Suppose there exists $e_j(X) \neq 0$ for some $j\neq i$. Recall that we can apply rectification and diagram maps to each vertex and embed the graph into a crystal. Denote by $D_X=\rectify(\D(X))$, then by Theorem~\ref{thm:commute}
we know that $\De_i(D_X) = D_Y$ and $\De_j(D_X)=D_{e_j(X)}$. Moreover, by Stembridge's crystal axioms the vertices $D_Y$ and $D_{e_j(X)}$ will be part of either a commuting square or an $sl_2$ relation. Either way, $\De_j(D_Y)\neq 0$. By Corollary~\ref{cor:commute2} there exists $W \in \SSKD(a)$ such that $D_W = \De_j(D_Y)$ and thus $e_j(Y)= W$ which contradicts the original assumption.

More generally, since the graph is connected we can assume there exists an a vertex $X$ with $\wt(X)<wt(Y)$ such that $e_{i_k}^{m_{i_k}}\dots e_{i_1}^{m_{i_1}}(X)=Y$ for some indexes $i_s$ and such that $e_j(X)\neq 0$ for some $j\neq i_1$. By considering a vertex $X$ with the previous properties and $k$ minimal we can iteratively apply the previous argument and obtain an identical contradiction.
\end{proof}

\begin{theorem}
  Each connected component of the graph on $\SSKD(a)$ defined by the raising operators $e_i$ is an extremal subcrystal of a normal crystal.
  \label{thm:extremal}
\end{theorem}

\begin{proof}
  Let $\mathcal{C}$ denote a connected component of the graph on $\SSKD(a)$ defined by the raising operators $e_i$. By Corollary~\ref{cor:commute2}, $\embed(\mathcal{C}) \subset \SSYT_n(\lambda)$ for some partition $\lambda$. We will show that $\mathcal{C}$ is an extremal subset of $\B(\lambda)$. By Lemma~\ref{lem:hwUnique}, there is a unique $Y\in\mathcal{C}$ such that $e_i(Y)=0$ for all $i$. By Lemma~\ref{lem:hwPartition}, $\embed(Y)$ is the highest weight in $\B(\lambda)$, where $\lambda = \wt(Y)$ is the partition weight of $\rectify(\D(Y))$. In particular, $\mathcal{C}$ contains the highest weight of $\B(\lambda)$, proving condition (1) of Definition~\ref{def:extremal}. By definition, $\mathcal{C}$ is closed under $e_i$, proving condition (2).

  Finally, to show condition (3) we note that if $x \in \mathcal{C}$ and $f_i(x) \neq 0$, then by definition $f_i(x) \in \mathcal{C}$. Thus, suppose $f_i(x) =0$ but that both $\Yf_i(\embed(x))\neq 0$ and $\Ye_i(\embed(x)) \neq 0$. If every cell with value $i$ of $x$ is $i$-paired, then every box in row $i$ of $\D(x)$ will be vertically $i$-paired. By Lemma~\ref{lem:PairingPreserve} it follows that every cell in row $i$ of $\rectify(\D(x))$ will also be vertically $i$-paired. Finally, it is straight forward to see that the tableau map $\T$ also preserves the number of $i$-paired entries with value $i$. Hence, if $f_i(x) =0$ because $x$ contains no unpaired $i$, then $\Yf_i(\embed(x))=0$ which contradicts the assumptions. Hence, $f_i(x) =0$ due to the Demazure condition. In this case, the leftmost unpaired $i$ of $x$ must lie in column $1$ and row $i$. However, this implies that $x$ contains no cells with value $i+1$ which are not $i$-paired. Thus $e_i(x) =0$. However, by Lemma~\ref{lem:PairingPreserve} and an analogous argument to the one above, this implies that $\Ye_i(\embed(x)) = 0$, which cannot be. Thus, if both $\Yf_i(\embed(x)) \neq 0$ and $\Ye_i(\embed(x)) \neq 0$, then there exists $e_i(x), f_i(x) \in \mathcal{C}$ such that $f_i(x)=\Yf_i(\embed(x))$ and $e_i(x)=\Ye_i(\embed(x))$.
\end{proof}

Finally, we prove the embedded subset is Demazure.

\begin{theorem}
  Each connected component of the graph on $\SSKD(a)$ defined by the raising operators $e_i$ is a Demazure subcrystal of a normal crystal.
  \label{thm:demazure}
\end{theorem}

\begin{proof}
Let $\mathcal{C}$ denote the connected component of the graph on $\SSKD(a)$ defined by raising operators $e_i$. By Theorem~\ref{thm:extremal} we know $\mathcal{C}$ is an extremal subset. Thus, it remains to show that $\mathcal{C}$ satisfies conditions $(4)-(6)$ of Definition~\ref{def:demazure}. Since $\mathcal{C} \subset \B(\lambda)$, for any $x \in \mathcal{C}$ it makes sense to consider the operator $\varphi_i(x)$ from Definition~\ref{def:base}. Recall from equation \eqref{eq:Fstring} that $\varphi_i(x)$ equals the number of cells with value $i$ which are not $i$-paired.

In particular, if $\varphi_i(x)>0$ and $f_i(x) \neq 0$ then
\begin{align*}
\varphi_{j}(f_i(x))= \begin{cases} \varphi_j(x) & |i-j|\geq 2\\
 \varphi_j(x) \;\text{or} \; \varphi_j(x)+1  & |i-j|=1\\
\varphi_j(x)-1 & i=j. \end{cases}
\end{align*}

Moreover, since $\mathcal{C}$ is an extremal subset of $\B(\lambda)$, then by condition $(3)$ of Definition~\ref{def:extremal}, if $|i-j|=1$ and both $\varphi_i(x), \varphi_j(x)>0$ then $\varphi_i(f_{j}^*(x))=\varphi_i(x)+\varphi_j(x)>1.$ We prove each condition of Definition~\ref{def:demazure} individually.
\begin{itemize}

\item[5a)] First, suppose that $x \in \mathcal{C}$ is extremal and $f_i^*f_{i+1}^*(x) \in \mathcal{C}$. If $\varphi_1(x)>0$ then $f_i(x)=0$ only if $x$ has a Demazure death for $i$. Hence, there is a column $c$ of $x$ containing the leftmost $i$ not paired with an $i+1$ such that all columns strictly to its left have an $i$ in row $i$ and an $i+1$ in a row above it. However, by $(iii)$ of Proposition~\ref{prop:raising-distribution} this immediately yields a contradiction. This is because even if there exists an $i+1$ that is not $i+2$-paired in a column $c'\leq c$ of $x$, the resulting element $f_{i+1}^*(x)$ will still contain an $i$ in row $i$ and column $c'$ that is not $i+1$ paired and will cause $f_if_{i+1}^*(x)=0$. Hence, $f_i(x)\neq0$. An analogous argument shows that if $f_i^*f_{i-1}^*(x) \in \mathcal{C}$ and $\varphi_i(x)>0$ then also $f_i(x) \in \mathcal{C}$. This proves condition $(5a)$ in Definition~\ref{def:demazure}.

\item[5b)] Now suppose that $x,y \in \mathcal{C}$ are extremal and that $e_i^*(x)=e_{i+1}^*(y) = u$ for some $u \in \mathcal{C}$ also extremal. Since $\varphi_i(u),\varphi_{i+1}(u)>0$ we know that $\varphi_{i+1}(x), \varphi_i(y)>1$. Thus, if $f_{i+1}(x)=0$ it must be due to Demazure death. We will show that if this is the case then $f_i(y)\neq 0$. So suppose $x$ has a continuous sequence of $i+1$'s in row $i+1$ and columns $1,\dots, c$, all of which are $i+1$-paired with an $i+2$ above them except for the $i+1$ in column $c$. Since $e_i^*(x)=u$ and $f_{i+1}(u) \neq 0$ then $x$ contains a consecutive sequence of $i's$ in situated above the $i+1'$ in columns $1,\dots, c'-1$ for some $c'\leq c$. If $c'<c$ then this would imply that column $c'$ of $u$ contains an $i$ that is not paired with an $i+1$ and thus, an $i+2$ that is also not $i+1$-paired. But $u$ is extremal and since $f_{i+1}(u) \neq 0$ then $e_{i+2}(u)=0$. Thus $c'=c$ and so column $c$ of $u$ has an $i$ that is not $i+1$-paired and whose columns $1,\dots, c-1$ each contain an $i$ in row $i+1$ which are paired with an $i+1$ above them, which in turn is also paired with an $i+2$ in a higher row. Therefore, $f_{i+1}(u)$ leaves columns $1,\dots,c$ untouched, and thus $f_i(y)=f_i(f_{i+1}^*(u))\neq 0$. The remaining case when $f_i(y)=0$ instead follows analogously. This proves the first part of $(5b)$ from Definition~\ref{def:demazure}.

Now suppose that neither $f_i(y)$ nor $f_{i+1}(x)$ are zero. Since $\mathcal{C} \subset \B(\lambda)$ is an extremal subset, then $\varphi_{i+1}(f_i^*(y))=\varphi_i(u)>0$ and $\varphi_i(f_{i+1}^*(x))=\varphi_{i+1}(u)>0$. Thus, if either $f_if_{i+1}^*(x)=0$ or $f_{i+1}f_i^*(y)=0$ it is due to Demazure death. As before, suppose this is the case and columns $1,\dots, c-1$ of $f_if_{i+1}^*(x)$ have a continuous sequence of $i$'s in row $i$ which are all $i$-paired with $i+1$'s above them and that column $c$ contains the leftmost $i$ which is not $i+1$-paired. If $col_{1,\dots, c}(x)=col_{1,\dots,c}(f_{i+1}^*(x))$ then $f_i(x)=0$ due to Demazure death, which is impossible since $e_i(x)\neq 0$ and $\mathcal{C}$ contains full $i$-strings. Hence, $col_{1,\dots, c'}(x)\neq col_{1,\dots,c'}(f_{i+1}^*(x)$ for some column $c'\leq c$. By an identical argument to the one above, it follows that $c'=c$ and so columns $1,\dots,c-1$ of $x$ have a continuous sequence of $i$'s in row $i$, a continuous sequence of $i+1$'s in row $r>i$, a continuous sequence of $i+2$'s above the $i+1$'s, and an $i$ and an $i+2$ in column $c$ both of which are not paired with an $i+1$. In particular, this implies that $col_{1,\dots,c}(u) = col_{1,\dots,c}(x)$ and thus $col_{1,\dots,c}(y)=col_{1,\dots,c}(f_{i+1}^*(x))$. However, this means $f_i(y)=0$ which contradicts the initial assumptions. If instead we assume that $f_{i+1}f_i^*(y)=0$ then an analogous contradiction can be derived. This proves the second part of $(5b)$ from Definition~\ref{def:demazure}.

\item[4a)] Suppose $|i-j|\geq 2$, $x,y \in \mathcal{C}$ are extremal, and $e_i^*(x)=e_j^*(y)=u$ for some $u \in \mathcal{C}$ also extremal. To prove $(4a)$ of Definition~\ref{def:demazure} we note that since $f_j(u) \neq 0$ and $f_i(u)$ will not affect any cells with values $j,j+1$ then all cells, paired an unpaired alike, with values $j,j+1$ of $u$ will remain the same in $x=f_i^*(u)$. Thus, $f_j(u)\neq 0$ implies $f_j(x) \neq 0$. Clearly, if $f_i(u) \neq 0$ then also $f_i(y) \neq 0$.

\item[4b)] Now suppose that $j=i+2$ and that as before $e_i^*(x)=e_{i+2}^*(y)=u$ for some extremal $u \in \mathcal{C}$. By condition $(4a)$ and the fact that $\mathcal{C}$ is an extremal subset of $\B(\lambda)$ we know this implies that $f_{i+2}^*(x)=f_i^*(y)=z$ for some $z\in \mathcal{C}$. Suppose also that $f_{i+1}(x) \neq 0$  and $f_{i+1}(z) \neq 0$. Since $f_{i+1}f_i(u) \neq 0$ and $e_{i+2}(u)=0$ then $\varphi_{i+1}(x) \geq \varphi_i(u)$. If $\varphi_{i+1}(y)=0$ then $\varphi_{i+1}(x)<\varphi_{i+1}(z)=\varphi_i(y)=\varphi_i(u)\leq \varphi_{i+1}(x)$, which is clearly nonsense. Hence, $\varphi_{i+1}(y)>0$ and so $f_{i+1}(y)=0$ only if $y$ has a Demazure death. So suppose this is the case and $y$ has a consecutive sequence of $i+1$'s in row $i+1$ with $i+2$'s above them in columns $1,\dots, c-1$ and whose leftmost unpaired $i+1$ lies in column $c$. If $col_{1,\dots,c}(u) \neq col_{1,\dots, c}(y)$ then $y$ contains a consecutive sequence of $i+3$'s in columns $1,\dots,c$ in a row between the $i+1's$ and $i+2's$, so that in $u=e_{i+2}^*(y)$ the $i+2$'s and $i+3's$ in these columns are swapped. The only way that $col_{1,\dots, c}(u) \neq col_{1,\dots, c}(x)$ is if $u$ has a consecutive sequence of $i's$ in row $i$ and columns $1,\dots,c+1$ that flip with the $i+1's$ in row $i+1$, but since $x=f_i^*(u)$ this would mean $x=0$. Thus $col_{1,\dots, c}(u) = col_{1,\dots, c}(x)$. Since $z=f_{i+2}^*(x)$, this implies that $col_{1,\dots,c}(z)=col_{1,\dots,c}(y)$ and thus $f_{i+1}(z)=0$ which contradicts our initial assumptions. Hence, it must be that $col_{1,\dots,c}(u) = col_{1,\dots, c}(y)$. By the same argument as above, $col_{1,\dots,c}(x) \neq col_{1,\dots, c}(u)$. However, this implies that $f_{i+1}(x)=0$ which is false by assumption. Hence, if $f_{i+1}(x)$ and $f_{i+1}(z)\neq 0$ then also $f_{i+1}(y) \neq 0$.

To prove the other direction suppose that instead $f_{i+1}(x)$ and $f_{i+1}(y)\neq 0$ but that $f_{i+1}(z) = 0$. As before, since $f_i(y),f_{i+1}(y) \neq 0$ then $\varphi_{i+1}(z)>1$. Thus, if $f_{i+1}(z)=0$ it is due to Demazure death. By condition $(5a)$ we know that if $f_{i+1}(z)=0$ then $f_{i}f_{i+1}^*(y)$ and $f_{i+2}f_{i+1}^*(x)$ are both nonzero. Now, by the proof of $(5a)$ above we know that $y$ must have its leftmost $i$ that is not paired with an $i+1$ in row $i$ and column $c$, such hat every column to its left has an $i$ in the same row and $i+1$ above it. Moreover, since $e_{i+2}^*(y) =u$, $f_{i+1}(x) \neq 0$, and $f_{i+1}(z)=0$ then by analogous arguments to those above we can deduce that columns $1,\dots, c-1$ of $y$ also contain a sequence of $i+2$-paired $i+3's$ with column $c$ containing the leftmost $i+3$ of $y$ not paired with an $i+2$. In particular, this means that every $i+2$ in columns $1,\dots, c$ is paired with an $i+3$. Since column $c$ of $y$ contains no $i+2$ and $f_{i+2}(y) \neq 0$ but $e_{i}(y) =0$ then $y$ must have some other cell right of column $c$ containing an $i+1$ that is not paired with and $i+2$ but is paired with an $i$. Furthermore, since $u=e_{i+2}^*(y)$ then this unpaired $i+1$ right of column $c$ is unaltered by $e_{i+2}$ and thus, remains unpaired with an $i$ in $u$. However, this means that $e_i(u) \neq 0$ which contradicts $u$ being extremal since $f^*_i(u) =x$ and so $u$ must be at the top of the $i$-string. Consequently, if both $f_{i+1}(x)$ and $f_{i+1}(y) \neq 0$ then also $f_{i+1}(z) \neq 0$.  This completes the proof for condition $(4b)$ of Definition~\ref{def:demazure}.

\item[6)] Finally, suppose that $x,y \in \mathcal{C}$ are extremal elements satisfying $e_i^*(x)=e_{i+1}^*e_i^*(y)=u$ for some $u\in \mathcal{C}$ also extremal and that, in addition, $f_k(x) \neq 0$ for some $k \neq i+1$.  Moreover, recall that $\varphi_{j}(f_i^*(u))\geq \varphi_{j}(u)$ whenever $|i-j|=1$ and $\varphi_j(f_{i}(u))=\varphi_j(u)$ whenever $|i-j|=2$. In order to show that $f_k(y)$ is also nonzero we observe how the columns of $u$ behave locally under the action of $f_i$ and $f_{i+1}$.  If a column $c$ of $u$ contains an $i$ that is not paired with an $i+1$ then under the action of $f_i^*$ the $i$ in column $c$ of $u$ will become an $i+1$-paired/unpaired $i+1$, depending on whether or not an unpaired $i+2$ exists in a column weakly right of $c$. Hence, either both or neither $f_{i+1}(x)$ and $f_{i+1}(y)$ are zero. If additionally, the $i$ in column $c$ of $u$ is $i-1$-paired, then in both $f_i^*(u)$ and $f_i^*f_{i+1}^*(u)$ the cell containing the $i-1$ with which this $i$ was paired will become unpaired.  Thus, $f_{i-1}(x)$ and $f_{i-1}(y)$ will be nonzero. Since the cells containing entries $k<i-1$ or $i+2<k$ are entirely unaffected by $f_i$ and $f_{i+1}$, then clearly if $f_k(x) \neq 0$ then $f_y(k) \neq 0$ also.

In an analogous manner as above, if we consider the action of $f_i$ and $f_{i+1}$ on a column $c$ of $u$ with an $i+1$ which is not paired with an $i+2$, we can see that if $f_k(x) \neq 0$ for $k\neq  i+1$, then $f_{i+1}(y)$ is also nonzero. Moreover, since the lowering operator simply swaps the rows of cells that are paired within the same column, then any cells with values $i-1,i,i+1, i+2$ that are paired with each other in the same column remain paired with each other after applying $f_i$ and $f_{i+1}$. Thus, if $f_{j_1}\dots f_{j_n}(x) \neq 0$ for some path $j_1, \dots, j_n$ for which the path $f_{j'_1}\dots f_{j'_{n-1}}f_{i+1}(x)$ for some other $j'_1,\dots,j'_{n-1}$ either does not exist or is not equal to $ f_{j_1}\dots f_{j_n}(x) $, then the path $f_{j_1}\dots f_{j_n}(y)$ is also nonzero. Thus, condition $(6)$ of Definition~\ref{def:demazure} holds for $\mathcal{C}$.
\end{itemize}
\end{proof}

In particular, we have a new proof of Theorem~\ref{thm:nskostka} that yields an explicit formula for the \newword{nonsymmetric Kostka--Foulkes polynomials} $K_{a,b}(q)$ defined by $E_b(X_n;q,0) = \sum_{a} K_{a,b}(q) \key_a(X_n)$.

\begin{corollary}
  For weak compositions $a,b$, we have
  \begin{equation}
    K_{a,b}(q) = \sum_{\substack{T \in \SSKD(b) \\ T \text{Demazure lowest weight} \\ \wt(T) = a }} q^{\maj(T)} .
    \label{e:nskostka-Dlw}
  \end{equation}
  In particular, $K_{a,b}(q) \in \mathbb{N}[q]$ and so nonsymmetric Macdonald polynomials specialized at $t=0$ are a nonnegative $q$-graded sum of Demazure characters.
  \label{cor:nskostka}
\end{corollary}

%
\section{Combinatorial formulas}
%
\label{sec:formulas}

Sanderson \cite{San00} first made the connection between specializations of Macdonald polynomials and Demazure characters by using the theory of nonsymmetric Macdonald polynomials in type A to construct an \emph{affine} Demazure module with graded character $P_{\mu}(X;q,0)$, parallel to the construction of Garsia and Procesi \cite{GP92} for Hall-Littlewood symmetric functions $H_{\mu}(X;0,t)$. Ion \cite{Ion03} generalized this result to nonsymmetric Macdonald polynomials in general type using the method of intertwiners in double affine Hecke algebras to realize $E_{a}(X;q,0)$ as an \emph{affine} Demazure character. Assaf \cite{Ass18} used the machinery of weak dual equivalence \cite{Ass-W} to realize $E_{a}(X;q,0)$ as a \emph{finite} Demazure character in type A. Corollary~\ref{cor:nskostka} gives an explicit formula for this expansion. In this final section, we consider consequences of the formula in Eq.~\eqref{e:nskostka-Dlw} in both the symmetric and nonsymmetric settings.

In \S\ref{sec:formulas-HL}, we review the Schur expansion of Hall--Littlewood symmetric functions. We also show how the highest weights of our Demazure crystals can be used to give an alternate formulation that uses the simple major index statistic instead of the intricate charge statistic. In \S\ref{sec:formulas-dem}, we use our explicit algorithm in Definition~\ref{def:dem-low} to generate the Demazure lowest weight of a component from the highest weight, making Eq.~\eqref{e:nskostka-Dlw} easy to compute. We also relate the symmetric and general cases to give a refinement of the Kostka--Foulkes coefficients in terms of the nonsymmetric Kostka--Foulkes polynomials.

\subsection{Hall--Littlewood polynomials}
\label{sec:formulas-HL}

The Hall--Littlewood symmetric functions $P_{\mu}(X;t)$ may be regarded as the $q=0$ specialization of Macdonald symmetric functions, i.e. $P_{\mu}(X;t) = P_{\mu}(X;0,t)$. The \newword{Kostka--Foulkes polynomials}, denoted $K_{\lambda,\mu}(t)$, give the transition coefficients between Hall--Littlewood symmetric functions and the Schur functions by
\begin{equation}
  s_{\lambda}(X) = \sum_{\mu} K_{\lambda,\mu}(t) P_{\mu}(X;t)
  \hspace{1em} \text{and} \hspace{1em}
  H_{\mu}(X;t) = \sum_{\lambda} K_{\lambda,\mu}(t) s_{\lambda}(X),
\end{equation}
where the modified version $H_{\mu}(X;t) = H_{\mu}(X;0,t)$ is defined analogously to Eq.~\eqref{e:modified}.

One readily observes that $K_{\lambda,\mu}(0)=\delta_{\lambda,\mu}$, equivalently $P_{\mu}(X;0) = s_{\mu}(X)$. It is also easy to verify that $P_{\mu}(X;1) = m_{\mu}(X)$, from which it follows that $K_{\lambda,\mu}(1)=K_{\lambda,\mu}$. That is, the Kostka--Foulkes polynomials are a $t$-graded version of the Kostka numbers, which have representation theoretic and geometric significance.

Hall--Littlewood polynomials arise in similar contexts as Schur functions, from which the representation theoretic and geometric importance of the Kostka--Foulkes polynomials becomes apparent. For $\chi_{\lambda}$ a unipotent character of $\mathrm{GL}_n(\mathbb{F}_t)$ and $\mu$ a conjugacy class, the evaluation of $\chi_{\lambda}$ at $\mu$ is given by $\chi_{\lambda}(\mu) = t^{n(\mu)} K_{\lambda,\mu}(1/t)$. For $R_{\mu}$ the $t$-graded $\mathcal{S}_n$-module constructed by Garsia and Procesi \cite{GP92}, the Frobenius character of $R_{\mu}$ is given by $\mathrm{ch}(R_{\mu}) = t^{n(\mu)}H_{\mu}(X;1/t)$. Geometrically, if we consider the Springer action of $\mathcal{S}_n$ on the cohomology ring $H^{*}(B_{\mu})$ of a Springer fiber $B_{\mu}$, then the cohomology ring $H^{*}(B_{\mu})$ has Frobenius series $t^{n(\mu)}H_{\mu}(X;1/t)$. For details of these connections, see Shoji \cite{Sho88}.

Recall the Kostka numbers $K_{\lambda,\mu}$ enumerate semistandard Young tableaux of shape $\lambda$ and partition weight $\mu$. Lascoux and Sch\"{u}tzenberger \cite{LS78} defined a statistic called \newword{charge} on these objects that precisely gives the $t$-grading of the Kostka--Foulkes polynomials $K_{\lambda,\mu}(t)$. More generally, we consider tableaux with partition weight and \emph{skew} shape, that is, of shape given by the set theoretic difference $\lambda\setminus\nu$ for $\nu \subset \lambda$. Sch\"{u}tzenberger \cite{Sch77} introduced the notion of \newword{jeu-de-taquin} slides that map skew tableaux to straight shapes. For details on \emph{jeu-de-taquin}, see Stanley \cite{EC2}(Appendix A).

\begin{definition}
  The \newword{cocharge} of a tableau $T$ with partition weight $\mu$ is the integer $cc(T)$ uniquely characterized by the following properties:
  \begin{enumerate}
  \item if $T$ is a single row, then $cc(T)=0$;
  \item if $S,T$ are \emph{jeu-de-taquin} equivalent, then $cc(S) = cc(T)$;
  \item if $T = R \cup S$ is a disjoint union of shapes with $R$ above and left of $S$ such that $R$ has no entry equal to $1$, then $cc(T) = cc(S \cup R) + \# R$, where $S \cup R$ has $S$ above and left of $R$.
  \end{enumerate}
  The \newword{charge} of $T$ is $c(T) = n(\mu)-cc(T)$, where $n(\mu)=\sum_i (i-1)\mu_i$.
  \label{def:charge}
\end{definition}

It is a theorem that such a statistic exists, but from this definition one obtains an algorithmic procedure, called \emph{catabolism}, for computing it. The main result, first asserted by Lascoux and Sch\"{u}tzenberger \cite{LS78} with omitted proof details supplied by Butler \cite{But86} is the following.

\begin{theorem}[\cite{LS78,But86}]
  The Kostka--Foulkes polynomials $K_{\lambda,\mu}(t)$ are given by
  \begin{equation}
    K_{\lambda,\mu}(t) = \sum_{\substack{T\in\SSYT(\lambda) \\ \wt(T) = \mu}} t^{c(T)} .
    \label{e:charge}
  \end{equation}
  \label{thm:charge}
\end{theorem}

For example, Fig.~\ref{fig:charge} shows the seven semistandard Young tableaux of partition shape and weight $(2,2,1)$. Their charges, from left to right, are $0,1,1,2,2,3,4$, from which we compute
\[  H_{(2,2,1)}(X;t)  =  s_{(2,2,1)}  +  t s_{(3,1,1)}  +  (t + t^2) s_{(3,2)}  +  (t^2 + t^3) s_{(4,1)} + t^4 s_{(5)} . \]

\begin{figure}[ht]
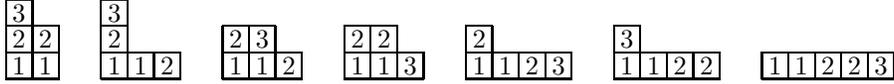

    \begin{displaymath}
      \arraycolsep=8pt
      \begin{array}{ccccccc}
        \tableau{3 \\ 2 & 2 \\ 1 & 1}   &
        \tableau{3 \\ 2 \\ 1 & 1 & 2}   &
        \tableau{\\ 2 & 3 \\ 1 & 1 & 2} &
        \tableau{\\ 2 & 2 \\ 1 & 1 & 3} &
        \tableau{\\ 2 \\ 1 & 1 & 2 & 3} &
        \tableau{\\ 3 \\ 1 & 1 & 2 & 2} &
        \tableau{\\ \\ 1 & 1 & 2 & 2 & 3}
      \end{array}
    \end{displaymath}
  \caption{\label{fig:charge}The seven semistandard Young tableaux of partition shape and weight $(2,2,1)$ used to compute $H_{(2,2,1)}(X;t)$.}
\end{figure}

Recall $E_{\mathrm{rev}(\lambda)}(X_n;q,0) = H_{\lambda}(X_n;q,0)$ and also by Eq.~\eqref{e:sym-key}, we have $\key_{\mathrm{rev}(\lambda)}(X_n) = s_{\lambda}(X_n)$. Therefore Corollary~\ref{cor:nskostka} gives a formula that we can relate to the Hall--Littlewood polynomial $H_{\lambda}(X_n;0,t)$ via the following result, proved combinatorially in \cite{Ass18}(Theorem~5.6).

\begin{theorem}[\cite{Ass18}]
  For $\lambda$ a partition of length $m$ and $\mathrm{rev}(\lambda)$ its weakly increasing rearrangement, we have
  \begin{equation}
    E_{\mathrm{rev}(\lambda)}(X_m;q,0) = \omega H_{\lambda^{\prime}}(X_m;0,q) ,
  \end{equation}
  where $\lambda^{\prime}$ denotes the conjugate (diagrammatic transpose) of $\lambda$, and $\omega$ is the symmetric function involution determined by $\omega s_{\lambda} = s_{\lambda^{\prime}}$.
  \label{thm:mac-stable}
\end{theorem}

To utilize Corollary~\ref{cor:nskostka} in the context of Theorem~\ref{thm:mac-stable}, we have the following.

\begin{lemma}
  Given a weakly increasing weak composition $b$, every connected component of the Demazure crystal on $\SSKD(b)$ is a normal crystal.
  \label{lem:full}
\end{lemma}

\begin{proof}
  When $b$ is weakly increasing of length $n$, the specialized nonsymmetric Macdonald polynomial $E_b(X_n;q,0)$ is symmetric in $x_1,\ldots,x_n$. Given a weak composition $a$ of length $n$, by \cite{AS18}(Theorem~4.2) the Demazure character $\key_a$ is symmetric in $x_1,\ldots,x_n$ if and only if $a$ is weakly increasing. By Theorem~\ref{thm:nskostka}, or equivalently by Corollary~\ref{cor:nskostka}, the coefficients in the Demazure expansion of $E_b(X_n;q,0)$ are polynomials in $q$ with nonnegative coefficients. Therefore every term that appears in the Demazure expansion of $E_b(X_n;q,0)$ must be symmetric, that is, every Demazure character that appears with nonzero coefficient is, in fact, a Schur polynomial in $n$ variables. Consequently, the corresponding crystals must be full crystals.
\end{proof}

Since normal $\gl_n$-crystals are uniquely determined by their highest weights, which also give their characters, we have a new paradigm for computing Kostka-Foulkes polynomials that utilizes the highest weight elements of the tabloid crystal together with the simple major index statistic.

\begin{theorem}
  The Kostka--Foulkes polynomials $K_{\lambda,\mu}(t)$ are given by
  \begin{equation}
    K_{\lambda,\mu}(t) = \sum_{\substack{T\in\SSKD(0^m\times\mathrm{rev}(\mu^{\prime})) \\ \wt(T) = \lambda^{\prime} \\ e_i(T) = 0 \forall i}} t^{\maj(T)} ,
    \label{e:charge-crystal}
  \end{equation}
  for any $m \geq |\mu|-\mu_1$.
  \label{thm:charge-crystal}
\end{theorem}

\begin{proof}
  Let $n=|\mu|$ and set $b = 0^{m}\times\mathrm{rev}(\mu^{\prime})$. By Lemma~\ref{lem:full}, since $b$ is weakly increasing, every component of the Demazure crystal on semistandard key tabloids is a full crystal. Thus components can be indexed by their highest weights and their characters are given by the corresponding Schur polynomials. Combining this with Theorem~\ref{thm:mac-stable} gives
  \[ \omega H_{\mu}(X_m;0,q) = E_{0^{m}\times\mathrm{rev}(\mu^{\prime})}(X_m;q,0)
  = \sum_{a \ \text{weakly inc.}} K_{a,b}(q) s_{\mathrm{rev}(a)}(X_m) . \]
  Applying $\omega$ to the expression above yields
  \[ H_{\mu}(X_m;0,q) = \sum_{a \ \text{weakly inc.}} K_{a,b}(q) \omega s_{\mathrm{rev}(a)}(X_m) = \sum_{a \ \text{weakly inc.}} K_{a,b}(q) s_{\mathrm{rev}(a)^{\prime}}(X_m) . \]
  Now fix a weakly increasing weak composition $a$ and set $\lambda = \mathrm{rev}(a)^{\prime}$. Using highest weights, Corollary~\ref{cor:nskostka} becomes
  \[  K_{\lambda,\mu}(q) = K_{a,b}(q) = \sum_{\substack{T \in \SSKD(b) \\ T \text{ highest weight} \\ \wt(T) = \mathrm{rev}(a) }} q^{\maj(T)} . \]
  The formula now follows.
\end{proof}

\begin{example}
  The seven highest weight semistandard key tabloids of shape $(0^3,2,3)$ are shown in Fig.~\ref{fig:highest}. The $q$-weight of these terms is easily determined by the major index statistic, giving
  \begin{displaymath}
    E_{(0^3,2,3)}(X;q,0)  =  \key_{(0^3,2,3)}  +  q \key_{(0^2,1,1,3)}  +  (q+q^2)\key_{(0^2,1,2,2)}  +  (q^2+q^3)\key_{(0,1,1,1,2)} + q^4 \key_{(1,1,1,1,1)}.
  \end{displaymath}
  Each of the above Demazure characters corresponds to a Schur polynomial in $x_1,\ldots,x_5$, and writing it as such we have
  \begin{displaymath}
    E_{(0^3,2,3)}(X;q,0)  =  s_{(3,2)}  +  q s_{(3,1,1)}  +  (q+q^2)s_{(2,2,1)}  +  (q^2+q^3)s_{(2,1,1,1)} + q^4 s_{(1,1,1,1,1)}.
  \end{displaymath}
  Exchanging $q$ with $t$ and conjugating each partition gives $H_{(2,2,1)}(X;t)$ computed earlier.
\end{example}

  \begin{figure}[ht]
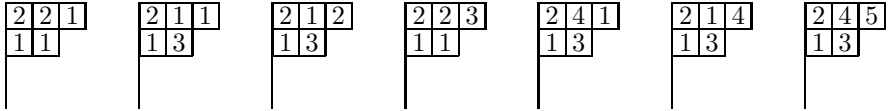

    \begin{displaymath}
      \arraycolsep=10pt
      \begin{array}{cccccccc}
        \vline\tableau{2 & 2 & 1 \\ 1 & 1 \\ & \\ & } &
        \vline\tableau{2 & 1 & 1 \\ 1 & 3 \\ & \\ & } &
        \vline\tableau{2 & 1 & 2 \\ 1 & 3 \\ & \\ & } &
        \vline\tableau{2 & 2 & 3 \\ 1 & 1 \\ & \\ & } &
        \vline\tableau{2 & 4 & 1 \\ 1 & 3 \\ & \\ & } &
        \vline\tableau{2 & 1 & 4 \\ 1 & 3 \\ & \\ & } &
        \vline\tableau{2 & 4 & 5 \\ 1 & 3 \\ & \\ & } 
      \end{array}
    \end{displaymath}
  \caption{\label{fig:highest}The highest weights for the Demazure crystal for $E_{(0^3,2,3)}(X;q,0)$.}
\end{figure}

\subsection{Explicit Demazure expansions}
\label{sec:formulas-dem}

Recall that highest weight elements of a Demazure crystal do not give the Demazure characters. Thus highest weights of the tabloid crystal do not immediately give a formula for the Demazure expansion of the specialized nonsymmetric Macdonald polynomial outside of the symmetric case resolved by Theorem~\ref{thm:charge-crystal}.

\begin{example}\label{ex:0302}
  The six highest weight semistandard key tabloids of shape $(0,3,0,2)$ are shown in Fig.~\ref{fig:highest}, indicating that the Demazure crystal has six connected components, and so the Demazure expansion of $E_{(0,3,0,2)}(X;q,0)$ has six terms. However, these tabloids do not determine the Demazure characters themselves.
\end{example}

\begin{figure}[ht]
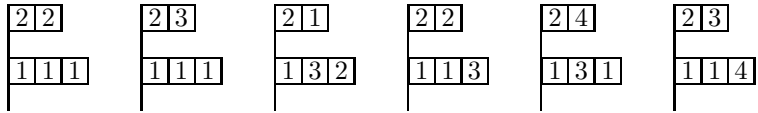

  \begin{displaymath}
    \begin{array}{c@{\hskip 2\cellsize}c@{\hskip 2\cellsize}c@{\hskip 2\cellsize}c@{\hskip 2\cellsize}c@{\hskip 2\cellsize}c}
      \vline\tableau{2 & 2 \\ & \\ 1 & 1 & 1 \\ & } &
      \vline\tableau{2 & 3 \\ & \\ 1 & 1 & 1 \\ & } &
      \vline\tableau{2 & 1 \\ & \\ 1 & 3 & 2 \\ & } &
      \vline\tableau{2 & 2 \\ & \\ 1 & 1 & 3 \\ & } &
      \vline\tableau{2 & 4 \\ & \\ 1 & 3 & 1 \\ & } &
      \vline\tableau{2 & 3 \\ & \\ 1 & 1 & 4 \\ & }
    \end{array}
  \end{displaymath}
  \caption{\label{fig:highest-Dem}The highest weights for the Demazure crystal for $E_{(0,3,0,2)}(X;q,0)$.}
\end{figure}

We can construct the Demazure lowest weights from the highest weights using Definition~\ref{def:dem-low}. Given the explicit objects, this is easy to compute.

\begin{example}
  Consider the leftmost tabloids Fig.~\ref{fig:high-low}, which are two highest weight elements both of weight $(2,2,1)$. Following Definition~\ref{def:dem-low}, we first act by $F_{[1,3]}$. For the second iteration, the top row will act with $F_{[3,3]}$ while the bottom row will act by $F_{[2,3]}$, after which both examples terminate at their respective lowest weight elements.
\end{example}

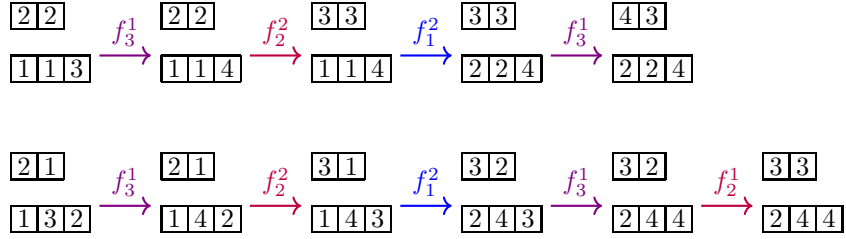
\begin{figure}[ht]
  \begin{tikzpicture}[xscale=2,yscale=2]
    \node at (0,0) (a) {\vline\tableau{2 & 1 \\ & \\ 1 & 3 & 2 \\ & }};
    \node at (1,0) (b) {\vline\tableau{2 & 1 \\ & \\ 1 & 4 & 2 \\ & }};
    \node at (2,0) (c) {\vline\tableau{3 & 1 \\ & \\ 1 & 4 & 3 \\ & }};
    \node at (3,0) (d) {\vline\tableau{3 & 2 \\ & \\ 2 & 4 & 3 \\ & }};
    \node at (4,0) (e) {\vline\tableau{3 & 2 \\ & \\ 2 & 4 & 4 \\ & }};
    \node at (5,0) (f) {\vline\tableau{3 & 3 \\ & \\ 2 & 4 & 4 \\ & }};
    \node at (0,1) (A) {\vline\tableau{2 & 2 \\ & \\ 1 & 1 & 3 \\ & }};
    \node at (1,1) (B) {\vline\tableau{2 & 2 \\ & \\ 1 & 1 & 4 \\ & }};
    \node at (2,1) (C) {\vline\tableau{3 & 3 \\ & \\ 1 & 1 & 4 \\ & }};
    \node at (3,1) (D) {\vline\tableau{3 & 3 \\ & \\ 2 & 2 & 4 \\ & }};
    \node at (4,1) (E) {\vline\tableau{4 & 3 \\ & \\ 2 & 2 & 4 \\ & }};
    \draw[thick,->,violet] (a) -- (b) node[midway,above] {$f_{3}^{1}$};
    \draw[thick,->,purple] (b) -- (c) node[midway,above] {$f_{2}^{2}$};
    \draw[thick,->,blue  ] (c) -- (d) node[midway,above] {$f_{1}^{2}$};
    \draw[thick,->,violet] (d) -- (e) node[midway,above] {$f_{3}^{1}$};
    \draw[thick,->,purple] (e) -- (f) node[midway,above] {$f_{2}^{1}$};
    \draw[thick,->,violet] (A) -- (B) node[midway,above] {$f_{3}^{1}$};
    \draw[thick,->,purple] (B) -- (C) node[midway,above] {$f_{2}^{2}$};
    \draw[thick,->,blue  ] (C) -- (D) node[midway,above] {$f_{1}^{2}$};
    \draw[thick,->,violet] (D) -- (E) node[midway,above] {$f_{3}^{1}$};
  \end{tikzpicture}
  \caption{\label{fig:high-low}Using Definition~\ref{def:dem-low} to construct the Demazure lowest weight tabloids from two highest weight tabloids.}
\end{figure}

Mapping each of the highest weight tabloids in Fig.~\ref{fig:highest-Dem} to their corresponding Demazure lowest weights results in the tabloids in Fig.~\ref{fig:lowest}. The $q$-weight of these terms is easily determined by the major index statistic, giving
  \[ E_{(0,3,0,2)}(X;q,0) = \key_{(0,3,0,2)} + q \key_{(0,3,1,1)} + q\key_{(0,2,1,2)} + q^2\key_{(0,1,2,2)} + q^2 \key_{(1,2,1,1)} + q^3 \key_{(1,1,1,2)} . \]

\begin{figure}[ht]
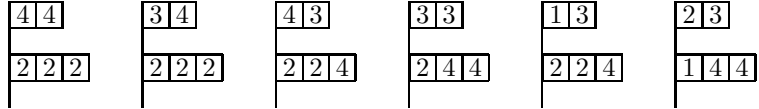

    \begin{displaymath}
      \begin{array}{c@{\hskip 2\cellsize}c@{\hskip 2\cellsize}c@{\hskip 2\cellsize}c@{\hskip 2\cellsize}c@{\hskip 2\cellsize}c}
        \vline\tableau{4 & 4 \\ & \\ 2 & 2 & 2 \\ & } &
        \vline\tableau{3 & 4 \\ & \\ 2 & 2 & 2 \\ & } &
        \vline\tableau{4 & 3 \\ & \\ 2 & 2 & 4 \\ & } &
        \vline\tableau{3 & 3 \\ & \\ 2 & 4 & 4 \\ & } &
        \vline\tableau{1 & 3 \\ & \\ 2 & 2 & 4 \\ & } &
        \vline\tableau{2 & 3 \\ & \\ 1 & 4 & 4 \\ & }
      \end{array}
    \end{displaymath}
  \caption{\label{fig:lowest}The Demazure lowest weights for the Demazure crystal for $E_{(0,3,0,2)}(X;q,0)$.}
\end{figure}

Notice that while $E_{(0^3,2,3)}(X;q,0)$ had multiplicity in its Schur expansion, the Demazure expansion of $E_{(0,3,0,2)}(X;q,0)$ is multiplicity-free. This is particularly interesting since these polynomials agree as \emph{functions} in the stable limit, i.e.
\[ \lim_{m\rightarrow\infty} E_{0^m \times (0,3,0,2)}(X;q,0) = \lim_{m\rightarrow\infty} E_{0^m \times (0^3,2,3)}(X;q,0) = \omega H_{(2,2,1)}(X;0,q) . \]
This happens precisely because, as demonstrated in the example above, the algorithm for computing the Demazure lowest weights differs for two highest weights of the same weight. Comparing expansions, we have
\[ E_{(0,3,0,2)}(X;q,0) =
\underbrace{\key_{(0,3,0,2)}}_{s_{(2,2,1)}} +
\underbrace{q \key_{(0,3,1,1)}}_{t s_{(3,1,1)}} +
\underbrace{q\key_{(0,2,1,2)} + q^2\key_{(0,1,2,2)}}_{(t + t^2)s_{(3,2)}} +
\underbrace{q^2 \key_{(1,2,1,1)} + q^3 \key_{(1,1,1,2)}}_{(t^2 + t^3)s_{(4,1)}} . \]

Summarizing this refinement, we have the following.

\begin{corollary}[\cite{Ass18}]
  Given a weak composition $b$ with column lengths $\mu$, for $m$ sufficiently large, we have
  \begin{equation}
    K_{\lambda,\mu}(t) = \sum_{\mathrm{sort}(a) = \lambda^{\prime}} K_{a,0^m\times b}(t) .
  \end{equation}
\end{corollary}

\subsection{Concluding remarks}
\label{sec:formulas-more}

Recall that the first results in the direction of this paper began with Sanderson \cite{San00} who showed that $E_a(X;q,0)$ is equal to a single affine Demazure character for the general linear group. Ion \cite{Ion03} generalized her result to other types using the framework of double affine Hecke algebras, and Lenart, Naito, Sagaki, Schilling and Shimozono \cite{LNSS17} gave a crystal-theoretic proof that also encompasses other types.

Our crystal-theoretic approach to the nonnegative expansion of $E_a(X;q,0)$ as a sum of finite Demazure characters for the general linear group was motivated by several factors. First, we hoped to improve upon the combinatorial formula for the expansion that came from the original proof of Assaf \cite{Ass18}, and Corollary~\ref{cor:nskostka} succeeds in that thanks to Theorem~\ref{thm:dem-low}. Second, the crystal approach gives a representation theoretic context for the nonnegativity, which suggests that an affine Demazure module should admit a finite Demazure flag, and that this can be proved using crystal theory. Third, our methods utilize tools such as crystals \cite{Kas93} and Stembridge's local characterization \cite{Ste03} that exist for other types, giving hope that our techniques can be generalized. 

While we might hope to extend our results to gain deeper understanding of the nonsymmetric Macdonald polynomials in two parameters, $E_a(X;q,t)$, the impediment here appears more daunting. The nonnegativity results and connections to representation theory and geometry in the classical symmetric case come only through plethysm, which has no known analog in the polynomial ring. Nevertheless, by considering specializations at other natural values of $t$, one can hope to gain insights to help to cross this final barrier.

%
%

\bibliographystyle{plain}
\bibliography{tabloid}

\begin{thebibliography}{10}

\bibitem{Ass-R}
Sami Assaf.
\newblock An inversion metric for reduced words.
\newblock {\em Adv. in Appl. Math.}
\newblock to appear.

\bibitem{Ass-W}
Sami Assaf.
\newblock Weak dual equivalence for polynomials.
\newblock arXiv:1702.04051.

\bibitem{Ass18}
Sami Assaf.
\newblock Nonsymmetric {M}acdonald polynomials and a refinement of
  {K}ostka--{F}oulkes polynomials.
\newblock {\em Trans. Amer. Math. Soc.}, 370(12):8777--8796, 2018.

\bibitem{ASc18}
Sami Assaf and Anne Schilling.
\newblock A {D}emazure crystal construction for {S}chubert polynomials.
\newblock {\em Algebraic Combinatorics}, 1(2):225--247, 2018.

\bibitem{AS17}
Sami Assaf and Dominic Searles.
\newblock Schubert polynomials, slide polynomials, {S}tanley symmetric
  functions and quasi-{Y}amanouchi pipe dreams.
\newblock {\em Adv. in Math.}, 306:89--122, 2017.

\bibitem{AS18}
Sami Assaf and Dominic Searles.
\newblock Kohnert tableaux and a lifting of quasi-{S}chur functions.
\newblock {\em J. Combin. Theory Ser. A}, 156:85--118, 2018.

\bibitem{But86}
Lynne~Marie Butler.
\newblock {\em Combinatorial properties of partially ordered sets associated
  with partitions and finite abelian groups}.
\newblock ProQuest LLC, Ann Arbor, MI, 1986.
\newblock Thesis (Ph.D.)--Massachusetts Institute of Technology.

\bibitem{Che95}
Ivan Cherednik.
\newblock Nonsymmetric {M}acdonald polynomials.
\newblock {\em Internat. Math. Res. Notices}, (10):483--515, 1995.

\bibitem{Dem74a}
Michel Demazure.
\newblock D\'esingularisation des vari\'et\'es de {S}chubert
  g\'en\'eralis\'ees.
\newblock {\em Ann. Sci. \'Ecole Norm. Sup. (4)}, 7:53--88, 1974.
\newblock Collection of articles dedicated to Henri Cartan on the occasion of
  his 70th birthday, I.

\bibitem{Dem74}
Michel Demazure.
\newblock Une nouvelle formule des caract\`eres.
\newblock {\em Bull. Sci. Math. (2)}, 98(3):163--172, 1974.

\bibitem{GH96}
A.~M. Garsia and M.~Haiman.
\newblock Some natural bigraded {$S_n$}-modules and {$q,t$}-{K}ostka
  coefficients.
\newblock {\em Electron. J. Combin.}, 3(2):Research Paper 24, approx.\ 60 pp.,
  1996.
\newblock The Foata Festschrift.

\bibitem{GP92}
A.~M. Garsia and C.~Procesi.
\newblock On certain graded {$S_n$}-modules and the {$q$}-{K}ostka polynomials.
\newblock {\em Adv. Math.}, 94(1):82--138, 1992.

\bibitem{GH93}
Adriano~M. Garsia and Mark Haiman.
\newblock A graded representation model for {M}acdonald's polynomials.
\newblock {\em Proc. Nat. Acad. Sci. U.S.A.}, 90(8):3607--3610, 1993.

\bibitem{Hag04}
J.~Haglund.
\newblock A combinatorial model for the {M}acdonald polynomials.
\newblock {\em Proc. Natl. Acad. Sci. USA}, 101(46):16127--16131, 2004.

\bibitem{HHL05}
J.~Haglund, M.~Haiman, and N.~Loehr.
\newblock A combinatorial formula for {M}acdonald polynomials.
\newblock {\em J. Amer. Math. Soc.}, 18(3):735--761, 2005.

\bibitem{HHL08}
J.~Haglund, M.~Haiman, and N.~Loehr.
\newblock A combinatorial formula for nonsymmetric {M}acdonald polynomials.
\newblock {\em Amer. J. Math.}, 130(2):359--383, 2008.

\bibitem{Hai01}
Mark Haiman.
\newblock Hilbert schemes, polygraphs and the {M}acdonald positivity
  conjecture.
\newblock {\em J. Amer. Math. Soc.}, 14(4):941--1006, 2001.

\bibitem{Ion03}
Bogdan Ion.
\newblock Nonsymmetric {M}acdonald polynomials and {D}emazure characters.
\newblock {\em Duke Math. J.}, 116(2):299--318, 2003.

\bibitem{Ion05}
Bogdan Ion.
\newblock A weight multiplicity formula for {D}emazure modules.
\newblock {\em Int. Math. Res. Not.}, (5):311--323, 2005.

\bibitem{Kas91}
M.~Kashiwara.
\newblock On crystal bases of the {$Q$}-analogue of universal enveloping
  algebras.
\newblock {\em Duke Math. J.}, 63(2):465--516, 1991.

\bibitem{Kas93}
Masaki Kashiwara.
\newblock The crystal base and {L}ittelmann's refined {D}emazure character
  formula.
\newblock {\em Duke Math. J.}, 71(3):839--858, 1993.

\bibitem{KN94}
Masaki Kashiwara and Toshiki Nakashima.
\newblock Crystal graphs for representations of the {$q$}-analogue of classical
  {L}ie algebras.
\newblock {\em J. Algebra}, 165(2):295--345, 1994.

\bibitem{Kno97}
Friedrich Knop.
\newblock Integrality of two variable {K}ostka functions.
\newblock {\em J. Reine Angew. Math.}, 482:177--189, 1997.

\bibitem{KS97}
Friedrich Knop and Siddhartha Sahi.
\newblock A recursion and a combinatorial formula for {J}ack polynomials.
\newblock {\em Invent. Math.}, 128(1):9--22, 1997.

\bibitem{Koh91}
Axel Kohnert.
\newblock Weintrauben, {P}olynome, {T}ableaux.
\newblock {\em Bayreuth. Math. Schr.}, (38):1--97, 1991.
\newblock Dissertation, Universit{\"a}t Bayreuth, Bayreuth, 1990.

\bibitem{LS78}
Alain Lascoux and Marcel-Paul Sch\"{u}tzenberger.
\newblock Sur une conjecture de {H}. {O}. {F}oulkes.
\newblock {\em C. R. Acad. Sci. Paris S\'{e}r. A-B}, 286(7):A323--A324, 1978.

\bibitem{LNSS17}
C.~Lenart, S.~Naito, D.~Sagaki, A.~Schilling, and M.~Shimozono.
\newblock A uniform model for {K}irillov-{R}eshetikhin crystals {III}:
  nonsymmetric {M}acdonald polynomials at {$t=0$} and {D}emazure characters.
\newblock {\em Transform. Groups}, 22(4):1041--1079, 2017.

\bibitem{Lit95}
Peter Littelmann.
\newblock Crystal graphs and {Y}oung tableaux.
\newblock {\em J. Algebra}, 175(1):65--87, 1995.

\bibitem{Lus90}
G.~Lusztig.
\newblock Canonical bases arising from quantized enveloping algebras.
\newblock {\em J. Amer. Math. Soc.}, 3(2):447--498, 1990.

\bibitem{Mac88}
I.~G. Macdonald.
\newblock A new class of symmetric functions.
\newblock {\em Actes du 20e Seminaire Lotharingien}, 372:131--171, 1988.

\bibitem{Mac95}
I.~G. Macdonald.
\newblock {\em Symmetric functions and {H}all polynomials}.
\newblock Oxford Mathematical Monographs. The Clarendon Press Oxford University
  Press, New York, second edition, 1995.
\newblock With contributions by A. Zelevinsky, Oxford Science Publications.

\bibitem{Mac96}
I.~G. Macdonald.
\newblock Affine {H}ecke algebras and orthogonal polynomials.
\newblock {\em Ast\'erisque}, (237):Exp.\ No.\ 797, 4, 189--207, 1996.
\newblock S\'eminaire Bourbaki, Vol. 1994/95.

\bibitem{Mas09}
Sarah Mason.
\newblock An explicit construction of type {A} {D}emazure atoms.
\newblock {\em J. Algebraic Combin.}, 29(3):295--313, 2009.

\bibitem{Opd95}
Eric~M. Opdam.
\newblock Harmonic analysis for certain representations of graded {H}ecke
  algebras.
\newblock {\em Acta Math.}, 175(1):75--121, 1995.

\bibitem{Sah96}
Siddhartha Sahi.
\newblock Interpolation, integrality, and a generalization of {M}acdonald's
  polynomials.
\newblock {\em Internat. Math. Res. Notices}, (10):457--471, 1996.

\bibitem{San00}
Yasmine~B. Sanderson.
\newblock On the connection between {M}acdonald polynomials and {D}emazure
  characters.
\newblock {\em J. Algebraic Combin.}, 11(3):269--275, 2000.

\bibitem{Sch77}
M.-P. Sch{\"u}tzenberger.
\newblock La correspondance de {R}obinson.
\newblock In {\em Combinatoire et repr\'esentation du groupe sym\'etrique
  ({A}ctes {T}able {R}onde {CNRS}, {U}niv. {L}ouis-{P}asteur {S}trasbourg,
  {S}trasbourg, 1976)}, pages 59--113. Lecture Notes in Math., Vol. 579.
  Springer, Berlin, 1977.

\bibitem{Sho88}
Toshiaki Shoji.
\newblock Geometry of orbits and {S}pringer correspondence.
\newblock {\em Ast\'{e}risque}, (168):9, 61--140, 1988.
\newblock Orbites unipotentes et repr\'{e}sentations, I.

\bibitem{EC2}
Richard~P. Stanley.
\newblock {\em Enumerative combinatorics. {V}ol. 2}, volume~62 of {\em
  Cambridge Studies in Advanced Mathematics}.
\newblock Cambridge University Press, Cambridge, 1999.
\newblock With a foreword by Gian-Carlo Rota and appendix 1 by Sergey Fomin.

\bibitem{Ste03}
John~R. Stembridge.
\newblock A local characterization of simply-laced crystals.
\newblock {\em Trans. Amer. Math. Soc.}, 355(12):4807--4823, 2003.

\end{thebibliography}

%
\appendix
%

\section{Complete example of the Demazure crystals for $E_{(0,3,0,2)}(X;q,0)$}

The following six Demazure crystals, in Fig.s~\ref{fig:demazure0311} and \ref{fig:demazure0122}, correspond to the Demazure expansion of $E_{(0,3,0,2)}(X;q,0)$ from Example~\ref{ex:0302}. 

As can be readily verified, the character of each crystal corresponds precisely to the key polynomial indexed by the weak composition equal to the lowest weight of each Demazure crystal. Hence, when taking the sum of the graded characters of all six crystals we indeed recover the nonsymmetric Macdonald polynomial $E_{(0,3,0,2)}(X;q,0)$.

\begin{figure}[ht]
\begin{tikzpicture}
\begin{scope}[shift={(-4,3)}]
\begin{scope}[shift={(0,3)}]
\draw (-.41,-.6)--(-.41,.41);
\node at (0,0)[scale=.5]{$\begin{ytableau}
2 & 4 \\
\\
1 & 3&1\\
\end{ytableau}$};
\draw [white, line width = 1mm](-.14,-.13)--(-.14,.13);
\end{scope}
\begin{scope}[shift={(-1,1.5)}]
\draw (-.41,-.6)--(-.41,.41);
\node at (0,0)[scale=.5]{$\begin{ytableau}
2 & 4 \\
\\
1 & 3&2\\
\end{ytableau}$};
\draw [white, line width = 1mm](-.14,-.13)--(-.14,.13);
\end{scope}
\begin{scope}[shift={(-1,0)}]
\draw (-.41,-.6)--(-.41,.41);
\node at (0,0)[scale=.5]{$\begin{ytableau}
2 & 4 \\
\\
1 & 3&3\\
\end{ytableau}$};
\draw [white, line width = 1mm](-.14,-.13)--(-.14,.13);
\end{scope}
\begin{scope}[shift={(0,-1.5)}]
\draw (-.41,-.6)--(-.41,.41);
\node at (0,0)[scale=.5]{$\begin{ytableau}
2 & 3 \\
\\
1 & 4&4\\
\end{ytableau}$};
\draw [white, line width = 1mm](-.14,-.13)--(-.14,.13);
\end{scope}
\draw[orange,->](-1,1)--(-1,.5);
\begin{scope}[yscale=-1]
\draw[blue, <-](-.75,-2)--(-.25,-2.5);
\end{scope}
\draw[purple, ->](-.75,-.5)--(-.25,-1);
\end{scope}
\begin{scope}[shift={(2,6)}]
\begin{scope}[shift={(-2,0)}]
\draw (-.41,-.6)--(-.41,.41);
\node at (0,0)[scale=.5]{$\begin{ytableau}
2 & 1 \\
\\
1 & 3&2\\
\end{ytableau}$};
\draw [white, line width = 1mm](-.14,-.13)--(-.14,.13);
\end{scope}
\end{scope}

\begin{scope}[shift={(2,4.5)}]
\begin{scope}[shift={(-1,0)}]
\draw (-.41,-.6)--(-.41,.41);
\node at (0,0)[scale=.5]{$\begin{ytableau}
2 & 1 \\
\\
1 &4&2\\
\end{ytableau}$};
\draw [white, line width = 1mm](-.14,-.13)--(-.14,.13);
\end{scope}
\begin{scope}[shift={(-2,0)}]
\draw (-.41,-.6)--(-.41,.41);
\node at (0,0)[scale=.5]{$\begin{ytableau}
2 & 1 \\
\\
1 & 3&3\\
\end{ytableau}$};
\draw [white, line width = 1mm](-.14,-.13)--(-.14,.13);
\end{scope}
\end{scope}

\begin{scope}[shift={(0,3)}]
\begin{scope}[shift={(2,0)}]
\draw (-.41,-.6)--(-.41,.41);
\node at (0,0)[scale=.5]{$\begin{ytableau}
2 & 1 \\
\\
1 & 4&3\\
\end{ytableau}$};
\draw [white, line width = 1mm](-.14,-.13)--(-.14,.13);
\end{scope} 
\begin{scope}[shift={(1,0)}]
\draw (-.41,-.6)--(-.41,.41);
\node at (0,0)[scale=.5]{$\begin{ytableau}
3 & 1 \\
\\
1 & 4&2\\
\end{ytableau}$};
\draw [white, line width = 1mm](-.14,-.13)--(-.14,.13);
\end{scope}
\begin{scope}[shift={(-2,0)}]
\draw (-.41,-.6)--(-.41,.41);
\node at (0,0)[scale=.5]{$\begin{ytableau}
2 & 2 \\
\\
1 & 3&3\\
\end{ytableau}$};
\draw [white, line width = 1mm](-.14,-.13)--(-.14,.13);
\end{scope}
\end{scope}

\begin{scope}[shift={(1,1.5)}]
\begin{scope}[shift={(2,0)}]
\draw (-.41,-.6)--(-.41,.41);
\node at (0,0)[scale=.5]{$\begin{ytableau}
2 & 1 \\
\\
1 & 4&4\\
\end{ytableau}$};
\draw [white, line width = 1mm](-.14,-.13)--(-.14,.13);
\end{scope} 
\begin{scope}[shift={(0,0)}]
\draw (-.41,-.6)--(-.41,.41);
\node at (0,0)[scale=.5]{$\begin{ytableau}
3 & 1 \\
\\
1 & 4&3\\
\end{ytableau}$};
\draw [white, line width = 1mm](-.14,-.13)--(-.14,.13);
\end{scope}
\begin{scope}[shift={(-1,0)}]
\draw (-.41,-.6)--(-.41,.41);
\node at (0,0)[scale=.5]{$\begin{ytableau}
3 & 1 \\
\\
2 & 4&2\\
\end{ytableau}$};
\draw [white, line width = 1mm](-.14,-.13)--(-.14,.13);
\end{scope}
\begin{scope}[shift={(-2,0)}]
\draw (-.41,-.6)--(-.41,.41);
\node at (0,0)[scale=.5]{$\begin{ytableau}
2 & 2 \\
\\
1 & 4&3\\
\end{ytableau}$};
\draw [white, line width = 1mm](-.14,-.13)--(-.14,.13);
\end{scope}
\end{scope}
\begin{scope}[shift={(3,0)}]
\draw (-.41,-.6)--(-.41,.41);
\node at (0,0)[scale=.5]{$\begin{ytableau}
3 & 1 \\
\\
1 & 4&4\\
\end{ytableau}$};
\draw [white, line width = 1mm](-.14,-.13)--(-.14,.13);
\end{scope} 
\begin{scope}[shift={(1,0)}]
\draw (-.41,-.6)--(-.41,.41);
\node at (0,0)[scale=.5]{$\begin{ytableau}
2 & 2 \\
\\
1 & 4&4\\
\end{ytableau}$};
\draw [white, line width = 1mm](-.14,-.13)--(-.14,.13);
\end{scope}
\begin{scope}[shift={(0,0)}]
\draw (-.41,-.6)--(-.41,.41);
\node at (0,0)[scale=.5]{$\begin{ytableau}
3& 1 \\
\\
2 & 4&3\\
\end{ytableau}$};
\draw [white, line width = 1mm](-.14,-.13)--(-.14,.13);
\end{scope}
\begin{scope}[shift={(-1,0)}]
\draw (-.41,-.6)--(-.41,.41);
\node at (0,0)[scale=.5]{$\begin{ytableau}
3 & 2 \\
\\
1 & 4&3\\
\end{ytableau}$};
\draw [white, line width = 1mm](-.14,-.13)--(-.14,.13);
\end{scope}

\begin{scope}[shift={(0,-4.5)}]
\draw (-.41,-.6)--(-.41,.41);
\node at (0,0)[scale=.5]{$\begin{ytableau}
3 & 3 \\
\\
2 & 4&4\\
\end{ytableau}$};
\draw [white, line width = 1mm](-.14,-.13)--(-.14,.13);
\end{scope}

\begin{scope}[shift={(1,-3)}]
\draw (-.41,-.6)--(-.41,.41);
\node at (0,0)[scale=.5]{$\begin{ytableau}
3 & 3 \\
\\
1 & 4&4\\
\end{ytableau}$};
\draw [white, line width = 1mm](-.14,-.13)--(-.14,.13);
\end{scope}
\begin{scope}[shift={(0,-3)}]
\draw (-.41,-.6)--(-.41,.41);
\node at (0,0)[scale=.5]{$\begin{ytableau}
3 & 2 \\
\\
2 & 4&4\\
\end{ytableau}$};
\draw [white, line width = 1mm](-.14,-.13)--(-.14,.13);
\end{scope}

\begin{scope}[shift={(2,-1.5)}]
\draw (-.41,-.6)--(-.41,.41);
\node at (0,0)[scale=.5]{$\begin{ytableau}
3 & 1 \\
\\
2 & 4&4\\
\end{ytableau}$};
\draw [white, line width = 1mm](-.14,-.13)--(-.14,.13);
\end{scope} 
\begin{scope}[shift={(1,-1.5)}]
\draw (-.41,-.6)--(-.41,.41);
\node at (0,0)[scale=.5]{$\begin{ytableau}
3 & 2 \\
\\
1 & 4&4\\
\end{ytableau}$};
\draw [white, line width = 1mm](-.14,-.13)--(-.14,.13);
\end{scope}
\begin{scope}[shift={(-2,-1.5)}]
\draw (-.41,-.6)--(-.41,.41);
\node at (0,0)[scale=.5]{$\begin{ytableau}
3 & 2 \\
\\
2 & 4&3\\
\end{ytableau}$};
\draw [white, line width = 1mm](-.14,-.13)--(-.14,.13);
\end{scope}

\draw[orange,->](0,5.5)--(0,5);
\draw[orange,->](1,4)--(1,3.5);
\draw[orange,->](1,2.5)--(1,2);
\draw[orange,->](3,1)--(3,.5);
\draw[orange,->](0,1)--(0,.5);
\draw[orange,->](-1,1)--(-1,.5);
\draw[orange,->](1,-.5)--(1,-1);
\draw[orange,->](1,-2)--(1,-2.5);
\draw[orange,->](0,-3.5)--(0,-4);
\draw[purple, ->](.25,5.5)--(.75,5);
\draw[purple, ->](.25,4)--(1.75,3.5);
\draw[purple, ->](-1.75,2.5)--(-1.25,2);
\draw[purple, ->](2.25,2.5)--(2.75,2);
\draw[purple, ->](-.75,1)--(.75,.5);
\draw[purple, ->](1.25,1)--(2.75,.5);
\draw[purple, ->](-.75,-.5)--(.75,-1);
\draw[purple, ->](.25,-.5)--(1.75,-1);
\draw[purple, ->](-1.75,-2)--(-.25,-2.5);
\draw[blue, ->](-.25,4)--(-1.75,3.5);
\draw[blue, ->](.75,2.5)--(.25,2);
\draw[blue, ->](1.75,2.5)--(-.75,2);
\draw[blue, ->](2.75,1)--(1.25,.5);
\draw[blue, ->](.75,1)--(.25,.5);
\draw[blue, ->](2.75,-.5)--(2.25,-1);
\draw[blue, ->](-.25,-.5)--(-1.75,-1);
\draw[blue, ->](1.75,-2)--(.25,-2.5);
\draw[blue, ->](.75,-3.5)--(.25,-4);
\begin{scope}[shift={(5,3)}]
\begin{scope}[shift={(0,3)}]
\draw (-.41,-.6)--(-.41,.41);
\node at (0,0)[scale=.5]{$\begin{ytableau}
2 & 3 \\
\\
1 & 1&4\\
\end{ytableau}$};
\draw [white, line width = 1mm](-.14,-.13)--(-.14,.13);
\end{scope}
\begin{scope}[shift={(-1,1.5)}]
\draw (-.41,-.6)--(-.41,.41);
\node at (0,0)[scale=.5]{$\begin{ytableau}
1 & 3 \\
\\
2 & 2&4\\
\end{ytableau}$};
\draw [white, line width = 1mm](-.14,-.13)--(-.14,.13);
\end{scope}
\begin{scope}[yscale=-1]
\draw[blue, <-](-.75,-2)--(-.25,-2.5);
\end{scope}
\end{scope}
\end{tikzpicture}
 \caption{\label{fig:demazure0122}The Demazure crystals for $E_{(0,3,0,2)}(X;q,0)$ on $\SSKD(0,3,0,2)$ with highest weights $(2,2,1,0)$ corresponding to $\key_{(0,1,2,2)}(X)$ (center), $(2,1,1,1)$ corresponding to $\key_{(1,1,1,2)}(X)$ (left), and $(2,1,1,1)$ corresponding to $\key_{(1,2,1,1)}(X)$ (right), with edges $f_1 \color{blue}\swarrow$, $f_2 \color{orange}\downarrow$, $f_3 \color{purple}\searrow$ defined by lowering operators.}
\end{figure}

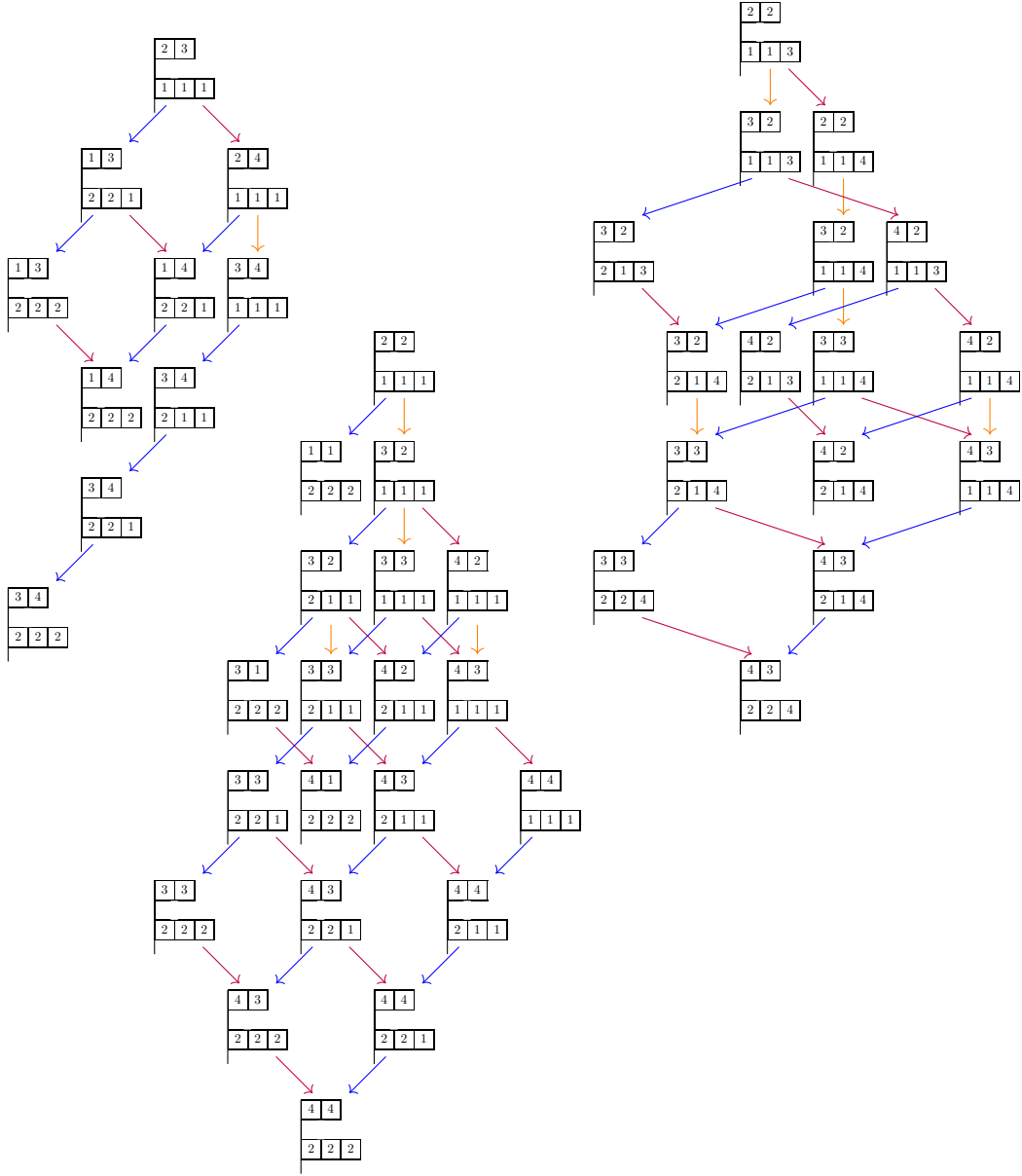
\begin{figure}[ht]
\begin{tikzpicture}

\begin{scope}[shift={(0,6)}]
\draw (-.41,-.6)--(-.41,.41);
\node at (0,0)[scale=.5]{$\begin{ytableau}
2 & 2 \\
\\
1 & 1&1\\
\end{ytableau}$};
\draw [white, line width = 1mm](-.14,-.13)--(-.14,.13);
\end{scope}
\begin{scope}[shift={(0,4.5)}]
\draw (-.41,-.6)--(-.41,.41);
\node at (0,0)[scale=.5]{$\begin{ytableau}
3 & 2 \\
\\
1 & 1&1\\
\end{ytableau}$};
\draw [white, line width = 1mm](-.14,-.13)--(-.14,.13);
\end{scope}
\begin{scope}[shift={(-1,4.5)}]
\draw (-.41,-.6)--(-.41,.41);
\node at (0,0)[scale=.5]{$\begin{ytableau}
1 & 1 \\
\\
2 & 2&2\\
\end{ytableau}$};
\draw [white, line width = 1mm](-.14,-.13)--(-.14,.13);
\end{scope}
%

\begin{scope}[shift={(-1,3)}]
\draw (-.41,-.6)--(-.41,.41);
\node at (0,0)[scale=.5]{$\begin{ytableau}
3 & 2 \\
\\
2& 1&1\\
\end{ytableau}$};
\draw [white, line width = 1mm](-.14,-.13)--(-.14,.13);
\end{scope}

\begin{scope}[shift={(0,3)}]
\draw (-.41,-.6)--(-.41,.41);
\node at (0,0)[scale=.5]{$\begin{ytableau}
3 & 3 \\
\\
1 & 1&1\\
\end{ytableau}$};
\draw [white, line width = 1mm](-.14,-.13)--(-.14,.13);
\end{scope}
\begin{scope}[shift={(1,3)}]
\draw (-.41,-.6)--(-.41,.41);
\node at (0,0)[scale=.5]{$\begin{ytableau}
4 & 2 \\
\\
1 & 1&1\\
\end{ytableau}$};
\draw [white, line width = 1mm](-.14,-.13)--(-.14,.13);
\end{scope}
%

\begin{scope}[shift={(-2,1.5)}]
\draw (-.41,-.6)--(-.41,.41);
\node at (0,0)[scale=.5]{$\begin{ytableau}
3 & 1 \\
\\
2 & 2&2\\
\end{ytableau}$};
\draw [white, line width = 1mm](-.14,-.13)--(-.14,.13);
\end{scope}
\begin{scope}[shift={(-1,1.5)}]
\draw (-.41,-.6)--(-.41,.41);
\node at (0,0)[scale=.5]{$\begin{ytableau}
3 & 3 \\
\\
2 & 1&1\\
\end{ytableau}$};
\draw [white, line width = 1mm](-.14,-.13)--(-.14,.13);
\end{scope}
\begin{scope}[shift={(1,1.5)}]
\draw (-.41,-.6)--(-.41,.41);
\node at (0,0)[scale=.5]{$\begin{ytableau}
4 & 3 \\
\\
1 & 1&1\\
\end{ytableau}$};
\draw [white, line width = 1mm](-.14,-.13)--(-.14,.13);
\end{scope}
\begin{scope}[shift={(0,1.5)}]
\draw (-.41,-.6)--(-.41,.41);
\node at (0,0)[scale=.5]{$\begin{ytableau}
4 & 2 \\
\\
2& 1&1\\
\end{ytableau}$};
\draw [white, line width = 1mm](-.14,-.13)--(-.14,.13);
\end{scope}
\begin{scope}[shift={(2,0)}]
\draw (-.41,-.6)--(-.41,.41);
\node at (0,0)[scale=.5]{$\begin{ytableau}
4 & 4 \\
\\
1 & 1&1\\
\end{ytableau}$};
\draw [white, line width = 1mm](-.14,-.13)--(-.14,.13);
\end{scope}
\draw (-.41,-.6)--(-.41,.41);
\node at (0,0)[scale=.5]{$\begin{ytableau}
4 & 3 \\
\\
2 & 1&1\\
\end{ytableau}$};
\draw [white, line width = 1mm](-.14,-.13)--(-.14,.13);
\begin{scope}[shift={(-1,0)}]
\draw (-.41,-.6)--(-.41,.41);
\node at (0,0)[scale=.5]{$\begin{ytableau}
4 & 1 \\
\\
2 & 2&2\\
\end{ytableau}$};
\draw [white, line width = 1mm](-.14,-.13)--(-.14,.13);
\end{scope}
\begin{scope}[shift={(-2,0)}]
\draw (-.41,-.6)--(-.41,.41);
\node at (0,0)[scale=.5]{$\begin{ytableau}
3 & 3 \\
\\
2 & 2&1\\
\end{ytableau}$};
\draw [white, line width = 1mm](-.14,-.13)--(-.14,.13);
\end{scope}
\begin{scope}[shift={(1,-1.5)}]
\draw (-.41,-.6)--(-.41,.41);
\node at (0,0)[scale=.5]{$\begin{ytableau}
4 & 4 \\
\\
2 & 1&1\\
\end{ytableau}$};
\draw [white, line width = 1mm](-.14,-.13)--(-.14,.13);
\end{scope}
\begin{scope}[shift={(-1,-1.5)}]
\draw (-.41,-.6)--(-.41,.41);
\node at (0,0)[scale=.5]{$\begin{ytableau}
4 & 3 \\
\\
2 & 2&1\\
\end{ytableau}$};
\draw [white, line width = 1mm](-.14,-.13)--(-.14,.13);
\end{scope}
\begin{scope}[shift={(-3,-1.5)}]
\draw (-.41,-.6)--(-.41,.41);
\node at (0,0)[scale=.5]{$\begin{ytableau}
3 & 3 \\
\\
2 & 2&2\\
\end{ytableau}$};
\draw [white, line width = 1mm](-.14,-.13)--(-.14,.13);
\end{scope}
\begin{scope}[shift={(0,-3)}]
\draw (-.41,-.6)--(-.41,.41);
\node at (0,0)[scale=.5]{$\begin{ytableau}
4 & 4 \\
\\
2 & 2&1\\
\end{ytableau}$};
\draw [white, line width = 1mm](-.14,-.13)--(-.14,.13);
\end{scope}
\begin{scope}[shift={(-2,-3)}]
\draw (-.41,-.6)--(-.41,.41);
\node at (0,0)[scale=.5]{$\begin{ytableau}
4 & 3 \\
\\
2 & 2&2\\
\end{ytableau}$};
\draw [white, line width = 1mm](-.14,-.13)--(-.14,.13);
\end{scope}
\begin{scope}[shift={(-1,-4.5)}]
\draw (-.41,-.6)--(-.41,.41);
\node at (0,0)[scale=.5]{$\begin{ytableau}
4 & 4 \\
\\
 2& 2&2\\
\end{ytableau}$};
\draw [white, line width = 1mm](-.14,-.13)--(-.14,.13);
\end{scope}
\draw[orange,->](0,5.5)--(0,5);
\draw[orange,->](0,4)--(0,3.5);
\draw[orange,->](1,2.4)--(1,2);
\draw[orange,->](-1,2.4)--(-1,2);
\draw[purple, ->](.25,4)--(.75,3.5);
\draw[purple, ->](0.25,2.5)--(.75,2);
\draw[purple, ->](-.75,2.5)--(-.25,2);
\draw[purple, ->](1.25,1)--(1.75,.5);
\draw[purple, ->](-0.75,1)--(-.25,.5);
\draw[purple, ->](-1.75,1)--(-1.25,.5);
\draw[purple, ->](-1.75,-.5)--(-1.25,-1);
\draw[purple, ->](0.25,-.5)--(0.75,-1);
\draw[purple, ->](-2.75,-2)--(-2.25,-2.5);
\draw[purple, ->](-.75,-2)--(-.25,-2.5);
\draw[purple, ->](-1.75,-3.5)--(-1.25,-4);
\draw[blue, ->](-.25,5.5)--(-.75,5);
\draw[blue, ->](-.25,4)--(-.75,3.5);
\draw[blue, ->](-.25,2.5)--(-.75,2);
\draw[blue, ->](-1.25,2.5)--(-1.75,2);
\draw[blue, ->](.75,2.5)--(.25,2);
\draw[blue, ->](.75,1)--(.25,.5);
\draw[blue, ->](-.25,1)--(-.75,.5);
\draw[blue, ->](1.75,-.5)--(1.25,.-1);
\begin{scope}[shift={(-1,-1.5)},yscale=-1]
\draw[blue, <-](0.25,2.5)--(.75,2);
\draw[blue, <-](1.25,1)--(1.75,.5);
\draw[blue, <-](-0.75,1)--(-.25,.5);
\draw[blue, <-](-1.75,-.5)--(-1.25,-1);
\draw[blue, <-](0.25,-.5)--(0.75,-1);
\draw[blue, <-](-.75,-2)--(-.25,-2.5);
\end{scope}
\begin{scope}[shift={(-3,7)}]
\begin{scope}[shift={(0,3)}]
\draw (-.41,-.6)--(-.41,.41);
\node at (0,0)[scale=.5]{$\begin{ytableau}
2 & 3 \\
\\
1 & 1&1\\
\end{ytableau}$};
\draw [white, line width = 1mm](-.14,-.13)--(-.14,.13);
\end{scope}
\begin{scope}[shift={(1,1.5)}]
\draw (-.41,-.6)--(-.41,.41);
\node at (0,0)[scale=.5]{$\begin{ytableau}
2 & 4 \\
\\
1 & 1&1\\
\end{ytableau}$};
\draw [white, line width = 1mm](-.14,-.13)--(-.14,.13);
\end{scope}
\begin{scope}[shift={(-1,1.5)}]
\draw (-.41,-.6)--(-.41,.41);
\node at (0,0)[scale=.5]{$\begin{ytableau}
1 & 3 \\
\\
2 & 2&1\\
\end{ytableau}$};
\draw [white, line width = 1mm](-.14,-.13)--(-.14,.13);
\end{scope}
\begin{scope}[shift={(1,0)}]
\draw (-.41,-.6)--(-.41,.41);
\node at (0,0)[scale=.5]{$\begin{ytableau}
3 & 4 \\
\\
1 & 1&1\\
\end{ytableau}$};
\draw [white, line width = 1mm](-.14,-.13)--(-.14,.13);
\end{scope}
\begin{scope}[shift={(0,0)}]
\draw (-.41,-.6)--(-.41,.41);
\node at (0,0)[scale=.5]{$\begin{ytableau}
1 & 4 \\
\\
2 & 2&1\\
\end{ytableau}$};
\draw [white, line width = 1mm](-.14,-.13)--(-.14,.13);
\end{scope}
\begin{scope}[shift={(-2,0)}]
\draw (-.41,-.6)--(-.41,.41);
\node at (0,0)[scale=.5]{$\begin{ytableau}
1 & 3 \\
\\
2 & 2&2\\
\end{ytableau}$};
\draw [white, line width = 1mm](-.14,-.13)--(-.14,.13);
\end{scope}
\begin{scope}[shift={(0,-1.5)}]
\draw (-.41,-.6)--(-.41,.41);
\node at (0,0)[scale=.5]{$\begin{ytableau}
3 & 4 \\
\\
2 & 1&1\\
\end{ytableau}$};
\draw [white, line width = 1mm](-.14,-.13)--(-.14,.13);
\end{scope}
\begin{scope}[shift={(-1,-1.5)}]
\draw (-.41,-.6)--(-.41,.41);
\node at (0,0)[scale=.5]{$\begin{ytableau}
1 & 4 \\
\\
2 & 2&2\\
\end{ytableau}$};
\draw [white, line width = 1mm](-.14,-.13)--(-.14,.13);
\end{scope}
%
\begin{scope}[shift={(-1,-3)}]
\draw (-.41,-.6)--(-.41,.41);
\node at (0,0)[scale=.5]{$\begin{ytableau}
3 & 4 \\
\\
2 & 2&1\\
\end{ytableau}$};
\draw [white, line width = 1mm](-.14,-.13)--(-.14,.13);
\end{scope}
\begin{scope}[shift={(-2,-4.5)}]
\draw (-.41,-.6)--(-.41,.41);
\node at (0,0)[scale=.5]{$\begin{ytableau}
3 & 4 \\
\\
2 & 2&2\\
\end{ytableau}$};
\draw [white, line width = 1mm](-.14,-.13)--(-.14,.13);
\end{scope}

\draw[orange,->](1,1)--(1,.5);
\draw[purple, ->](.25,2.5)--(.75,2);
\draw[purple, ->](-.75,1)--(-.25,.5);
\draw[purple, ->](-1.75,-.5)--(-1.25,-1);
\begin{scope}[yscale=-1]
\draw[blue, <-](-.75,-2)--(-.25,-2.5);
\draw[blue, <-](-.75,1)--(-.25,.5);
\draw[blue, <-](.25,1)--(.75,.5);
\draw[blue, <-](-1.75,-.5)--(-1.25,-1);
\draw[blue, <-](.25,-.5)--(.75,-1);
\end{scope}
\begin{scope}[yscale=-1, shift={(-1,1.5)}]
\draw[blue, <-](-.75,2.5)--(-.25,2);
\draw[blue, <-](.25,1)--(.75,.5);
\end{scope}
\end{scope}
\begin{scope}[shift={(6,6)}]
\begin{scope}[shift={(0,4.5)}]
\begin{scope}[shift={(-1,0)}]
\draw (-.41,-.6)--(-.41,.41);
\node at (0,0)[scale=.5]{$\begin{ytableau}
2 & 2 \\
\\
1 &1&3\\
\end{ytableau}$};
\draw [white, line width = 1mm](-.14,-.13)--(-.14,.13);
\end{scope}
\end{scope}
\begin{scope}[shift={(0,3)}]
\begin{scope}[shift={(0,0)}]
\draw (-.41,-.6)--(-.41,.41);
\node at (0,0)[scale=.5]{$\begin{ytableau}
2 & 2 \\
\\
1 &1&4\\
\end{ytableau}$};
\draw [white, line width = 1mm](-.14,-.13)--(-.14,.13);
\end{scope}
\begin{scope}[shift={(-1,0)}]
\draw (-.41,-.6)--(-.41,.41);
\node at (0,0)[scale=.5]{$\begin{ytableau}
3 & 2 \\
\\
1 &1&3\\
\end{ytableau}$};
\draw [white, line width = 1mm](-.14,-.13)--(-.14,.13);
\end{scope}
\end{scope}
\begin{scope}[shift={(0,1.5)}]
\begin{scope}[shift={(-3,0)}]
\draw (-.41,-.6)--(-.41,.41);
\node at (0,0)[scale=.5]{$\begin{ytableau}
3 &2 \\
\\
2 &1&3\\
\end{ytableau}$};
\draw [white, line width = 1mm](-.14,-.13)--(-.14,.13);
\end{scope}
\begin{scope}[shift={(0,0)}]
\draw (-.41,-.6)--(-.41,.41);
\node at (0,0)[scale=.5]{$\begin{ytableau}
3 & 2 \\
\\
1 &1&4\\
\end{ytableau}$};
\draw [white, line width = 1mm](-.14,-.13)--(-.14,.13);
\end{scope}
\begin{scope}[shift={(1,0)}]
\draw (-.41,-.6)--(-.41,.41);
\node at (0,0)[scale=.5]{$\begin{ytableau}
4 & 2 \\
\\
1 &1&3\\
\end{ytableau}$};
\draw [white, line width = 1mm](-.14,-.13)--(-.14,.13);
\end{scope}
\end{scope}
\begin{scope}[shift={(-2,0)}]
\draw (-.41,-.6)--(-.41,.41);
\node at (0,0)[scale=.5]{$\begin{ytableau}
3 & 2 \\
\\
2 &1&4\\
\end{ytableau}$};
\draw [white, line width = 1mm](-.14,-.13)--(-.14,.13);
\end{scope}
\begin{scope}[shift={(-1,0)}]
\draw (-.41,-.6)--(-.41,.41);
\node at (0,0)[scale=.5]{$\begin{ytableau}
4 & 2 \\
\\
2 &1&3\\
\end{ytableau}$};
\draw [white, line width = 1mm](-.14,-.13)--(-.14,.13);
\end{scope}
\begin{scope}[shift={(0,0)}]
\draw (-.41,-.6)--(-.41,.41);
\node at (0,0)[scale=.5]{$\begin{ytableau}
3 & 3 \\
\\
1 &1&4\\
\end{ytableau}$};
\draw [white, line width = 1mm](-.14,-.13)--(-.14,.13);
\end{scope}
\begin{scope}[shift={(2,0)}]
\draw (-.41,-.6)--(-.41,.41);
\node at (0,0)[scale=.5]{$\begin{ytableau}
4 & 2 \\
\\
1 &1&4\\
\end{ytableau}$};
\draw [white, line width = 1mm](-.14,-.13)--(-.14,.13);
\end{scope}
\begin{scope}[shift={(0,-1.5)}]
\begin{scope}[shift={(2,0)}]
\draw (-.41,-.6)--(-.41,.41);
\node at (0,0)[scale=.5]{$\begin{ytableau}
4 & 3 \\
\\
1 &1&4\\
\end{ytableau}$};
\draw [white, line width = 1mm](-.14,-.13)--(-.14,.13);
\end{scope}
\begin{scope}[shift={(0,0)}]
\draw (-.41,-.6)--(-.41,.41);
\node at (0,0)[scale=.5]{$\begin{ytableau}
4 & 2 \\
\\
2 &1&4\\
\end{ytableau}$};
\draw [white, line width = 1mm](-.14,-.13)--(-.14,.13);
\end{scope}
\begin{scope}[shift={(-2,0)}]
\draw (-.41,-.6)--(-.41,.41);
\node at (0,0)[scale=.5]{$\begin{ytableau}
3 & 3 \\
\\
2 &1&4\\
\end{ytableau}$};
\draw [white, line width = 1mm](-.14,-.13)--(-.14,.13);
\end{scope}
\end{scope}
\begin{scope}[shift={(0,-3)}]
\begin{scope}[shift={(0,0)}]
\draw (-.41,-.6)--(-.41,.41);
\node at (0,0)[scale=.5]{$\begin{ytableau}
4 & 3 \\
\\
2 &1&4\\
\end{ytableau}$};
\draw [white, line width = 1mm](-.14,-.13)--(-.14,.13);
\end{scope}
\begin{scope}[shift={(-3,0)}]
\draw (-.41,-.6)--(-.41,.41);
\node at (0,0)[scale=.5]{$\begin{ytableau}
3 & 3 \\
\\
2 &2&4\\
\end{ytableau}$};
\draw [white, line width = 1mm](-.14,-.13)--(-.14,.13);
\end{scope}
\end{scope}
\begin{scope}[shift={(0,-4.5)}]
\begin{scope}[shift={(-1,0)}]
\draw (-.41,-.6)--(-.41,.41);
\node at (0,0)[scale=.5]{$\begin{ytableau}
4 & 3 \\
\\
2 &2&4\\
\end{ytableau}$};
\draw [white, line width = 1mm](-.14,-.13)--(-.14,.13);
\end{scope}
\end{scope}
\draw[orange,->](-1,4)--(-1,3.5);
\draw[orange,->](0,2.5)--(0,2);
\draw[orange,->](-2,-.5)--(-2,-1);
\draw[orange,->](2,-.5)--(2,-1);
\draw[orange,->](0,1)--(0,.5);
\draw[purple, ->](-.75,4)--(-.25,3.5);
\draw[purple, ->](-.75,2.5)--(.75,2);
\draw[purple, ->](1.25,1)--(1.75,.5);
\draw[purple, ->](-2.75,1)--(-2.25,.5);
\draw[purple, ->](.25,-.5)--(1.75,-1);
\draw[purple, ->](-.75,-.5)--(-.25,-1);
\draw[purple, ->](-1.75,-2)--(-.25,-2.5);
\draw[purple, ->](-2.75,-3.5)--(-1.25,-4);
\begin{scope}[yscale=-1, shift={(0,1.5)}]
\draw[blue, <-](-.75,2.5)--(-.25,2);
\draw[blue, <-](.25,1)--(1.75,.5);
\draw[blue, <-](-2.75,1)--(-2.25,.5);
\draw[blue, <-](.25,-.5)--(1.75,-1);
\draw[blue, <-](-1.75,-.5)--(-.25,-1);
\draw[blue, <-](-1.75,-2)--(-.25,-2.5);
\draw[blue, <-](-.75,-2)--(.75,-2.5);
\draw[blue, <-](-2.75,-3.5)--(-1.25,-4);
\end{scope}
\end{scope}
\end{tikzpicture}
\caption{\label{fig:demazure0311}The Demazure crystals for $E_{(0,3,0,2)}(X;q,0)$ on $\SSKD(0,3,0,2)$ with highest weights $(3,1,1,0)$ corresponding to $\key_{(0,3,1,1)}(X)$ (top left), $(3,2,0,0)$ corresponding to $\key_{(0,3,0,2)}(X)$ (bottom center), and $(2,2,1,0)$ corresponding to $\key_{(0,2,1,2)}(X)$ (top right), where the edges $f_1 \color{blue}\swarrow$, $f_2 \color{orange}\downarrow$, $f_3 \color{purple}\searrow$ are defined by lowering operators.}
\end{figure}

\section{Detailed example of the embedding map $\embed$}

In Fig.~\ref{fig:embedExample} we present a complete example of the embedding map $\embed = \T \circ \rectify \circ \D$ acting on the entire Demazure crystal $\B_{4123}(3,1,1,0)$. The colored balls indicate the balls on which the lowering operators act. Lowering operators for $i=1,2,3$ are presented in colors blue $\color{blue}\swarrow$, yellow $\color{orange}\downarrow$, and red $\color{purple}\searrow$, respectively. 

At each step we see how the diagram map $\D$ from semistandard key tabloids to diagrams, the rectification map from diagrams to Kohnert diagrams, and the tableau map $\T$ from Kohnert diagrams to semistandard Young tableau, intertwine the crystal operators $f_i$, $\Df_i$, and $\Yf_i$, respectively. 

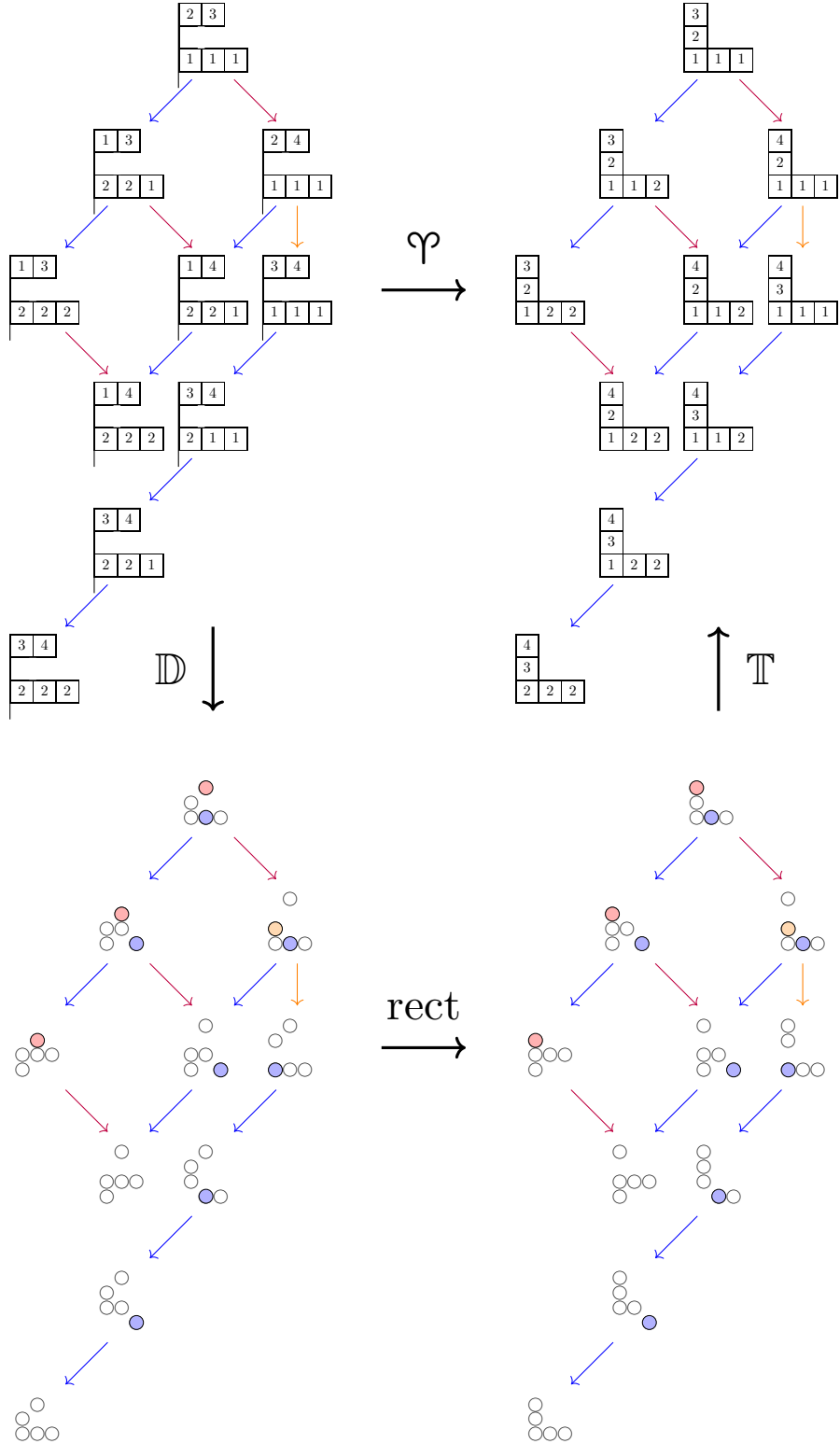
\begin{figure}[ht]
\begin{tikzpicture}[scale=1.2,every node/.style={scale=1.2}]
\begin{scope}[shift={(0,3)}]
\draw (-.41,-.6)--(-.41,.41);
\node at (0,0)[scale=.5]{$\begin{ytableau}
2 & 3 \\
\\
1 & 1&1\\
\end{ytableau}$};
\draw [white, line width = 1mm](-.14,-.13)--(-.14,.13);
\end{scope}
\begin{scope}[shift={(1,1.5)}]
\draw (-.41,-.6)--(-.41,.41);
\node at (0,0)[scale=.5]{$\begin{ytableau}
2 & 4 \\
\\
1 & 1&1\\
\end{ytableau}$};
\draw [white, line width = 1mm](-.14,-.13)--(-.14,.13);
\end{scope}
\begin{scope}[shift={(-1,1.5)}]
\draw (-.41,-.6)--(-.41,.41);
\node at (0,0)[scale=.5]{$\begin{ytableau}
1 & 3 \\
\\
2 & 2&1\\
\end{ytableau}$};
\draw [white, line width = 1mm](-.14,-.13)--(-.14,.13);
\end{scope}
\begin{scope}[shift={(1,0)}]
\draw (-.41,-.6)--(-.41,.41);
\node at (0,0)[scale=.5]{$\begin{ytableau}
3 & 4 \\
\\
1 & 1&1\\
\end{ytableau}$};
\draw [white, line width = 1mm](-.14,-.13)--(-.14,.13);
\end{scope}
\begin{scope}[shift={(0,0)}]
\draw (-.41,-.6)--(-.41,.41);
\node at (0,0)[scale=.5]{$\begin{ytableau}
1 & 4 \\
\\
2 & 2&1\\
\end{ytableau}$};
\draw [white, line width = 1mm](-.14,-.13)--(-.14,.13);
\end{scope}
\begin{scope}[shift={(-2,0)}]
\draw (-.41,-.6)--(-.41,.41);
\node at (0,0)[scale=.5]{$\begin{ytableau}
1 & 3 \\
\\
2 & 2&2\\
\end{ytableau}$};
\draw [white, line width = 1mm](-.14,-.13)--(-.14,.13);
\end{scope}
\begin{scope}[shift={(0,-1.5)}]
\draw (-.41,-.6)--(-.41,.41);
\node at (0,0)[scale=.5]{$\begin{ytableau}
3 & 4 \\
\\
2 & 1&1\\
\end{ytableau}$};
\draw [white, line width = 1mm](-.14,-.13)--(-.14,.13);
\end{scope}
\begin{scope}[shift={(-1,-1.5)}]
\draw (-.41,-.6)--(-.41,.41);
\node at (0,0)[scale=.5]{$\begin{ytableau}
1 & 4 \\
\\
2 & 2&2\\
\end{ytableau}$};
\draw [white, line width = 1mm](-.14,-.13)--(-.14,.13);
\end{scope}
%
\begin{scope}[shift={(-1,-3)}]
\draw (-.41,-.6)--(-.41,.41);
\node at (0,0)[scale=.5]{$\begin{ytableau}
3 & 4 \\
\\
2 & 2&1\\
\end{ytableau}$};
\draw [white, line width = 1mm](-.14,-.13)--(-.14,.13);
\end{scope}
\begin{scope}[shift={(-2,-4.5)}]
\draw (-.41,-.6)--(-.41,.41);
\node at (0,0)[scale=.5]{$\begin{ytableau}
3 & 4 \\
\\
2 & 2&2\\
\end{ytableau}$};
\draw [white, line width = 1mm](-.14,-.13)--(-.14,.13);
\end{scope}

\draw[orange,->](1,1)--(1,.5);
\draw[purple, ->](.25,2.5)--(.75,2);
\draw[purple, ->](-.75,1)--(-.25,.5);
\draw[purple, ->](-1.75,-.5)--(-1.25,-1);
\begin{scope}[yscale=-1]
\draw[blue, <-](-.75,-2)--(-.25,-2.5);
\draw[blue, <-](-.75,1)--(-.25,.5);
\draw[blue, <-](.25,1)--(.75,.5);
\draw[blue, <-](-1.75,-.5)--(-1.25,-1);
\draw[blue, <-](.25,-.5)--(.75,-1);
\end{scope}
\begin{scope}[yscale=-1, shift={(-1,1.5)}]
\draw[blue, <-](-.75,2.5)--(-.25,2);
\draw[blue, <-](.25,1)--(.75,.5);
\end{scope}

\begin{scope}[shift={(0,-9)}]
\draw[very thick,<-](0,4)--(0,5);
\node at (-.5,4.5) [scale=1.5]{$\D$};
\begin{scope}[shift={(0,3)}]
\node at (0,0)[scale=.5]{$\vline \cirtab{  & & & \\  & \rball & & \\  \ & & \\  \ & \bball &  \ \\}$};
\end{scope}
\begin{scope}[shift={(1,1.5)}]
\node at (0,0)[scale=.5]{$\vline \cirtab{  & \ & & \\  \\ \yball & & \\  \ & \bball &  \ \\}$};
\end{scope}
\begin{scope}[shift={(-1,1.5)}]
\node at (0,0)[scale=.5]{$\vline \cirtab{  & & & \\  & \rball & & \\  \ & \ & \\  \ &  &  \bball \\}$};
\end{scope}
\begin{scope}[shift={(1,0)}]
\node at (0,0)[scale=.5]{$\vline \cirtab{  & \ & & \\   \ & & \\ \\  \bball & \ &  \ \\}$};
\end{scope}
\begin{scope}[shift={(0,0)}]
\node at (0,0)[scale=.5]{$\vline \cirtab{  & \ & & \\  \\ \ &\ & \\  \ &  &  \bball \\}$};
\end{scope}
\begin{scope}[shift={(-2,0)}]
\node at (0,0)[scale=.5]{$\vline \cirtab{ \\  & \rball & &   \\ \ & \ & \ \\  \ &  &  \\}$};
\end{scope}
\begin{scope}[shift={(0,-1.5)}]
\node at (0,0)[scale=.5]{$\vline \cirtab{  & \ & & \\   \ & & \\ \ \\   & \bball &  \ \\}$};
\end{scope}
\begin{scope}[shift={(-1,-1.5)}]
\node at (0,0)[scale=.5]{$\vline \cirtab{  & \ & & \\  \\ \ &\ & \ \\  \ &  &  \\}$};
\end{scope}
%
\begin{scope}[shift={(-1,-3)}]
\node at (0,0)[scale=.5]{$\vline \cirtab{  & \ & & \\   \ & & \\ \ &\ & \\   &  &  \bball \\}$};
\end{scope}
\begin{scope}[shift={(-2,-4.5)}]
\node at (0,0)[scale=.5]{$\vline \cirtab{  & \ & & \\   \ & & \\ \ &\ & \ \\   &  &  \\}$};
\end{scope}

\draw[orange,->](1,1)--(1,.5);
\draw[purple, ->](.25,2.5)--(.75,2);
\draw[purple, ->](-.75,1)--(-.25,.5);
\draw[purple, ->](-1.75,-.5)--(-1.25,-1);
\begin{scope}[yscale=-1]
\draw[blue, <-](-.75,-2)--(-.25,-2.5);
\draw[blue, <-](-.75,1)--(-.25,.5);
\draw[blue, <-](.25,1)--(.75,.5);
\draw[blue, <-](-1.75,-.5)--(-1.25,-1);
\draw[blue, <-](.25,-.5)--(.75,-1);
\end{scope}
\begin{scope}[yscale=-1, shift={(-1,1.5)}]
\draw[blue, <-](-.75,2.5)--(-.25,2);
\draw[blue, <-](.25,1)--(.75,.5);
\end{scope}

\begin{scope}[shift={(6,0)}]
\draw[very thick,<-](-3,0)--(-4,0);
\node at (-3.5,0.5) [scale=1.5]{$\rectify$};
\begin{scope}[shift={(0,3)}]
\node at (0,0)[scale=.5]{$\vline \cirtab{  & & & \\  \rball & & \\  \ & & \\  \ & \bball &  \ \\}$};
\end{scope}
\begin{scope}[shift={(1,1.5)}]
\node at (0,0)[scale=.5]{$\vline \cirtab{   \ & & \\  \\ \yball & & \\  \ & \bball &  \ \\}$};
\end{scope}
\begin{scope}[shift={(-1,1.5)}]
\node at (0,0)[scale=.5]{$\vline \cirtab{  & & & \\   \rball & & \\  \ & \ & \\  \ &  &  \bball \\}$};
\end{scope}
\begin{scope}[shift={(1,0)}]
\node at (0,0)[scale=.5]{$\vline \cirtab{   \ & & \\   \ & & \\ \\  \bball & \ &  \ \\}$};
\end{scope}
\begin{scope}[shift={(0,0)}]
\node at (0,0)[scale=.5]{$\vline \cirtab{   \ & & \\  \\ \ &\ & \\  \ &  &  \bball \\}$};
\end{scope}
\begin{scope}[shift={(-2,0)}]
\node at (0,0)[scale=.5]{$\vline \cirtab{ \\   \rball & &   \\ \ & \ & \ \\  \ &  &  \\}$};
\end{scope}
\begin{scope}[shift={(0,-1.5)}]
\node at (0,0)[scale=.5]{$\vline \cirtab{  \ & & \\   \ & & \\ \ \\   & \bball &  \ \\}$};
\end{scope}
\begin{scope}[shift={(-1,-1.5)}]
\node at (0,0)[scale=.5]{$\vline \cirtab{   \ & & \\  \\ \ &\ & \ \\  \ &  &  \\}$};
\end{scope}
%
\begin{scope}[shift={(-1,-3)}]
\node at (0,0)[scale=.5]{$\vline \cirtab{   \ & & \\   \ & & \\ \ &\ & \\   &  &  \bball \\}$};
\end{scope}
\begin{scope}[shift={(-2,-4.5)}]
\node at (0,0)[scale=.5]{$\vline \cirtab{   \ & & \\   \ & & \\ \ &\ & \ \\   &  &  \\}$};
\end{scope}

\draw[orange,->](1,1)--(1,.5);
\draw[purple, ->](.25,2.5)--(.75,2);
\draw[purple, ->](-.75,1)--(-.25,.5);
\draw[purple, ->](-1.75,-.5)--(-1.25,-1);
\begin{scope}[yscale=-1]
\draw[blue, <-](-.75,-2)--(-.25,-2.5);
\draw[blue, <-](-.75,1)--(-.25,.5);
\draw[blue, <-](.25,1)--(.75,.5);
\draw[blue, <-](-1.75,-.5)--(-1.25,-1);
\draw[blue, <-](.25,-.5)--(.75,-1);
\end{scope}
\begin{scope}[yscale=-1, shift={(-1,1.5)}]
\draw[blue, <-](-.75,2.5)--(-.25,2);
\draw[blue, <-](.25,1)--(.75,.5);
\end{scope}
\end{scope}
\end{scope}

\begin{scope}[shift={(6,0)}]
\draw[very thick,<-](-3,0)--(-4,0);
\node at (-3.5,0.5) [scale=1.5]{$\embed$};
\draw[very thick,->](0,-5)--(0,-4);
\node at (.5,-4.5) [scale=1.5]{$\T$};
\begin{scope}[shift={(0,3)}]
\node at (0,0)[scale=.5]{$\begin{ytableau}
3\\
2 \\
1 & 1&1\\
\end{ytableau}$};
\end{scope}
\begin{scope}[shift={(1,1.5)}]
\node at (0,0)[scale=.5]{$\begin{ytableau}
4\\
2 \\
1 & 1&1\\
\end{ytableau}$};\end{scope}
\begin{scope}[shift={(-1,1.5)}]
\node at (0,0)[scale=.5]{$\begin{ytableau}
3\\
2 \\
1 & 1&2\\
\end{ytableau}$};
\end{scope}
\begin{scope}[shift={(1,0)}]
\node at (0,0)[scale=.5]{$\begin{ytableau}
4\\
3 \\
1 & 1&1\\
\end{ytableau}$};
\end{scope}
\begin{scope}[shift={(0,0)}]
\node at (0,0)[scale=.5]{$\begin{ytableau}
4\\
2 \\
1 & 1&2\\
\end{ytableau}$};\end{scope}
\begin{scope}[shift={(-2,0)}]
\node at (0,0)[scale=.5]{$\begin{ytableau}
3\\
2 \\
1 & 2&2\\
\end{ytableau}$};
\end{scope}
\begin{scope}[shift={(0,-1.5)}]
\node at (0,0)[scale=.5]{$\begin{ytableau}
4\\
3 \\
1 & 1&2\\
\end{ytableau}$};
\end{scope}
\begin{scope}[shift={(-1,-1.5)}]
\node at (0,0)[scale=.5]{$\begin{ytableau}
4\\
2 \\
1 & 2&2\\
\end{ytableau}$};\end{scope}
%
\begin{scope}[shift={(-1,-3)}]
\node at (0,0)[scale=.5]{$\begin{ytableau}
4\\
3 \\
1 & 2&2\\
\end{ytableau}$};
\end{scope}
\begin{scope}[shift={(-2,-4.5)}]
\node at (0,0)[scale=.5]{$\begin{ytableau}
4\\
3 \\
2 & 2&2\\
\end{ytableau}$};
\end{scope}

\draw[orange,->](1,1)--(1,.5);
\draw[purple, ->](.25,2.5)--(.75,2);
\draw[purple, ->](-.75,1)--(-.25,.5);
\draw[purple, ->](-1.75,-.5)--(-1.25,-1);
\begin{scope}[yscale=-1]
\draw[blue, <-](-.75,-2)--(-.25,-2.5);
\draw[blue, <-](-.75,1)--(-.25,.5);
\draw[blue, <-](.25,1)--(.75,.5);
\draw[blue, <-](-1.75,-.5)--(-1.25,-1);
\draw[blue, <-](.25,-.5)--(.75,-1);
\end{scope}
\begin{scope}[yscale=-1, shift={(-1,1.5)}]
\draw[blue, <-](-.75,2.5)--(-.25,2);
\draw[blue, <-](.25,1)--(.75,.5);
\end{scope}
\end{scope}
\end{tikzpicture}
\caption{Detailed example of the embedding map $\embed: \SSKD(a) \to \SSYT(\lambda)$ on the Demazure crystal $\B_{4123}(3,1,1,0)$. The colored balls indicate the ball on which the lowering operator is acting.}\label{fig:embedExample}
\end{figure}

\end{document}